\documentclass[reqno,a4paper,11pt]{amsart}
\usepackage{amsmath,amsthm,amssymb,mathrsfs}
\numberwithin{equation}{section}
\usepackage[
left = 2.25cm,
right = 2.25cm,
top = 2.2 cm,
bottom = 2.2 cm
]{geometry}
\usepackage{graphicx,xcolor,tikz,mathtools}
\usepackage{bm,stmaryrd}
\usepackage{amsaddr}
\usepackage{enumitem}
\usepackage[toc,page]{appendix}
\usepackage[
pdfencoding=auto,
colorlinks = true,
citecolor = blue,
linkcolor = blue,
anchorcolor = blue,
urlcolor = blue, 
filecolor=blue,
pdfkeywords={}
]{hyperref}
\usepackage{esint} 
\usepackage{cite}
\usepackage[numbers]{natbib}     
\usepackage{url}    
\usepackage{todonotes}
\makeatletter
\define@key{todonotes}{Helmut}[]{%
	\setkeys{todonotes}{author=\textbf{Helmut},inline,color=red!10}}%
\define@key{todonotes}{Andrea}[]{%
	\setkeys{todonotes}{author=\textbf{Andrea},inline,color=green!20}}%
\define@key{todonotes}{Yadong}[]{%
	\setkeys{todonotes}{author=\textbf{Yadong},inline,color=cyan!20}}%
\makeatother

\usepackage{amsthm}

\theoremstyle{plain}
\newtheorem{theorem}{Theorem}[section]
\newtheorem*{theorem*}{Theorem}

\newtheorem{definition}[theorem]{Definition}
\newtheorem{lemma}[theorem]{Lemma}

\newtheorem{remark}[theorem]{Remark}
\theoremstyle{definition}

\def\supp{\mathrm{supp }\ }
\numberwithin{equation}{section}

\def\sf{\mathbb S^{d-1}}

\def\Ah{\mathcal A_{\alpha,h}}
\def\mA{\mathcal A}

\newcommand{\bu}{\mathbf{u}}
\newcommand{\bv}{\mathbf{v}}

\newcommand{\bn}{\mathbf{n}}

\newcommand{\bA}{\mathbf{A}}
\newcommand{\bH}{\mathbf{H}}

\newcommand{\bP}{\mathbf{P}}

\newcommand{\cA}{\mathcal{A}}

\newcommand{\cE}{\mathcal{E}}

\newcommand{\cH}{\mathcal{H}}
\newcommand{\cI}{\mathcal{I}}

\newcommand{\cZ}{\mathcal{Z}}

\def\p{\partial}

\renewcommand{\d}{\mathrm{d}}
\newcommand{\dx}{\,\d x}
\newcommand{\dt}{\,\d t}
\newcommand{\dxdt}{\,\d x \d t}

\newcommand{\ddt}{\frac{\d}{\d t}}

\newcommand{\ptial}[1]{ \partial_{#1} }
\newcommand{\pt}{\ptial{t}}
\newcommand{\onehalf}{\frac{1}{2}}

\def\H{\mathcal{H}}

\newcommand{\jump}[1]{\left\llbracket #1 \right\rrbracket}
\newcommand{\abs}[1]{\vert #1 \vert}

\newcommand{\norm}[1]{\left\Vert #1 \right \Vert}
\def\II{\mathcal{I}}

\newcommand{\normmm}[1]{\left\vert #1 \right\vert}

\def\mmu{\normmm{\mu_t}_{\sf}}

\def\Div{\mathrm{div}\,}
\DeclareMathOperator*{\esssup}{\mathrm{esssup}}
\newcommand{\dist}{\mathrm{dist}}

\DeclarePairedDelimiter\ceil{\lceil}{\rceil}
\DeclarePairedDelimiter\floor{\lfloor}{\rfloor}
\def\xhi{\chi}
\def\tTh{_t^{T_h(t)}}
\def\Th#1{{T_h(#1)}}

\def\Tstrong{T_{\mathrm{strong}}}

\def\Hmz{H^{1}({\mathbb T^d})'}
\def\Hz{H^{1}({\mathbb T^d})}
\def\normh#1{\norm{#1}_{L^2({\mathbb T^d})}}

\def\Hus{\dot{\mathbf H}^1_{\sigma}({\mathbb T^d})}
\def\Rd{\R^d}
\def\S{{\mathbb S^{d-1}}}
\def\Hs{\dot{\mathbf H}^1_{\sigma}({\mathbb T^d})}

\def\bzeta{\boldsymbol \zeta}
\def\bL{\mathbf L}

\def\Hd{{\mathcal H^{d-1}}}

\def\Ts{{T^*}}
\def\tF{\widetilde{F}}
\def\wchi{\widehat\chi}

\def\tmmu{\normmm{\mu_t}_{\mathbb{S}^{d-1}}}
\def\om{{\mathbb T^d}}

\def\ov{\overline}
\def\wtilde{\widetilde}
\def\R{\mathbb R}
\def\Ls{\bL^2_\sigma({\mathbb T^d})}

\def\N{\mathbb N}
\def\R{\mathbb R}

\def\Er{E[\mu,\chi_A,\bv|\xhi_{\mA},\bu](t)}
\def\Erel#1{E[\mu,\chi_A,\bv|\xhi_{\mA},\bu]({#1})}
\def\Eb{E_{\mathrm{bulk}}[\chi_A|\xhi_{\mA}](t)}
\def\Ebulk#1{E_{\mathrm{bulk}}[\chi_A|\xhi_{\mA}]({#1})}

\def\cA{\mathcal A}
\def\pcA{\mathcal I}

\def\pA{\partial^*A}

\def\tw{w}

\def\Ld#1{\norm{#1}_{(C({\mathbb T^d})\cap H^1({\mathbb T^d}))'}}
\def\dual{(C({\mathbb T^d})\cap H^1({\mathbb T^d}))'}

\allowdisplaybreaks[4]

\begin{document}
	
	\title[Weak varifold solutions to Navier-Stokes/Mean Curvature Flow system]{Global Weak varifold solutions\\ to Navier-Stokes/Mean Curvature Flow system}
    

    
	\author[A. Poiatti]{
		\small
		Andrea Poiatti$^\ddagger$
	}

	\address{
		  $^\ddagger$Dipartimento di Scienze Matematiche, Fisiche e Informatiche, \\ Università degli Studi di Parma, 43124 Parma, Italy
	}
	\email{andrea.poiatti@unipr.it}


	
\begin{abstract}
The Navier-Stokes/Mean Curvature Flow system describes a two-phase flow of incompressible, viscous and immiscible
fluids separated by a sharp interface, whose evolution is governed by a
convective mean curvature flow equation, coupled to a two-phase Navier-Stokes equation accounting for surface tension. Among others, this model describes the dynamics of dry foams and it also arises as a sharp interface limit of Navier-Stokes/Allen-Cahn systems with non-vanishing mobility. Despite this wide range of applications, global-in-time existence of any kind of weak solutions remains a challenging open problem. In this work, we introduce a novel notion of varifold weak solution in two and three ambient dimensions, based on a sharp energy dissipation 
principle à la De Giorgi, and we show for the first time the unconditional global-in-time existence of such a weak solution. We also show, by means of a relative entropy approach, that any classical solution to the Navier-Stokes/Mean Curvature Flow system is unique in the class of our new weak varifold solutions. 
 \end{abstract}
	
	\maketitle
	
	\section{Introduction}

    Liquid foams are ubiquitous in daily life: typical examples include the soapy froth in a washing basin and the head on a pint of beer. A foam can be described by a gas-liquid
mixture in which the volume of liquid is pretty much smaller than the volume of the gas. Usually, in a ‘‘wet” foam (for which the
volume fraction of liquid is around 10-20\%), the bubbles are mostly spherical, whereas in a ‘‘dry” foam (volume fraction of liquid is less than 10\%), bubbles are closer to a polyhedral shape, and they are the most relevant for our modeling purposes. We are interested in a specific phenomenon, called diffusive coarsening, which consists of the evolution of bubble size and topological structure occurring as an effect of gas
exchange between bubbles \cite{base}. This exchange of gas is the effect of diffusion through the thin liquid films separating one bubble
from another. This diffusive flux is proportional to the pressure difference between two adjacent bubbles. During diffusive coarsening phenomenon, bubbles with a large number of faces
grow in size at the expense of bubbles with a small number of faces which consequently shrink. As a dry foam evolves, bubbles can change their  contact surfaces with the surrounding bubbles, leading to intricate topological changes, especially in
three dimensions \cite{Sethian1}.

One can model the dynamics of a dry foam by considering the fluid mechanics of the gas phase coupled with membrane
surface tension and permeability. For simplicity, we consider the case where the membranes can be idealized as without mass, infinitely thin, and with only two chambers. As common in the literature (cf., e.g.,  \cite{AbelsMoser}), we refer to the model as Navier-Stokes/Mean Curvature Flow system. We thus study the flow of two incompressible, viscous and immiscible fluids on the flat torus $\om$, $d = 2, 3$, i.e., assuming periodic boundary conditions. We assume the fluids to fill domains $\Omega^+(t)$ and $\Omega^-(t)$, $t>0$, respectively, with a shared interface $\Gamma(t)$ between them. The flow is described by means of the velocity $\bv : \om \times(0,\infty)\to \R^d$, and the pressure $\pi : \om \times (0,\infty)\to\R$ in both fluids (in Eulerian coordinates). We assume the fluids to be Newtonian, so that the stress tensors are of the form $\nu_{i}D\bv - \pi Id$, $i=1,2$, with constant viscosities $\nu_1,\nu_2>0$ corresponding to the viscosities of the fluids filling $\Omega^+(t)$ and $\Omega^-(t)$, respectively. Here $D\bv=\tfrac{\nabla\bv+(\nabla\bv)^T}{2}$ is the symmetric strain rate tensor. We are interested in the case in which the interface surface tension is present, so that the interface provides a force that induces a pressure jump proportional to its mean curvature. Without permeability, the interface is just advected by the velocity $\bv$ of the fluid, assumed to
be continuous across the interface. When there is permeability, which is the case we are considering in this contribution, the rate of diffusion of fluid is proportional to the pressure difference \cite{base} and, following \cite{basic}, we model the permeability as a slip of the interface in the normal direction
relative to the fluid velocity $\bv$. The velocity of the interface is represented by the normal component of $\bv$ minus a term in the normal direction, which is proportional to the interface mean curvature, accounting for permeability. This justifies the model denomination as Navier-Stokes/Mean Curvature Flow system, as the model consists in the coupling of Navier-Stokes equations with an advective Mean Curvature Flow for the fluidic interface.

Consider a finite time horizon $\Ts\in(0,\infty)$ and let $\mathcal A =
	(\mathcal A (t))_t\in[0,\Ts)$ be a time-dependent family of smoothly evolving open subsets $\mathcal A (t) \subset \om$ with $\partial\mathcal A(t) = \partial^*\mathcal A(t)$, $t \in [0, \Ts)$,
	where $\partial^*\mathcal A$ is the reduced boundary of $\cA$. Set $\cA_0:=\cA(0)$ and denote by $\xhi_0$ the characteristic function corresponding to the initial smooth set $\cA_0$. In the fluid setting above introduced, we have $\Omega^+(t)=\cA(t)$, $\Omega^-(t)=\om\setminus \overline{\cA(t)}$, and $\pcA(t)={\partial \cA(t)}$, so that $\om\setminus \pcA(t)=\Omega^+(t)\cup \Omega^-(t)$. Now we define, for every $t \in [0, \Ts)$,
 $V_{\pcA} (t)$ and $\bH_{\pcA (t)}=H_{\pcA(t)}\bn_{\pcA(t)}$ as the associated scalar normal velocity and mean curvature vector, with respect to $\bn_{\pcA(t)}$, respectively. We also denote by $H_{\pcA(t)}$ the scalar mean curvature. Here $\bn_{\pcA(t)}$ is the
 unit normal vector field along $\pcA(t)$ pointing inside $\Omega^+(t)=\cA (t)$.  Then denote by $\xhi(t)$ the characteristic function associated to $\cA(t)$, that is associated to the domain $\Omega^+(t)$ filled by fluid 1. We say that the family $\cA$ and the velocity $\bv$  evolve by the Navier-Stokes/Mean Curvature Flow system if for each
	$t \in (0, \Ts)$ it holds
	\begin{align}
		\label{A0}
V_{\pcA(t)}&=\bv\cdot\bn_{\pcA(t)}+c_0MH_{\II(t)}\quad\quad\quad&&\text{ on }\pcA(t),
	\end{align}
	and the velocity $\bv$ satisfies
	\begin{align}
\pt\bv+\Div(\bv\otimes \bv)-\Div(\nu(\chi)D\bv)+\nabla \pi&=\mathbf 0,&& \text{ in }\om\setminus \pcA(t),\label{A6}\\\label{A7}
		\Div \bv&=0,&&\text{ in }\om\setminus \pcA(t),\\\label{A8}
		-[[\nu(\xhi)D\bv]]\bn_{\pcA(t)}+[[p]]\bn_{\pcA(t)}&=-c_0\bH_{\pcA(t)},\qquad&&\text{ on }\pcA(t),\\\label{A91}
		\bv_{|t=0}&=\bv_0,&&\text{ in }\om,
	\end{align}
	for any $t\in[0,\Ts)$, with $\bv$ satisfying periodic boundary conditions. Here $c_0>0$ is the fixed surface tension constant and $M>0$ is a constant, denoting the amount of
permeability. Note that when $M=0$ the model reduces to the classical two-phase flow of incompressible viscous fluids with surface tension. Moreover, the jump $[[\cdot]]$ across the interface $\II(t)$ is defined for a function $f$ as
	\begin{align*}
		[[f]](x):=\lim_{h	\to0^+}(f (x +h\bn_{\pcA(t)})-f (x -h\bn_{\pcA(t)})),\quad \forall x\in \pcA(t).
	\end{align*}
	Furthermore, we have set  $\nu(\xhi)=\nu_1\xhi+\nu_2(1-\xhi)\in\{\nu_1,\nu_2\}$. 
    
 The system above was extensively studied numerically. We refer, for instance, to \cite{basic}, where the authors couple the fluid mechanics of the gas to the evolution of the interface. Moreover, in the fundamental series of works by Saye and Sethian \cite{Sethian1,sethian2}, the authors introduce the Voronoi Implicit Interface
Method for tracking multiple interacting and evolving regions, and apply it also to incorporate fluid dynamics into the
dynamics of dry foams, studying the same model even in its multi-component variant.

    The interest received by Navier-Stokes/Mean Curvature Flow system in recent years is not only restricted to the applications in dry foam modeling, though. Indeed, on the one hand it can be seen as a relaxation of the two-phase flow of incompressible viscous fluids with surface tension ($M=0$), since the curvature term in the interface velocity equation has a regularizing effect for the interface. Such relaxation is used in the level set methods in numerical analysis (inevitably losing mass conservation). In this regard, the Navier-Stokes/Mean Curvature Flow system is an alternative to Navier-Stokes/Mullins-Sekerka system, which can also be seen as an approximation of the two-phase problem (cf. \cite{AbelsRoger}), this time with mass conservation.
    
    On the other hand, the Navier-Stokes/Mean Curvature Flow system has often appeared in the literature as a sharp
interface limit of the Navier-Stokes/Allen-Cahn system with constant non-vanishing mobility \cite{HenselLiu, LiuSatoTone, 13LiuTone,ultimo,AbelsFei2023,AbelsLiu2018}, justifying the strong interest in its study. We also refer to \cite{LauxTaka, FM} for recent results on multi-component sharp interface limits of Allen-Cahn equations.

Starting with the literature review, it appears clear that, albeit in principle with less available regularity, the classical two-phase flow of incompressible viscous fluids with surface tension, say \eqref{A0}-\eqref{A91} with the permeability $M=0$, has been well-studied in the literature both for generalized/varifold solutions \cite{Abels2,Plotnikov}, strong solutions \cite{DenisovaSolonnikov,KPW,PSS}, as well as weak-strong uniqueness \cite{FischerHensel}.

On the contrary, the case $M>0$ still presents many unsolved open questions. Concerning local-in-time strong solutions for sufficiently smooth initial data, the literature is quite well-established and many results can be found, starting from \cite{Maekawa} to the more recent \cite{AbelsMoser} and \cite{HenselLiu}. Nevertheless, in general even if the initial datum is sufficiently smooth, the eventual
occurrence of singularities of the phase boundary or fluid velocity may not be avoided. To
account for this issue, one is required to introduce the notion of varifolds
from geometric measure theory, in order to obtain the existence of global-in-time solutions. In this direction, to the best of our knowledge the only available result for the fully coupled problem \eqref{A0}-\eqref{A91} has been obtained by Liu, Sato, and Tonegawa in \cite{LiuSatoTone}, if the Newtonian stress tensor is modified to account for
a shear thickening non-Newtonian fluid of power-law type. Namely, by means of a Galerkin approximation of a corresponding Navier-Stokes/Allen-Cahn
system they pass to the limit at the same time in the Galerkin approximation and as the thickness parameter $\varepsilon \to0$, obtaining the existence of a global weak solution. The main drawback of their approach is the fact that they are not able to treat the standard case of Newtonian fluids, as they need to assume a power law ensuring that the velocity belongs to $L^p_tW^{1,p}_x$ with $p>\frac{d+2}2$, i.e., $p>2$ when $d=2$ and $p>\frac52$ when $d=3$.  The reason is that it seems essential for the transport term $\bv$ to be regarded as a perturbative term
for the surface measure of the phase boundary. In particular, the non-Newtonian assumption $p >\frac{d+2}2$ (in place of the Newtonian case $p=2$) allows the Mean Curvature Flow to become
dominant with respect to the transport effect due to the fluid velocity $\bv$.  

Very recently Liu and Tonegawa \cite{LiuTone} have shown, in the case $d=2$, the existence of a global-in-time weak solution for $p=2$ in the sense of Brakke to the decoupled forced Mean Curvature Flow (i.e., eq. \eqref{A0}) with \textit{given} velocity field $\bv$, even in the case of multi-component Mean Curvature Flow. This solution has turned out (cf. \cite{Tashiro}) to be a generalized BV flow in the sense of \cite{StuvardTonegawa}.  This result has the potential to be extended to the coupled problem    \eqref{A0}-\eqref{A91}, but to the best of the author's knowledge, no weak existence results are so far available for    \eqref{A0}-\eqref{A91} in general dimensions. Let us point out that, when $d=3$, the authors of \cite{LiuTone} explicitly mention that the existence is an open problem, since apparently Brakke flows have some limitations in this case. Indeed, $p=2$ and $d=3$ is part of the  supercritical cases and they expect that no solution of similar type can exist, in the
sense that \textit{the forcing field might be too turbulent for the evolving surface to even retain the finiteness of the surface measure}. 

In the present contribution we overcome these issues proposing a novel notion of weak solution which is not affected by the difficulties appearing in the Brakke flow, like, for instance, the necessity of controlling the $L^2$-norm of the trace of the Sobolev function $\bv$ on the interface. The idea is to define a weak notion of solution through an energy dissipation inequality à la De Giorgi (cf. \cite{LH, Poiatti} for the -possibly volume preserving- Mean Curvature Flow, see also \cite{SH, AbelsP} for the Mullins-Sekerka flow and a related coupled system, respectively), exploiting the gradient flow structure of the Mean Curvature Flow, which can be still used in the presence of an advective term. This structure can be seen from the formal energy identity associated to the system. Let us introduce the following energy functional
	\begin{align}
		E[\xhi]:=\int_{\om}c_0 \d \normmm{\nabla\xhi}.
		\label{E1}
	\end{align}
     The energy identity related to the Mean Curvature Flow part of the problem, i.e., to \eqref{A0}, reads, for a sufficiently smooth solution,
\begin{align}
		&\ddt E[\xhi(t)]
		=-c_0\int_{\pcA(t) }(\bv(t)\cdot\bn_{\pcA(t)})H_\II\d \mathcal H^{d-1}-c_0M\int_{\II(t)}\normmm{H_\II}^2\dx,\label{enid2A0}
	\end{align}
	whereas from the velocity $\bv$ equation we infer 
	\begin{align}
		&\frac12\ddt \int_\om \normmm{\bv(t)}^2\dx =c_0\int_{\pcA(t) }(\bv(t)\cdot\bn_{\pcA(t)})H_\II\d \mathcal H^{d-1}-\int_\om \nu(\xhi(t))\normmm{D\bv(t)}^2\dx.\label{enid1A0}
	\end{align}
Therefore, summing the two, we have the total energy identity
	\begin{align}
		& \ddt \cE(t)=-\int_\om \nu(\xhi(t))\normmm{D\bv(t)}^2\dx-c_0M\int_{\II(t)}\normmm{H_\II(\cdot,t)}^2\dx,\label{enid1A01}
	\end{align}
    where we have introduced the total energy of the system:
	\begin{align}
		\mathcal{E}(t):=E[\xhi(t)]+\frac12 \int_\om \normmm{\bv(t)}^2\dx.
		\label{total1}
	\end{align}
We can then actually define a weak solution which satisfies in a suitable sense a sharp energy dissipation inequality corresponding to the indentity \eqref{enid1A01}. 

This novel notion of solution is the direct extension from the De Giorgi varifold solution to Mean Curvature Flow proposed in \cite{LH}. In particular, the main feature of this notion is that the interface velocity equation is written in the form 
\begin{align}
    \pt\xhi+\bv\cdot{\nabla \chi}+{\mmu}V=0,\label{materialvelox}
\end{align}
with $\mmu$ as the varifold mass measure,
in place of the standard equation, in the case of a smooth solution,
\begin{align*}
    \pt\xhi+\bv\cdot{\nabla \chi}+{\normmm{\nabla\xhi}}V=0,
\end{align*}
where $V=H_\II$. This structure justifies the fact that we called $V$ the total material interface (normal) velocity, since up to the sign it is the Radon-Nikod\'ym derivative of $\pt\xhi+\bv\cdot{\nabla \chi}$ (to be thought of as a Radon measure) with respect to the varifold mass measure $\mmu$. Of course, the fundamental idea of the De Giorgi solution is that the identity $V=H_\II$ is substituted by the sharp energy dissipation inequality.

Still, proving the unconditional existence of such a solution is challenging. Indeed, the existence proof proposed in \cite{LH} (and also extended to the volume preserving Mean Curvature Flow by the author in \cite{Poiatti}) relies on a sharp interface limit from the Allen-Cahn equation, based on the fundamental work by Ilmanen \cite{Ilmanen}. This is strongly connected to the use of maximum principles to prove the non-positivity of the corresponding discrepancy measures, which is not possible when an advective term like the fluid velocity is present. The sharp interface limit in this sense seems not to be the right choice to prove the existence of solutions, although it is still an interesting topic on its own. Of course, everything changes if we consider a conditional existence, assuming a well known energy convergence assumption \cite{LS}. In this case the result can be obtained by following the ideas of \cite{LauxSimon, FLS}, as it is shown in \cite{AMP} for a non isothermal version of Navier-Stokes/Allen-Cahn system. If we aim at proving the existence result \textit{unconditionally}, our idea is to rely on the novel existence proof of a De Giorgi varifold solution to Mean Curvature Flow proposed by Laux and the author in \cite{PL}. This way, no maximum principles are involved and the proof is entirely variational. The proof is based on a minimizing movement scheme which suitably approximates the degenerate $L^2$ metric associated to the gradient flow structure. Then we rely on a suitable discretization of the material time derivative, which was first introduced by Pironneau (cf. also the subsequent \cite{Gigli}), and then proposed in the context of geometric flows by Abels and the author \cite{AbelsP} to prove the existence of a De Giorgi varifold solution to a Navier-Stokes/
Mullins-Sekerka system. This is quite natural in our case, since the minimizing movement scheme contains a regularizing term which is of Mullins-Sekerka type.
Therefore, carefully combining and extending these techincal tools we are able to prove the unconditional existence result of De Giorgi solutions to system \eqref{A0}-\eqref{A91}, giving the first existence of a global weak solution to the Navier-Stokes/Mean Curvature Flow system. Namely, our first result informally reads as follows (cf. Theorem \ref{thm1} for the precise statement) 
\begin{theorem*}[Varifold solutions to Navier-Stokes/Mean Curvature Flow system]
	Let $d \in \{2,3\}$, $\xhi_0 \in BV({\mathbb T^d};\{0,1\})$, and $\bv_0\in \dot\bL^2_\sigma(\om)$.
	Then, there exists a De Giorgi varifold solution to the Navier-Stokes/Mean Curvature Flow system $(\mu,\chi,\bv)$, globally defined on $(0,\infty)$, with
	initial data $\chi_0$ and $\bv_0$.
\end{theorem*} 

Not only this, but in order to show the consistency of this new notion of solution, in the case of matched viscosities, $\nu_1=\nu_2$, we show that the De Giorgi varifold solution to Navier-Stokes/Mean Curvaure Flow does not allow for unphysical non-uniqueness before the appearance of singularities. Namely, we prove that as long as a classical solution to \eqref{A0}-\eqref{A91}, any De Giorgi varifold solution in the sense of the present work coincides with this classical flow. Our second main result then roughly reads as follows (cf. Theorem \ref{theo:weakStrongUniqueness}).
	\begin{theorem*}[Weak-strong uniqueness of varifold solutions to Navier-Stokes/Mean Curvature Flow system]
		Let $\Tstrong\in (0,\infty)$ be a finite time horizon, 
		and let $(\xhi_\mA,\bu)$ be a strong solution on $[0,\Tstrong)$ to Navier-Stokes/Mean Curvature Flow system starting from $(\xhi_{\mA(0)},\bu_0)$.
		Let also $(\mu,\chi_A,\bv)$ be a (De Giorgi) varifold solution to Navier-Stokes/Mean Curvature Flow system with initial data~$(\chi_0,\bv_0)$.
		Then the following stability estimate holds
		\begin{align*}
		 &	E[\mu,\chi_A,\bv|\xhi_\mA,\bu](T)+E_{\mathrm{bulk}}[\chi_A|\xhi_\mA](T)
			\leq C(T) (E[\chi_0,\bv_0|\xhi_{\mA(0)},\bu_0]+E_{\mathrm{bulk}}[\chi_0|\xhi_{\mA(0)}]
			)
		\end{align*}
		for some constant $C=C(T)>0$ and for almost every $T \in [0,\Tstrong]$. Here $\Er$ and $\Eb$ are a relative entropy functional and a bulk error functional, respectively (see \eqref{eq:relativeEntropyIntro}-\eqref{eq:bulkErrorIntro}).
		
		More precisely, if it holds that $\bv_0=\bu_0$ and  $\chi_0=\chi_{\mA(0)}$, then for almost any $t \in [0,\Tstrong]$ the two solutions coincide.
	\end{theorem*}

As it can be seen from the above quantitative stability inequality involving $\Er$ and $\Eb$, the proof of this second result is based on a relative entropy approach. This approach has been developed to prove the weak-strong uniqueness for the two-phase flow of incompressible fluids by Fischer and Hensel \cite{FischerHensel}, and since then it has been extended to deal with many different problems, like for many notions of weak solutions to (volume-preserving) Mean Curvature flow (cf. \cite{FHLS, LH, Poiatti, Timvol}) or Mullins-Sekerka flow (\cite{WSuniqueness}). 
The method can also be used to
show quantitative convergence of the Navier-Stokes/Allen-Cahn equation to \eqref{A0}-\eqref{A91} (cf. \cite{HenselLiu}) or to the two-phase flow \cite{AbelsFischerMoser}. These results are based on the pioneering \cite{FLS}, and the method has also been extended in \cite{Timkroemer} to the convergence of a suitable nonlocal Allen-Cahn equation to a volume-preserving mean curvature flow. Additionally, the extension of the relative entropy technique to smooth, strongly convex anisotropies in  surface tension and mobility has been addressed in the literature, and we refer to \cite{Timnotes,EP, LauxClemens}. Note that, by following the technique developed in \cite{FischerHensel}, we expect to be able to prove the same weak-strong uniqueness result also in the case of unmatched viscosities $\nu_1\not=\nu_2$, at the price of a more technical analysis, which goes beyond the scopes of the present work and will be postponed to a future dedicated study.

Finally, we mention that the main drawback of the Navier-Stokes/Mean Curvature Flow system \eqref{A0}-\eqref{A91} is the lack of mass conservation along the flow. This can be overcome if we consider the volume-preserving variant of the Mean Curvature Flow in the coupling. This problem has also recently appeared as the sharp interface limit of a mass-conserving Navier-Stokes/Allen-Cahn system \cite{AbelsHanifah}. By means of the minimizing movement scheme developed in \cite{PL}, which also accounts for the volume-preserving Mean Curvature Flow, we are confident that the existence of global weak solutions can be extended to the volume-preserving setting. This will be objective of future work as well.

  \subsection*{Structure of the paper}
The rest of the paper is organized as follows. In Section \ref{prelim} we give some preliminary notation which will be of use in the sequel. Section \ref{mres} is then devoted to the presentation and discussion of the main results. The following sections are dedicated to the proofs. Namely, in Section \ref{proofthm1} we prove Theorem \ref{thm1}, whereas in Section \ref{proofws} we show the proof of Theorem \ref{theo:weakStrongUniqueness}.

\section{Preliminaries}
\label{prelim}
	 Let $\om$, $d=2,3$, be the flat torus in two and three dimensions. We denote by $\mathcal L^d$ the $d$-dimensional Lebesgue measure, and  $\mathcal H^m$ is the $m$-dimensional Hausdorff measure. We then set $C^k(\om)$, $k\in\N\cup\{\infty\}$ to be the space of continuous functions on $\om$ which are $k$-times continuously differentiable. Sobolev spaces are as usual denoted by $W^{k,p}(\om )$%
	, where $k\in \mathbb{N}$ and $1\leq p\leq \infty $, with norm $\Vert \cdot
	\Vert _{W^{k,p}(\om )}$. The Hilbert space $W^{k,2}(\om )$ is here identified with $H^{k}(\om )$ with norm $\Vert \cdot \Vert _{H^{k}(\om )}$. Additionally, we denote by $BV(\om;\{0,1\})$ the space of functions of bounded variation over $\om$ with values in $\{0,1\}$.
	Moreover, given a (real) vector space $X$, we denote by $\mathbf{X}$ the space of vectors or matrices with each component belonging to $X$. In this situation $\vert \mathbf{v} \vert$ is the Euclidean norm
	of $\mathbf{v}\in  \mathbf{X}$.
	We then denote by $(\cdot,\cdot )_\om$ the inner product in ${L}^{2}(\om )$ (or $\bL^2(\om)$) and by $\Vert \cdot \Vert_{L^2(\om)}$ ($\norm{}_{\bL^2(\om)}$, respectively)
	the corresponding induced norm. Moreover, we denote by $\langle\cdot,\cdot \rangle_{X',X}$ the duality paring between $X'$ and $X$, given $X$ a generic Banach space.  We recall that on $\om$ generic functions are characterized by their Fourier expansion
$$
f= \sum_{ k\in \mathbb{Z}^2} \widehat{f}_k \mathrm{e}^{i k\cdot x}, 
\quad \text{where} \quad \widehat{f}_{-k}=\overline{\widehat{f}_{k}}, \quad \widehat{f_k}= \frac{1}{(2\pi)^2} \int_{\om^2} f(x) \mathrm{e}^{- i k\cdot x} \, \d x,
$$
with $\overline{z}$ as the complex conjugate of $z \in \mathbb{C}$.
For the velocity field we then set (cf. \cite{Temam})

\begin{align*}
&\dot\bL^2_\sigma=\{ \bu\in \mathbf{L}^2(\om): \widehat{\bu}_k \cdot k=0 \quad \forall \, k \in \mathbb{Z}^2,\quad  \widehat\bu_0=\mathbf 0 \}, \\& \Hs=\bH^1(\om)\cap \dot\bL^2_\sigma,
\end{align*}
which are endowed with the norms $\| \bu\|_{\dot\bL^2_\sigma}= \| \bu\|_{\bL^2(\om)}$, $\| \bu\|_{\Hs}= \| \nabla\bu\|_{\bL^2(\om)}$. This is possible by Poincaré's inequality since elements in $\dot\bL_\sigma^2(\om)$ have zero spatial average.
Morever, Korn's inequality gives
\begin{equation}
\label{korn} 
\|\nabla\bu\|_{\bL^2(\om)} \leq \sqrt2\|D\bu\|_{L^2(\om)}, \quad \forall \, \bu \in \Hs,
\end{equation}
so that also $\| D \bu\|_{\bL^2(\om)}$ is a norm on $\Hs$, which is equivalent to $\| \bu\|_{\Hs}$.

Then, we also define
\begin{align*}
	\dot\bH^k_\sigma(\om):=\Hs\cap \bH^k(\om),\quad k\in \N\setminus\{0,1\}.
\end{align*}
We now set $\mathbf{P}_\sigma$ as the Leray projector from $\mathbf{L}^2(\om)$ to $\bL_\sigma(\om)$ and we also introduce $\bA=-\bP_\sigma\Delta: \mathfrak{D}(\bA)=\bH^2(\om)\cap \Hus\subset \dot\bL^2_\sigma(\om)\to \dot\bL^2_\sigma(\om)$ as the standard Stokes operator.

 Let $X$ be a Banach space. We denote by $L^q (a,b; X )$, $0\leq a\leq b$, $q\in[1,\infty]$, the Bochner space of $X$-valued $q$-integrable (essentially bounded functions, respectively).
Furthermore, $f:[0,\infty)\to X$ belongs to $L^q_{loc} ([0, \infty); X )$ if and only if it is strongly measurable with values in $X$ and $f \in L^q (0, T ; X )$ for every $T > 0$. Given an interval $J$, the function space 
$C_c^\infty (J ; X)$ represents the vector space of all $C^\infty$-functions $f : J\to X$ with compact support in $J$. Moreover, $W^{1,p} (0, T ; X)$, $1 \leq p < \infty$, is the space of functions $f$ such that $\ddt f\in L^p(0,T;X)$ and $f\in L^p(0,T;X)$, with $\ddt$ denoting the vector-valued distributional derivative of $f$. Finally, if $Y=X'$ is a dual space, then $L^\infty_{w*}(0,T;Y)$, $T>0$, denotes the space of all functions $\eta: (0,T)\to Y$ that are weakly* measurable and essentially bounded, that is, the map
$$
t\mapsto \langle \eta(t),f\rangle_{Y,X}
$$ 
is measurable for any $f\in X$, and 
\begin{align}
\norm{\eta}_{L^\infty_{w*}(0,T;Y)}:=\esssup_{t\in(0,T)}\norm{\eta(t)}_Y.
	\label{norm1}
\end{align}
Recalling (cf. \cite{AmbrosioFuscoPallara}) that $BV(\om)$ is the dual of a separable Banach space $X$, we have that $L^\infty_{w*}(0,T;BV(\om))=(L^1(0,T;X))'$ and that uniformly bounded sets in $L^\infty_{w*}(0,T;BV(\om))$ are weakly* precompact.

	\subsection{Measures and varifolds}
	For a locally compact and separable metric space $X$, we denote
	by $\mathcal M(X)$ the space of (finite) Radon measures on $X$. An (oriented) varifold $\mu$ on $\om$ is defined as positive Radon measure $\mu \in \mathcal M(\om\times\mathbb S^{d-1})$. The corresponding varifold mass measure is denoted by $\normmm{\mu}_{\mathbb S^{d-1}}\in \mathcal M({\om})$.  We say that the varifold $\mu$ is
$(d-1)$-rectifiable if $\normmm{\mu}_{\mathbb S^{d-1}}$ is $(d-1)$-rectifiable. We also say that $\mu$ is $(d - 1)$-integer-rectifiable if the $(d - 1)$-density $\theta^{d-1}(\normmm{\mu}_{\mathbb S^{d-1}},\cdot)$ of $\normmm{\mu}_{\mathbb S^{d-1}}$ is integer valued.  
In conclusion, the first inner variation of the varifold is here denoted by $\delta\mu(B)$ (see \cite{Allard}), i.e.,
\begin{align}
	\delta \mu(B)=\int_{{\om}\times\mathbb{S}^{d-1}}(Id-p\otimes p):\nabla B(x)\d\mu(x,p),
	\label{allard1}
\end{align}
for any $B \in C^1({\om};\R^d )$. 
	\section{Main results}
    \label{mres}
    \subsection{Varifold weak solution to Navier-Stokes/Mean Curvature system}
For the sake of simplicity, in the following we assume without loss of generality that $M=c_0=1$ in system \eqref{A0}-\eqref{A91}.
    
    \label{Sec::mainresults}
	First we recall the admissibility conditions for a couple given by evolving phase indicators and varifolds (cf., e.g., \cite{LH,PL}).
	\begin{definition}[Admissible couples of evolving phase indicators and varifolds]\label{admissible}
	Let $d=2,3$, and set $\Ts\in(0,\infty)$. Let $\xhi\in L^\infty_{w*}(0,\Ts; BV({\mathbb T^d};\{0,1\}))\cap C([0,\Ts];L^p(\om))$ for any $p\in[1,\infty)$. Consider a family $\mu=(\mu_t)_{t\in(0,\Ts)}$ of weakly* measurable in time oriented varifolds $\mu_t\in \mathcal M({\mathbb T^d}\times \mathbb S^{d-1})$, for all $t\in(0,\Ts)$. We say that the couple $(\chi,\mu)$ is admissible if the following properties are satisfied: 
	\begin{enumerate}
	
		\item\textit{(Compatibility with the phase indicator).} 
        The oriented varifold $\mu_t$ is compatible with the interface associated  with the phases modeled by $\chi$, i.e., for almost every $t\in(0,\Ts)$, for every $\eta\in C^{1}({\mathbb T^d};\mathbb R^d)$ 
      
		\begin{align}
			&\int_{{\mathbb T^d}\times \mathbb{S}^{d-1}}p\cdot \eta(x)\,\d\mu_t(x,p)=\int_{\mathbb T^d}  \eta(x)\cdot \d{\nabla\chi(x,t)}.\label{PE}
		\end{align}

\item\textit{(Measurability in time of the mass measure).} The total mass measure associated with the oriented varifold $\mu_t$, say
\begin{align}
	E[\mu_t]:=\normmm{\mu_t}_{\mathbb S^{d-1}}({\om}),
	\label{PL}	
\end{align}
is a measurable map from $t\in(0,\Ts)$ to $E[\mu_t]\in[0,\infty)$.
\end{enumerate}
\label{admissibility}
\end{definition}
 In order to give our definition of weak solutions to Navier-Stokes/Mean Curvature Flow system, following \cite{PL},  for almost any $t>0$, given $(\xhi,\mu)$ admissible as in Definition \ref{admissible}, we introduce the Hilbert space
    \begin{align}
        {\mathcal V_t}:=L^2({\mathbb T^d};\d \normmm{\mu_t}_{\mathbb S^{d-1}}),\label{Vchi}
    \end{align}
    endowed with the canonical $L^2({\mathbb T^d};\d \normmm{\mu_t}_{\mathbb S^{d-1}})$ inner product.
   We can now give our definition of weak solution.
	\begin{definition}
		[Varifold solutions to Navier-Stokes/Mean Curvature Flow system]
		Let $d=2,3$, consider $\Ts\in(0,\infty)$, and let $\mathbb T^d$ be the $d$-dimensional flat torus. Fix also $\chi_0\in BV({\mathbb T^d};\{0,1\})$ and $\bv_0\in \dot\bL^2_\sigma(\om)$. 
		A vector field $\bv:\om\times(0,\Ts)\to \R^d$, together with a measurable map $\xhi:{\mathbb T^d}\times(0,\Ts)\to \{0,1\}$, and a family $\mu=(\mu)_{t\in(0,\Ts)}$ of oriented varifolds $\mu_t\in \mathcal M (\om\times\mathbb{S}^{d-1})$, $t\in(0,\Ts)$, is called a \textit{varifold solution }for the Navier-Stokes/Mean Curvature Flow system, with time horizon $\Ts$ and initial data $(\xhi_0,\bv_0)$, if 
		\begin{enumerate}
			\item\textit{(Admissibility and regularity).} $(\xhi,\mu)$ is an admissible couple in the sense of Definition \ref{admissibility}. 
	\item \textit{(Regularity of fluid velocity $\bv$).} The vector field $\bv$ is such that 
    \begin{align*}
        \bv\in H^1(0,\Ts;\dot\bH_\sigma^3(\om)')\cap L^\infty(0,\Ts;\dot\bL^2_\sigma(\om))\cap L^2(0,\Ts;\dot\bH^1_\sigma(\om)),
    \end{align*}
    with $\bv(0)=\bv_0$ in $\dot\bL_\sigma^2(\om)$.
		\item\textit{(Square integrable material interface velocity).} There exists a total material interface velocity $V\in L^2(0,\Ts;L^2({\mathbb T^d},\d\normmm{\mu_{(\cdot)}}_{\mathbb S^{d-1}}))$, such that 
		\begin{align}
		&\int_{\om}\xhi(\cdot,T')\zeta(\cdot, T')\dx-\int_{\om}\xhi_0\zeta(\cdot,0)\dx\nonumber\\&
        =\int_0^{T'}\int_{\mathbb T^d} \xhi(\cdot,t)(\pt \zeta(\cdot,t)+\bv(\cdot,t)\cdot\nabla\zeta(\cdot,t))\dxdt-\int_0^{T'}\int_{\om}V\zeta\d\normmm{\mu_t}_{\mathbb S^{d-1}}\dt,
			\label{kinetic}
		\end{align} 
		for all $\zeta\in C^\infty_c({\mathbb T^d}\times[0,\Ts])$, for almost any $T'\in(0,\Ts)$; 
		\item \textit{(Navier-Stokes equation)} It holds 
\begin{align}
    \langle \pt\bv,\bzeta\rangle_{\dot\bH^3_\sigma(\om)',\dot\bH^3_\sigma(\om)}-\int_\om \bv\otimes \bv: \nabla\bzeta\dx+\int_\om \nu(\xhi)D\bv:D\bzeta\dx=-\delta\mu_t(\bzeta),\label{NavS}
\end{align}
for every $\bzeta\in \dot\bH^3_\sigma(\om)\hookrightarrow \mathbf C^1(\om)$ and for almost any $t\in(0,\Ts)$. 
		\item \textit{(Sharp energy dissipation inequality).} The following energy inequality holds: 
		\begin{align}
	        \nonumber&\mathcal E(T')+\int_s^{T'}\int_\om\nu(\xhi)\normmm{D\bv}^2\dx\dt+\onehalf\int_s^{T'}\int_{\mathbb T^d}\normmm{V}^2\d \tmmu\dt +\onehalf\int_s^{T'} \normmm{\partial E(\mu_t)}^2_{{\mathcal V_t}}\dt\\&\leq \mathcal E(s),
			\label{advectiveMullins}	
		\end{align}
		 for almost any $T'\in(0,\Ts)$ and almost any $s\in[0,T')$, including $s=0$, where
        $$
        \mathcal E(t):=\frac12\int_\om\normmm{\bv(t)}^2\dx +E[\mu_{t}],\quad
        \text{with }\quad 
		 E[\mu_t]:=\begin{cases}
		 	\tmmu(\om),& t>0,\\
		 	E[\xhi_0],& t=0,
		 	\end{cases}
		 $$
		 and $E[\xhi_0]=\normmm{\nabla \chi_0}({\mathbb T^d})$, together with
          \begin{align}
   \frac12 \int_s^{T'}\normmm{\partial E(\mu_t)}^2_{{\mathcal V_t}}\dt :=\sup_{B\in C^1_c((s,T');C^1(\om;\R^d))}\left(\int_s^{T'}\delta \mu_t(B(t))\dt-\frac12\int_s^{T'}\norm{B(t)}_{\mathcal V_t}^2\dt\right),\label{slope}
\end{align}
		\end{enumerate}	 
		\label{weaksol}   
	\end{definition}
    \begin{remark}[Existence of a generalized mean curvature vector]\label{meancurvature}
    Recalling \cite[{Remark 3.3}]{PL}, the energy inequality \eqref{advectiveMullins} allows to obtain the existence of a generalized mean curvature vector directly as a consequence of \eqref{advectiveMullins}. Indeed, in the case of Navier-Stokes/Mean Curvature Flow system, we infer by the Riesz representation theorem that there exists a vector field $\mathbf H_{\normmm{\mu}_{\mathbb S^{d-1}}}\in L^2(0,\Ts;L^2({\mathbb T^d}; \d\normmm{\mu_{(\cdot)}}_{\sf};\R^d))$ satisfying 

	\begin{align}
		\int_0^\Ts\delta\mu_t(B)\dt=-\int_0^\Ts\int_\om \mathbf H_{\tmmu}\cdot B\,\d\tmmu\dt,\quad \forall B\in C_c^1((0,\Ts)\times\om; \mathbb R^d),
		\label{PH}
	\end{align}
    as well as
$$
 \frac12 \int_s^{T'}\normmm{\partial E(\mu_t)}^2_{\mathcal V_t}\dt=\frac12 \int_s^{T'}\int_{{\mathbb T^d}} \normmm{\mathbf H_{\normmm{\mu_t}_{\mathbb S^{d-1}}}}^2\d\tmmu\dt,\quad \forall s,T'\in(0,\Ts).
$$
As a further consequence, after a localization in time and a density argument, this also entails that the Navier-Stokes equation reads
\begin{align}
    \langle \pt\bv,\bzeta\rangle_{\dot\bH^2_\sigma(\om)',\dot\bH^2_\sigma(\om)}-\int_\om \bv\otimes \bv: \nabla\bzeta\dx+\int_\om \nu(\xhi)D\bv:D\bzeta\dx=\int_\om \bH_{\normmm{\mu_t}_{\sf}}\cdot \bzeta\d\normmm{\mu_t}_{\sf},\label{NavS2}
\end{align}
for any $\bzeta\in \dot\bH^2_\sigma(\om)\hookrightarrow\mathbf C(\om)$ and almost any $t\in(0,\Ts)$. This also allows to improve the regularity of the fluid velocity $\bv$ up to 
\begin{align}
    \bv\in H^1(0,\Ts;\dot\bH^2_\sigma(\om)').
\end{align}

\end{remark}

We then state our main existence result.
\begin{theorem}[Varifold solutions to Navier-Stokes/Mean Curvature Flow system]
	Let $d \in \{2,3\}$, $\xhi_0 \in BV({\mathbb T^d};\{0,1\})$, and $\bv_0\in \dot\bL^2_\sigma(\om)$.
	Then, there exists a weak varifold solution to the Navier-Stokes/Mean Curvature Flow system, globally defined on $(0,\infty)$, with
	initial data $\chi_0$ and $\bv_0$, in the sense of Definition \ref{weaksol} on any interval $(0,\Ts)$, for any $\Ts>0$.
	\label{thm1}
\end{theorem}  
\ \\ \

\begin{remark}[Generalized mean curvature vectors and varifold rectifiability]
Remark \ref{meancurvature} holds for any $\Ts>0$ (cf. \cite[Remark 2.6]{PL}). Therefore, we obtain the existence of a generalized mean curvature vector such that $\mathbf H_{\normmm{\mu}_{\mathbb S^{d-1}}}\in L^2(0,\infty;L^2({\mathbb T^d}; \d\normmm{\mu_{(\cdot)}}_{\sf};\R^d))$ and \eqref{PH} holds with $\Ts=\infty$.
As a consequence, Allard's rectifiability criterion \cite{Allard} allows to conclude that the unoriented varifold $\widehat\mu_t$, corresponding to $\mu_t$  (identifying antipodal points $\pm p$ on $\mathbb S^{d-1}$), is $(d-1)$-rectifiable for almost any $t\in(0,\infty)$.
Observe that, differently from the notion of generalized solutions for two-phase flows for
viscous incompressible fluids (cf. \cite[Appendix A]{Abels2} where the author clarifies that no rectfiability seems to be possible to prove for the associated varifold), here varifold rectifiability can be obtained as a consequence of the interface motion by the (advective) Mean Curvature Flow.
On the other hand, integral rectifiability remains an interesting open question also in this case, since we adopt a minimizing movements scheme in the spirit of \cite{PL} (cf. \cite[Remark 2.7]{PL}).
\end{remark}

\subsection{Weak-strong uniqueness principle}
	In addition to Theorem \ref{thm1}, we can also show a weak-strong uniqueness result for our new notion of solution. To this aim we first define the notion of strong solution to Navier-Stokes/Mean Curvature Flow system. Since we need to introduce a suitable calibration, we follow the precise definition proposed in \cite[Def. 4]{HenselLiu}. As already observed in the introduction, here we concentrate for simplicity on the case with same viscosities, i.e., $\nu_1=\nu_2=1$. The case with variable viscosities is considerably more technical, but we expect that it can be studied following \cite{FischerHensel} and it will be subject to future work. 

\begin{definition}[Strong solutions to Navier-Stokes/Mean Curvature Flow system]
\label{def:strongSolSharpInterfaceLimit}
\begin{subequations}
Let~$d \in \{2,3\}$
and consider a finite time horizon $0 < \Tstrong < \infty$. Assume an initial velocity field $\bu_{0} \in C^1(\om;\Rd)$ with $\int_\om \bu_0\dx=\mathbf0$, such that $\nabla\cdot \bu_{0} = 0$ in~$\om$, and an initial phase indicator $\chi_{0} \in BV(\om;\{0,1\})$. Assume also that the set
$\mA(0):=\{x\in \om:\ \chi_0(x) {=} 1\}$ is an open subset of~$\om$ consisting of finitely
many connected components with ${\mA(0)} \subset \om$,
the associated boundary~$\cI(0)=\partial\mA(0)$ is orientable and uniformly of class~$C^{3}$, 
and the following compatibility conditions hold true:
\begin{align*}
\jump{ \bu _0 }=0 ~&\text {on}~ \II(0),\\
(\mathrm{Id} {-} \bn_{\II(0)} \otimes \bn_{\II(0)})
\jump{   D \bu _0  \cdot\bn_{\II(0)}  }=0~&\text{on}~\II(0).
\end{align*}

A pair of measurable maps $(\xhi,\bu)\colon\om{\times}(0,\Tstrong] \to \{0,1\}{\times}\Rd$
is called a {strong solution to the }Navier-Stokes/Mean Curvature Flow system  with time horizon~$\Tstrong$ and initial data~$(\chi_0,\bu_0)$
if the following conditions hold true:
\begin{itemize}[leftmargin=0.7cm]
\item[i)] The fluid velocity~$v$ is such that
					\begin{align}
					\label{eq:conditionsStrongSolSharpInterfaceLimit1}
					\bu \in H^1\big(0,\Tstrong;\dot\bL^2_\sigma(\om)\big)
					\cap L^2\big(0,\Tstrong;\Hs\big)
					\end{align}
					such that $\bu(\cdot,t) \to \bu_0\ \text{strongly in } \dot\bL_\sigma^2(\om) \text{ as } t \downarrow 0$. Moreover, for each $T \in [0,\Tstrong]$ there exists
					$C=C(T) \in (0,\infty)$ such that
					\begin{align}
					\label{eq:conditionsStrongSolSharpInterfaceLimit4}
					\bu &\in L^\infty\big(0,T;C^{1}(\om;\Rd)\big),\\&
\label{eq:onditionsStrongSolSharpInterfaceLimit6}
					\|\nabla \bu\|_{L^\infty(\om{\times}(0,T))} \leq C.
					\end{align}
\item[ii)] The velocity field~$\bu$ solves~\eqref{A6}--\eqref{A91} 
					 in the sense that
					 \begin{align}
					 \label{eq:conditionsStrongSolSharpInterfaceLimit7}
					 \nonumber
					 &\int_{\om} \bu(\cdot,T) \cdot \eta(\cdot,T) \,dx 
					 - \int_{\om} \bu_0 \cdot \eta(\cdot,0) \,dx
					 \\&
					 = \int_0^T\int_{\om} \bu\cdot \partial_t\eta \,dx dt
					 - \int_0^T\int_{\om} \eta \cdot (\bu \cdot \nabla) \bu \,dx dt
					 \\&~~~ 
					 \nonumber
					 - \int_0^T\int_{\om} D\bu : D \eta \,dx dt
					 + \int_0^T\int_{\om} \frac{\nabla\xhi}{|\nabla\xhi|} 
					 \otimes \frac{\nabla\xhi}{|\nabla\xhi|} : \nabla \eta \,d|\nabla\xhi| dt
					 \end{align}
					 holds true for a.e.\ $T \in (0,\Tstrong)$ and all $\eta \in 
					 C^\infty_c([0,\Tstrong];C^\infty_c(\om;\Rd))$ with $\nabla\cdot\eta=0$.
\item[iii)] Define for every $t \in (0,\Tstrong]$ the set $\mA(t) := \{x\in\om:\ \chi(x,t){=}1\}$.
            There exists a map $\Psi\colon 
						\om {\times} [0,T_*) \to \om$ such that
						$\Psi(\cdot,0) = \mathrm{Id}$, $\Psi(\cdot,t) \in 
						C^{3}\text{-Diffeo}(\om,\om)$ for all $t \in (0,\Tstrong]$,
						and $\mA(t) = \Psi(\mA(0),t)$	for all $t \in (0,\Tstrong)$.
						We also assume that for all $T \in (0,\Tstrong]$
						it holds
						\begin{align}
						\label{eq:conditionsStrongSolSharpInterfaceLimit9}
						\Psi &\in C^1\big([0,T];C^{1}(\om;\Rd)\big)
						\cap C\big([0,T];C^{3}(\om;\Rd)\big).
						\end{align}
\item[iv)] The following geometric evolution equation is satisfied:
\begin{align}
V_\II=\bu\cdot\bn_\II+H_\II\quad\text{ on }\II,
    \label{geom}
\end{align}
		where $V_\II(·, t)$ and $H_\II(·, t)$ denote the normal speed, satisfying $\pt\xhi+V_\II\normmm{\nabla\xhi}=0$, and the scalar mean curvature of the interface $\II(t)$, $t\in[0, \Tstrong]$, respectively. 			 
\end{itemize}
\end{subequations}
\end{definition}
\begin{remark}
    Existence and uniqueness of a local-in-time strong solution in the above sense can be found, for instance, in \cite[Appendix A]{HenselLiu}, which also relies on the well-posedness result in \cite{AbelsMoser}.
    \end{remark}
We now address the weak-strong uniqueness enjoyed by our new notion of
	  (De~Giorgi-type) varifold solution to Navier-Stokes/Mean Curvature Flow system.  Let $(\mu,\chi_A,\bv)$ be a varifold solution
	according to Definition~\ref{weaksol}, where we set $A(t):=\{x\in\om:\xhi(x,t)=1\}$. Let also $(\xhi_\mA,\bu)$ be a strong solution according to Definition \ref{def:strongSolSharpInterfaceLimit}.
	
	To measure the distance
	between $(\xhi_\mA,\bu)$ and~$(\mu,\chi_A,\bv)$ we exploit the relative entropy
	\begin{align}
		\label{eq:relativeEntropyIntro}
		\Er :=\frac12\int_\om \normmm{\bv(x,t)\-\bu(x,t)}^2\dx+ \int_{\om {\times} \S} 
		1 - p \cdot \xi(x,t) \d\mu_t(x,p), 
		\quad t \in [0,\Tstrong],
	\end{align}
	together with the bulk error, which is fundamental to control some coupling terms, 
	\begin{align}
		\label{eq:bulkErrorIntro}
		\Eb := \int_{\om} 
		|\chi_{A(t)}(x) {-} \chi_{\mA(t)}(x)| |\vartheta(x,t)| \dx, 
		\quad t \in [0,\Tstrong].
	\end{align}
    Notice that, since in the definition of weak solution, i.e., Definition \ref{weaksol}, we assumed that $E[\mu_t|_{t=0}]=\normmm{\nabla\xhi_0}(\om)$, using the compatibility condition \eqref{PE} we have
    \begin{align}
        E[\xhi_0,\bv_0|\xhi_{\mA(0)},\bu_0]:=\Erel{0}=\normmm{\nabla\xhi_0}(\om)+\int_\om\xhi_0\Div \xi(\cdot,0)\dx.\label{initi}
    \end{align}
One can of course generalize assuming $\mu_t|_{t=0}=\mu_0$ to be a general varifold, without necessarily assuming $E[\mu_0]=\normmm{\nabla\xhi_0}(\om)$. The vector field~$\xi$ in the definition \eqref{eq:relativeEntropyIntro} is a suitable
	extension of the normal vector field associated
	with the strong solution~$(\xhi_\mA,\bu)$ (pointing inward), and the weight~$|\vartheta|$
	is a smooth version of $(x,t) \mapsto \min\{1,\dist(x,\mathcal{I}(t))\}$, which disappears on $\mathcal I(t):=\partial\mA(t)$.
	For a rigorous construction of $(\xi,\vartheta)$ we refer to Lemma \ref{lemma:calibration}, whose proof is straightforward and can be found, for instance, in \cite[Proof of Theorem 1]{HenselLiu}. Based upon this construction, we can show a quantitative stability for (De~Giorgi-type) varifold solutions to Navier-Stokes/Mean Curvature Flow system as in Definition \ref{weaksol}. We thus have our second main result as follows
	\begin{theorem}[Weak-strong uniqueness  for varifold solutions to Navier-Stokes/Mean Curvature Flow system]
		\label{theo:weakStrongUniqueness}
		Let $\Tstrong\in (0,\infty)$ be a finite time horizon, 
		and let $(\xhi_\mA,\bu)$ be a strong solution on $[0,\Tstrong)$ according to Definition \ref{def:strongSolSharpInterfaceLimit}.
		Let also $(\mu,\chi_A,\bv)$ be a varifold solution to Navier-Stokes/Mean Curvature Flow system with initial data~$(\chi_0,\bv_0)$, according to
		Definition~\ref{weaksol}.
		The following stability estimate holds
		\begin{align}
		 &	E[\mu,\chi_A,\bv|\xhi_\mA,\bu](T)+E_{\mathrm{bulk}}[\chi_A|\xhi_\mA](T)
			\leq C(T) (E[\chi_0,\bv_0|\xhi_{\mA (0)},\bu_0]+E_{\mathrm{bulk}}[\chi_0|\xhi_{\mA(0)}]
			)
			 \label{eq:stabilityglobal}
		\end{align}
		for some constant $C=C(T)>0$ and for almost every $T \in [0,\Tstrong]$.
		
		More precisely, if it holds that $\bv_0=\bu_0$ and  $\chi_0=\chi_{\mA(0)}$, then for almost any $t \in [0,\Tstrong]$ we have
		\begin{align}
        \bv(\cdot,t)&=\bu(\cdot,t) \text{ almost everywhere in }\om,\\
			\label{eq:weakStrong1}
			\chi_{A(t)}(\cdot) &= \chi_{\mA(t)}(\cdot)\text{ almost everywhere in } \om,
			\\ \label{eq:weakStrong2}
			\mu_t &= |\nabla\chi(\cdot,t)|\otimes(\delta_{\frac{\nabla\chi(\cdot,t)}{|\nabla\chi(\cdot,t)|}(x)})_{x\in\om}.
		\end{align}
	\end{theorem}
\begin{remark}[Navier-Stokes/Mean Curvature Flow approximating the classical two-phase flow]
Estimate \eqref{eq:stabilityglobal} allows to infer that the distance between a weak solution to Navier-Stokes/Mean Curvature Flow system and a strong solution to the classical two-phase flow of incompressible viscous fluids with surface tension (on the time interval where it exists) can be estimated just after a comparison between such a strong solution and a suitable strong solution to the Navier-Stokes/Mean Curvature Flow system. This comparison between strong solutions has already been performed (cf. \cite[Lemma 7.9]{AbelsFischerMoser}) and thus, thanks to our new result \eqref{eq:stabilityglobal}, the same quantitative comparison gives a rate of convergence, as the permeability coefficient vanishes, also for our varifold solution. This proves that the Navier-Stokes/Mean Curvature Flow system can be effectively seen as an approximation of the classical two-phase flow also at the level of a varifold solution.          
\end{remark}

\section{Proof of Theorem \ref{thm1}} 
\label{proofthm1}
The proof of the existence of De Giorgi varifold solutions is rather technical and combines the minimizing movements approach first introduced in \cite{PL}, together with the discretization of the material derivative used in \cite{AbelsP}, from which we take inspiration also for the approximation scheme of the Navier-Stokes equation. 
Let us introduce some notation. We fix $h>0$ and $\alpha>0$.  Let us then consider a symmetric mollifier such that 
\begin{align*}
    \rho\in C^\infty_c(B_1(0)),\quad \int_{\mathbb T^d} \rho(x)\dx=1,\quad \rho\geq 0,\quad \rho(x)=\rho(\normmm{x}),
\end{align*}
    and define 
    \begin{align*}
        \rho_h(x):=h^{-d}\rho\left(\frac{x}h\right),
    \end{align*}
    We then define
\begin{align*}
    \mathcal M:=BV(\om;\{0,1\}).
\end{align*}
    Given $\xhi\in \mathcal M$, we introduce as in \cite{PL} the selfadjoint invertible (with compact inverse) operator 
\begin{align*}
   \mathcal A_{\alpha,h}[\xhi]:= \normmm{\rho_h\ast \nabla \xhi}+\alpha^2(I-\Delta).
\end{align*}
For the sake of readability, we will make the slight about of notation writing $\Ah^s[\xhi]$ in place of the more formal $(\Ah[\chi])^s$ for any power $s\in\R$. We finally recall here that, for any $\xhi\in \mathcal M$,
\begin{align}
\normh{\Ah^\onehalf[\xhi]f}^2=\alpha^2\normh{\nabla f}^2+\alpha^2\normh{f}^2+\int_\om\normmm{\rho_h\ast \nabla\xhi}f^2\dx.
    \label{identif}
\end{align}
\subsection{Some preliminary lemmas}
We first start by presenting some preliminary lemmas, which are suitable modifications of the ones presented in \cite{AbelsP,PL}. Let us start with a lemma on a norm equivalence, which extends \cite[Lemma A.1]{PL}.
\begin{lemma}
    \label{Lemmaequiv1}
    Let $\xhi\in \mathcal M$. Then it holds, for any $\alpha>0$,
\begin{align}
\norm{f}_{(C({\mathbb T^d})\cap H^1({\mathbb T^d}))'}\leq \sqrt 3\sqrt{\max\{ \alpha^2, {\normmm{\nabla\xhi}({\mathbb T^d})}\}}\normh{\mathcal A_{\alpha,h}^{-\frac12}[\xhi]f},\quad \forall f\in H^1(\om)'.
\label{equiv1a}
\end{align}
 Moreover, it also holds 
\begin{align}
\nonumber  &\normh{\mathcal A_{\alpha,h}^{-\frac12}[\xhi]f}\\&=\sup_{w\in H^1(\om),\ \normh{\mathcal A_{\alpha,h}^{{\frac12}}[\xhi]w}\leq1}\ \langle f,w\rangle_{H^1(\om)',H^1(\om)},\qquad \forall f\in H^1(\om)',\label{crucial}
\end{align}
and also
\begin{align}
    \norm{\Ah^{-1}[\xhi]f}_{H^1(\om)}\leq \alpha^{-1}\norm{f}_{H^1(\om)'},\quad \forall f\in H^1(\om)'.\label{forbounds}
\end{align}
Finally, we have
\begin{align}
    \normh{\nabla \Ah^{-1}[\xhi]f}\leq \frac1\alpha\normh{\Ah^{-\onehalf}[\xhi]f}.\label{ct1}
\end{align}
\end{lemma}
\begin{proof}
    The proof of \eqref{equiv1a} can be found in the proof of \cite[Lemma A.1]{PL}. Concerning \eqref{crucial}, given $f\in H^1(\om)'$, we immediately verify that 
    \begin{align*}
        &\normh{\mathcal A_{\alpha,h}^{-\onehalf}[\xhi]f}\\&
        =\sup_{\xi\in H^1(\om)',\ \normh{\mathcal A_{\alpha,h}^{-\onehalf}[\xhi]\xi}\leq1}\left(\mathcal A_{\alpha,h}^{-\onehalf}[\xhi]f,\mathcal A_{\alpha,h}^{-\onehalf}[\xhi]\xi\right)_{\om},
    \end{align*}
    since $\mathcal A_{\alpha,h}^{-\onehalf}[\xhi]f\in L^2(\om)$. Therefore, choosing $w=\mathcal A_{\alpha,h}^{-1}[\xhi]\xi$ it is also immediate to deduce \eqref{crucial}, since the operator $\mathcal A_{\alpha,h}[\xhi]$ is selfadjoint. To show \eqref{forbounds} we just consider $g\in H^1(\om)$ as the solution to 
    $$
    \alpha^2(\nabla g,\nabla \xi)_{\om}+\alpha^2(g,\xi)_{\om}+(\normmm{\rho_h\ast\nabla \xhi}g,\xi)_\om=(f,\xi)_\om,\quad  \forall \xi\in H^1(\om),
    $$
    given $f\in H^1(\om)$, which always exists by Lax-Milgram's lemma, meaning that $g=\Ah^{-1}[\chi]f$. Then, choosing $\xi=g$ and by Cauchy-Schwarz and Young's inequalities, we easily infer \eqref{forbounds} recalling that $\normmm{\rho_h\ast\nabla \xhi}\geq0$. In conclusion, to prove \eqref{ct1} it is enough to observe that, for any $f\in H^1(\om)'$, 
    \begin{align*}
        &\alpha^2\normh{\nabla\Ah^{-1}[\xhi]f}^2\\&\leq \alpha^2\normh{\nabla\Ah^{-1}[\xhi]f}^2+\alpha^2\normh{\Ah^{-1}[\xhi]f}^2+\int_\om \normmm{\rho_h\ast\nabla\xhi}\normmm{\Ah^{-1}[\xhi]f}^2\dx\\&=\langle\Ah^{-1}[\xhi]f,f\rangle_{H^1(\om),H^1(\om)'}=\normh{\Ah^{-\onehalf}[\xhi]f}^2.
    \end{align*}
\end{proof}
We then proceed with some further lemmas related to the discretization of the material derivative. Let us first introduce the flow map: given a $\mathbf C^1(\om)$  divergence-free vector field $\bv$, find $X_t^\bv:\om\to \om$ satisfying, for any $t\in \mathbb R$,
	\begin{align}
		\begin{cases}
			\ddt X_t^{\bv}=\bv(X_t^\bv),\\
			X_0^\bv=Id.
		\end{cases}\label{flowmap}
	\end{align}
	This problem clearly has a unique solution $X^\bv\in C^1((-\infty,+\infty),\mathbf C^1(\om))$. Moreover $X_t^\bv$ is a $C^1(\om)$-diffeomorphism for any $t\in \mathbb R$.
    Additionally, it holds (cf. \cite[(4.9)]{AbelsP})
    \begin{align}
			\norm{\nabla X_s^\bv}_{\mathbf C(\om)}\leq \sqrt d\ e^{s\norm{\bv}_{\mathbf C^1({\om})}},\quad \forall s\geq0.
			\label{Gronw}
		\end{align}
    We start from a suitable adaptation of \cite[Lemma 4.5]{AbelsP}.
\begin{lemma}
   	Let $\bv\in \mathbf C^1(\om)\cap \Hs$. Then it holds 
		\begin{align}
			\normh{\Ah^{-\onehalf}[\xhi](\xhi-\xhi\circ X_{-s}^\bv)}\leq  \max\{1,{\alpha}^{-1}\}s\norm{\bv}_{\mathbf C(\om)},\quad \forall \xhi\in \mathcal M,\quad \forall s\geq 0.
			\label{H1c}
		\end{align}
        Furthermore, it holds
        \begin{align}
	\normh{\Ah^{-\onehalf}[\xhi]\frac{\xhi-\xhi\circ X_{-s}^{\bv}}{s}}\le \max\{1,\alpha^{-1}\}\norm{\bv}_{\Ls},\quad \forall s>0.
	\label{bounddiff}
\end{align}
\end{lemma}
\begin{proof}
From the proof of  \cite[Lemma 4.5]{AbelsP} we know that, given $\phi\in H^1(\om)$, it holds that  
	\begin{align*}
		\normmm{\int_\om (\xhi-\xhi\circ X_{-s}^\bv)\phi\dx}\leq s\norm{\bv}_{\mathbf C({\om})}\norm{\phi}_{\Hz}.
		\end{align*} 
        Now we recall that from \eqref{identif} we have
\begin{align}
\norm{w}_{H^1(\om)}\leq \max\{1,{\alpha}^{-1}\}\normh{\mathcal A_{\alpha,h}^\onehalf[\xhi] w},\quad \forall w\in H^1(\om),\label{correction}
\end{align}
entailing 
	\begin{align*}
		\normmm{\int_\om (\xhi-\xhi\circ X_{-s}^\bv)\phi\dx}\leq  \max\{1,{\alpha}^{-1}\}s\norm{\bv}_{\mathbf C({\om})}\normh{\Ah^\onehalf[\xhi]\phi},
		\end{align*} 
which concludes the proof of \eqref{H1c} by means of \eqref{crucial}. Finally, to show \eqref{bounddiff}, notice that, arguing as in \cite[(4.29)]{AbelsP}, we infer, for any $s>0$, 
\begin{align*}
			&	\int_\om \frac{\xhi-\xhi\circ X_{-s}^{\bv}}{s}\varphi\dx
			=-\int_\om \nabla \varphi\cdot \bv\int_0^{s} \frac{\xhi\circ X_{-\tau}^{\bv}}{s}\d\tau\dx
		\end{align*}	
		for any $\varphi\in \Hz$, where of course 
    \begin{align*}
	\norm{\int_{0}^{s}	\frac{ \xhi\circ X_{-\tau}^{\bv}}{s}\d\tau}_{L^\infty(\om\times(0,\Ts))}\leq 1,
\end{align*} 
which gives, after an application of Cauchy Schwarz inequality,
\begin{align*}
    \normmm{\int_\om \frac{\xhi-\xhi\circ X_{-s}^{\bv}}{s}\varphi\dx}\leq \normh{\nabla\varphi}\norm{\bv}_{\Ls}\leq \max\{1,\alpha^{-1}\}\normh{\Ah^{\onehalf}[\xhi]\varphi}\norm{\bv}_{\Ls}, 
\end{align*}
entailing \eqref{bounddiff} recalling the representation \eqref{crucial}.
\end{proof}
Arguing as in the proof of \cite[Lemma 4.7]{AbelsP}, we also have the following 
\begin{lemma}
		Let $\ov\xhi\in \mathcal M$ and $\xhi_A,\xhi\in L^\infty(\om)$, with $\xhi_A,\xhi\in \{0,1\}$ almost eveywhere in $\om$, and $\bv\in \mathbf C^1({\om})\cap \Hs$. Then, for almost any $t\in \mathbb R$ it holds
		\begin{align}
			\frac12\ddt\normh{\Ah^{-\onehalf}[\ov\xhi](\xhi_A-\xhi\circ X_{-t}^\bv)}^2=-\left(\xhi_A\bv, \nabla ((\Ah^{-1}[\ov\xhi](\xhi_A-\xhi\circ X_{-t}^\bv))\circ X_t^\bv)\right)_{\om}.\label{cv1}
		\end{align}
		\label{convH1}
	\end{lemma}
\begin{proof} 
    We only show the result assuming $\xhi\in H^2(\om)$, as the case $\xhi\in L^\infty(\om)$ can be obtained by approximation arguing as in the proof of \cite[Lemma 4.7]{AbelsP}. We then have
\begin{align*}
			\onehalf\ddt\normh{\Ah^{-\onehalf}[\ov\xhi](\chi_A-\xhi\circ X_{-t}^\bv)}^2&=-(\pt(\xhi\circ X_{-t}^\bv), \Ah^{-1}[\ov\xhi](\chi_A-\xhi\circ X_{-t}^\bv))_\om
			\\&
			=((\nabla \xhi) \circ X_{-t}^\bv\cdot \bv\circ X_{-t}^\bv, \Ah^{-1}[\ov\xhi](\chi_A-\xhi\circ X_{-t}^\bv))_\om\\&
			= (\nabla \xhi\cdot \bv, (\Ah^{-1}[\ov\xhi](\chi_A-\xhi\circ X_{-t}^\bv)) \circ X_{t}^\bv)_\om\\&
			=-(\xhi \bv, \nabla ((\Ah^{-1}[\ov\xhi](\chi_A-\xhi\circ X_{-t}^\bv)) \circ X_{t}^\bv))_\om,
		\end{align*}
        concluding the proof.
\end{proof}
Finally, we recall the fundamental \cite[Lemma 4.6]{AbelsP}, which will be often used in the following.
	\begin{lemma}
    \label{base2}
		Let $T>0$, $\{h_n\}_{n\in \mathbb{N}}\subset(0,1)$, with $h_n\to0$ as $n\to\infty$, and assume that a sequence $\{\bv_n\}_n \subset C([0,T];\mathbf C^1(\om)\cap \Hs)$ is such that, for some $\gamma\in [0,1)$ and some constant $C>0$, 
			\begin{align}
				\norm{\bv_n}_{C([0,T];\mathbf C^1(\om))}\leq Ch_n^{-\gamma},\quad \forall n\in \mathbb N.
					\label{unifC1}
			\end{align}
		Consider the map $X_{s}^{\bv_n(t)}$, $t\in(0,T)$, $s\in[-h_n,h_n]$, which is a solution to the problem
		\begin{align}
			\begin{cases}\frac{\d}{\d s} X_{s}^{\bv_n(t)}=\bv_n(X_{s}^{\bv_n(t)}, t),\\
				X_0=Id.
			\end{cases}
			\label{map1}
		\end{align}
	
		Then, given any bounded sequence $\{w_n\}_n\subset L^2(0,T;\Hz)$, such that $w_n\rightharpoonup w$ weakly in $L^2(0,T;\Hz)$ as $n\to\infty$, it holds 
		\begin{align}
			w_n\circ X_{\cdot-\floor{\frac \cdot{h_n}}}^{\bv_n(\cdot)}\rightharpoonup  w,\quad\text{ weakly in }L^2(0,T;\Hz)\text{ as }n\to\infty,
			\label{w}
		\end{align}
		and 
		\begin{align}
			\nabla	(w_n\circ X_{\cdot-\floor{\frac \cdot{h_n}}}^{\bv_n(\cdot)})\rightharpoonup \nabla w,\quad\text{ weakly in }L^2(0,T;\mathbf L^2(\om))\text{ as }n\to \infty.
			\label{w1}
		\end{align}
		Moreover, if, additionally to \eqref{unifC1}, it holds
		\begin{align}
		\bv_n\to\bv\quad\text{ strongly in }L^2(0,T;\dot{\mathbf L}^2_\sigma(\om))\text{ as }n\to\infty,\qquad 	\norm{\bv}_{L^3(\om\times(0,T))}\leq C,
			\label{b1}
		\end{align}
		for some $C>0$, then for any sequence $\{\xhi_n\}_n\subset L^\infty(0,T;\mathcal M)$, such that there exists $C>0$ so that
		\begin{align}
			\esssup_{t\in(0,T)}\normmm{\nabla \xhi_n(t)}(\om)\leq C,\quad\forall n\in\N,
			\label{b2}
		\end{align}
		as well as $\xhi_n\to \xhi$ strongly in $L^1(\om\times(0,T))$ as $n\to\infty$, then  
		\begin{align}
			\frac{\xhi_n-\xhi_n\circ X_{-h_n}^{\bv_n(\cdot)}}{h_n}\to \bv\cdot \nabla \xhi\quad\text{ strongly
				in }L^2(0,T;\Hmz)\text{ as }n\to\infty,
			\label{ww1}
		\end{align}
		where $\bv\cdot \nabla\xhi\in H^1(\om)'$ means
		\begin{align*}
			\langle\bv\cdot \nabla\xhi, \varphi\rangle_{\Hmz,\Hz}=-(\bv\cdot \nabla \varphi, \xhi)_\om,
		\end{align*}
		for any $\varphi\in \Hz$.
		\label{convh1}
	\end{lemma}
\subsection{Approximating velocity for the advective Mean Curvature Flow}
Let us fix $h>0$. Analogously to \cite{AbelsP}, for $0<b$, given $\bv_h\in L^2(0,b;\dot\bL^2_\sigma(\om))$, we define $\bu^k_h:\om\times(0,b)\to \R^d$, $k\in \mathbb N$, as the solution to:
\begin{align}
	\frac1k\pt\bA^2 \bu^k_h+\frac1k\bA^2\bu^k_h+\bu^k_h=\bP_\sigma M_k(\bv_h),\quad \bu^k_h(0)=\mathbf 0,
	\label{bvk}
\end{align}
where, for any component $i=1,\ldots,d$, 
\begin{align}
	M_k(\bv_h)_i:=\begin{cases}
		k,\quad &\text{ if }\bv_{h,i}\geq k,\\
		\bv_{h,i},\quad&\text{ if }-k<\bv_{h,i}< k,\\
		-k,\quad &\text{ if }\bv_{h,i}\leq -k.
	\end{cases}
	\label{cutoff0}
\end{align}
Of course, the existence of a solution to \eqref{bvk} can be easily shown using a Galerkin scheme, whereas uniqueness  comes from linearity and standard energy estimates. We only point out the main estimate, which is (cf. \cite[(5.4)]{AbelsP})
\begin{align}
	\norm{\bu^k_h}_{C^\frac12([0,b];\mathbf C^1({\om}))}\leq C\norm{\bu^k_h}_{ H^1(0,b;\bH^4(\om)\cap \Hus)}\leq  C(b)k^2,
	\label{basic0}
\end{align}
where the constant $C(b)$ is \textit{independent} of $\bv_h$.
From now on we set $$k:=k(h)=h^{-\frac1{8}},$$ and thus from \eqref{basic0} we infer
\begin{align}
	\norm{\bu^k_h}_{C^\frac12([0,b];\mathbf C^1(\om))}\leq C\norm{\bu^k_h}_{ H^1(0,b;\bH^4(\om)\cap \Hus)}\leq  C(b)h^{-\frac14}.
	\label{basic}
\end{align}
To make a distinction between the time discretization scale $h$ and the $k$ scale, we will still distinguish between $h$ and $k$, having in mind that $k$ depends on $h$. Actually, in this proof one could also first let $h\to0$ and then $k\to\infty$, but we prefer to stick to the approach of \cite{AbelsP} choosing $k(h)$, to avoid the presence of too many parameters and independent limit procedures (we would have three independent parameters, namely $h,k,\alpha$).

We then set $\bu_n^k:=\bu^k_h(nh)$, for $n\in \N$, $nh\leq b$, together with $\bu_0^k:=\mathbf 0$, so that we have a sequence $\{\bu_n^k\}_n\subset\bH^4(\om)\cap \Hus\hookrightarrow \mathbf C^1(\om)$, for any fixed $h>0$. Note that the initial condition is here totally irrelevant, since we expect only convergence in lower-order spaces.
	\subsection{First discretized step}\label{firstminmovement}
	Set $h>0$ to be the time-step size. For $s\geq0$, we introduce 
	$$T_h(s):=s-(n-1)h,\quad s\in((n-1)h,nh],$$
as well as the following minimization problem, for $h>0$ and $\alpha>0$ ,
	\begin{align}
		&\chi^h_{0}=\chi_0,\\
		&\chi_{1}^{h}\in\arg\min_{\chi\in \mathcal M}\left(E[\chi]+\frac 1{2h}\normh{\Ah^{-\frac12}[\xhi_0^h](\chi-\chi_{0}^{h}\circ X_{-h}^{\bu_0^k})}^2\right),\label{M1}
	\end{align}
	with
	$$
	E[\xhi]=\normmm{\nabla \xhi}({\mathbb T^d}).
	$$
	We also set the De Giorgi interpolation:
	\begin{align}
		& \xhi^{h}_{DG}(0):=\xhi_{0}^h,\\&
		\xhi_{DG}^h(t)\in\arg\min_{\chi\in \mathcal M}\left(E[\chi]+\frac 1{2T_h(t)}\normh{\Ah^{-\frac12}[\xhi_0^h](\chi- \xhi^{h}_{0}\circ X_{-{T_h(t)}}^{\bu_0^k})}^2\right),
		\text{ for }t\in (0,h].\label{M2}
	\end{align}
	For any fixed $t$, the problem admits at least one solution by the Direct Method of the Calculus of Variations. The solution can be chosen to be strongly measurable with values in $L^p({\mathbb T^d})$, for any $p>1$. This can be shown as in \cite[Appendix B]{AbelsP}.

Let us now define the function 
 \begin{align}
 	f(t):=E[\xhi_{DG}^h(t)]+\frac{1}{2T_h(t)}\normh{\Ah^{-\frac12}[\xhi_0^h](\chi_{DG}^h(t)- \xhi^{h}_{0}\circ X_{-T_h(t)}^{\bu_0^k})}^2,
 	\label{ftt}
 \end{align}
 where $\xhi_{DG}^h(t)$ is defined in \eqref{M2}.
 By the same arguments as to obtain inequality \cite[(5.11)]{AbelsP}, using Lemma \ref{convH1} in place of \cite[Lemma 4.7]{AbelsP}, it can be shown that $f(t)$ is locally Lipschitz continuous on $(0,h]$ and, for almost any $t\in(0,h]$,
 \begin{align}
 &\nonumber\ddt f(t)\\&=-\frac{1}{2{\Th{t}}^2}\normh{\Ah^{-\frac12}[\xhi_0^h](\chi_{DG}^h(t)- \xhi^{h}_{0}\circ X_{-\Th{t}}^{\bu_0^k})}^2\nonumber\\&\quad -\frac{1}{\Th{t}}\left( \xhi_{0}^h\cdot \bu_0^k,\nabla \left((\Ah^{-1}[\xhi_0^h](\chi\tTh-\xhi_0^h\circ X_{-T_h(t)}^{\bu_0^k}))\circ X_{T_h(t)}^{\bu_0^k}\right)\right)_\om.\label{derivative}
 \end{align}
 Observe that, to show the analog of \cite[(5.13)]{AbelsP}, we need to use \eqref{forbounds} to obtain, recalling \eqref{Gronw},
\begin{align}
 			\nonumber	&\normmm{\left(\xhi_0^h\bu_0^k,\nabla\left( \left(\Ah^{-1}[\xhi_0^h]\left(\xhi_{DG}^h(s)-\chi_0^h\circ X_{-T_h(r)}^{\bu_0^k}\right)\right)\circ X_{T_h(r)}^{\bu_0^k}\right)\right)_\om}\\\nonumber&\leq \norm{\bu_0^k}_{\mathbf C(\om)}\normh{\nabla\left( \left(\Ah^{-1}[\xhi_0^h]\left(\xhi_{DG}^h(s)-\chi_0^h\circ X_{-T_h(r)}^{\bu_0^k}\right)\right)\circ X_{T_h(r)}^{\bu_0^k}\right)}\\&
 			\leq C(\alpha^{-1},h,\norm{\bu_0^k}_{\mathbf C^1(\om)})\norm{\xhi_{DG}^h(s)-\chi_0^h\circ X_{-T_h(r)}^{\bu_0^k}}_{H^1(\om)'}
            \leq C(\alpha^{-1},h,\norm{\bu_0^k}_{\mathbf C^1(\om)}),\label{Nm1b}
 		\end{align}
        for any $r,s\in[0,h)$.
Note now that, since $T_h(s)\to0$ as $s\to0$, thanks to \eqref{H1c} it holds  
\begin{align}
 	\frac{s}{2}\normh{\Ah^{-\onehalf}[\xhi_0^h]\left(\frac{\xhi_0^h-\xhi_0^h\circ X^{\bu_0^k}_{-\Th{s}}}{s}\right)}^2\to 0,\quad\text{ as }s\to 0.\label{limitA}
 \end{align}
 Therefore, we can argue as for inequality \cite[(5.24)]{AbelsP} to infer from \eqref{derivative}, after integrating over $(0,h)$, that
 \begin{align}
 	\nonumber& E[\xhi^h_{DG}(h)]+\frac1{2h}\normh{\Ah^{-\onehalf}[\xhi_0^h](\xhi_{DG}^{h}(h)-\xhi_0^h\circ X_{-\Th{h}}^{\bu_0^k})}^2\\&\nonumber+\int_0^h\frac{1}{2T_h(\tau)^2}\normh{\Ah^{-\onehalf}[\xhi_0^h](\chi_{DG}^{h}(\tau)-\chi_0^h\circ X_{-T_h(\tau)}^{\bu_0^k})}^2\d \tau\\& \leq E[\xhi_0]-\int_0^h\frac{1}{T_h(\tau)}\left( \xhi_{0}^h \bu_0^k, \nabla \left((\Ah^{-1}[\chi_0^h](\xhi_{DG}^h(\tau)-\xhi_0\circ X_{-T_h(\tau)}^{\bu_0^k}))\circ X_{T_h(\tau)}^{\bu_0^k}\right)\right)_\om\d \tau.\label{derivative3b}
 \end{align}
Now, arguing as in \eqref{Nm1b} and using \eqref{ct1}, for any $s\in[0,h)$ we infer  \begin{align}
 			\nonumber	&\normmm{\left((\xhi_0^h\bu_0^k,\nabla\left( \left(\Ah^{-1}[\xhi_0^h]\left(\xhi_{DG}^h(s)-\chi_0^h\circ X_{-T_h(s)}^{\bu_0^k}\right)\right)\circ X_{T_h(s)}^{\bu_0^k}\right)\right)_\om}\\\nonumber&\leq \nonumber
 			 C(h,\norm{\bu_0^k}_{\mathbf C^1(\om)})\normh{\left(\nabla \left(\Ah^{-1}[\xhi_0^h]\left(\xhi_{DG}^h(s)-\chi_0^h\circ X_{-T_h(s)}^{\bu_0^k}\right)\right)\circ X_{T_h(s)}^{\bu_0^k}\right)}\\&
 			\leq C(\alpha^{-1},h,\norm{\bu_0^k}_{\mathbf C^1(\om)})\normh{\Ah^{-\onehalf}[\xhi_0^h](\xhi_{DG}^h(s)-\chi_0^h\circ X_{-T_h(s)}^{\bu_0^k})},\label{Nm1}
 		\end{align}
so that, by Young's inequality, we have
\begin{align*}
    &\normmm{\int_0^h\frac{1}{T_h(\tau)}\left( \xhi_{0}^h \bu_0^k, \nabla \left((\Ah^{-1}[\chi_0^h](\xhi_{DG}^h(\tau)-\xhi_0\circ X_{-T_h(\tau)}^{\bu_0^k}))\circ X_{T_h(\tau)}^{\bu_0^k}\right)\right)_\om\d \tau}\\&\leq C(\alpha^{-1},h,\norm{\bu_0^k}_{\mathbf C^1(\om)})+\frac14\int_0^h\frac{1}{T_h(\tau)^2}\normh{\Ah^{-\onehalf}[\xhi_0^h](\xhi_{DG}^h(s)-\chi_0^h\circ X_{-T_h(s)}^{\bu_0^k})}^2\d\tau,
\end{align*}
which gives from \eqref{derivative3b} that
 \begin{align}
 	\nonumber& E[\xhi^h_{DG}(h)]+\frac1{2h}\normh{\Ah^{-\onehalf}[\xhi_0^h](\xhi_{DG}^{h}(h)-\xhi_0^h\circ X_{-\Th{h}}^{\bu_0^k})}^2\\&\nonumber+\int_0^h\frac{1}{4T_h(\tau)^2}\normh{\Ah^{-\onehalf}[\xhi_0^h](\chi_{DG}^{h}(\tau)-\chi_0^h\circ X_{-T_h(\tau)}^{\bu_0^k})}^2\d \tau\\& \leq E[\xhi_0]+C(\alpha^{-1},h,\norm{\bu_0^k}_{\mathbf C^1(\om)}).\label{derivative3bc}
 \end{align}
  \subsubsection{Approximating scheme for the fluid equation: interval $[0,h]$}
We construct the following two potentials:
 \begin{align*}
 	&u_h^1:=-\Ah^{-1}[\xhi_0^h]\frac{\xhi_1^{h}-\xhi_0\circ X_{-h}^{\bu_0^k}}{h},\\&
 	w_h(t):=-\Ah^{-1}[\xhi_0^h]\frac{\chi_{DG}^h(t)-\chi_0\circ X_{-T_h(t)}^{\bu_0^k}}{t},\quad t\in(0,h),
 \end{align*}
 so that \eqref{derivative3bc} reads
  \begin{align}
 	\nonumber& E[\xhi^h_{DG}(h)]+\frac1{2}\normh{\Ah^\onehalf[\xhi_0^h] u_h^1}^2+\int_0^h\frac{1}{4}\normh{\Ah^{\onehalf}[\xhi_0^h]w_h(\tau)}^2\d \tau\\& \leq E[\xhi_0]+C(\alpha^{-1},h,\norm{\bu_0^k}_{\mathbf C^1(\om)}),\label{derbc}
 \end{align}
 where we recall \eqref{identif}. We can now introduce the approximation scheme for the fluid velocity $\bv_h$. As in \cite{AbelsP} we consider a time-continuous scheme, adding a further higher-order regularization term. The equation reads, in strong formulation,
 \begin{align*}
     \partial_t \bv_h+\bP_\sigma(\bv_h\cdot\nabla)\bv_h-\bP_\sigma\Div(\nu(\xhi_0^h)D\bv_h)+\frac1k\bA^2\bv_h=-\bP_\sigma(\nabla w_h\xhi_0^h), \quad t\in [0,h],
 \end{align*}
 with $\bv_h(0)=\bv_0$.
Of course the problem admits a unique solution, thanks to the regularization term $\bA^2$. Existence can be proven by a standard Galerkin approximation scheme. Note that this is possible since, thanks to \eqref{derbc}, it holds 
$$
\norm{w_h}_{L^2(0,h;H^1(\om))}\leq C,
$$
for some $C=C(h,\alpha^{-1},k)>0$. In particular, we can find the existence of a unique solution satisfying (cf. estimate \eqref{timder} for the time derivative estimate)
\begin{align}
\bv_h\in H^1(0,h;\dot\bH^2_\sigma(\om)')\cap L^2(0,h;\bH^2(\om))\cap L^\infty(0,h;\dot{\bL}^2_\sigma(\om)).
    \label{vregularity}
\end{align}
The energy estimate, which can be obtained by testing with $\bv_h$ passing to the limit in a further suitable approximating scheme, then reads 
\begin{align}
 	&\onehalf\int_\om \normmm{\bv_h(t)}^2\dx+\frac1k\int_s^t \normh{\bA\bv_h}^2\d \tau+\int_s^t\normh{\sqrt{\nu(\xhi_0^h)}D \bv_h}^2\d \tau\nonumber\\&
 	\leq 
 	\onehalf\int_\om \normmm{\bv_h(s)}^2\dx-\int_s^t\int_\om \nabla w_h(\tau)\xhi_0^h\cdot \bv_h(\tau) dx\d \tau,\label{energy}
 \end{align}
 for almost any $0\leq s\leq t<h$, with $s=0$ included.

\subsection{Iterative scheme at general step $n>1$}
Let us assume to know $\xhi_{DG}^h(t)$ up to step $n-1$, i.e., we assume $\xhi_{DG}^h(t)$ to be defined on $[0,(n-1)h]$, $n\in\N$, $n>1$. Let also $\xhi^h_{m-1}$ be given for any $0\leq m\leq n$, as well as assume to know $\bv_h$ to be defined on $[0,(n-1)h]$. Then we compute $\bu^k_h$, defined on $[0,(n-1)h]$, from \eqref{bvk}, and consider the sequence $\{\bu_m^k\}_{m\in\N,\ m\leq {n-1}}$, where $\bu^k_m:=\bu^k_h(mh)$, for $m\in \N$, $m\leq n-1$. We can now define the minimizing step $n$. 
\subsubsection{Minimizing movements scheme}
We set the following minimizing sequence:
\begin{align}
	&\chi_{n}^{h}\in\arg\min_{\chi\in \mathcal M}\left(E[\chi]+\frac 1{2h}\normh{\Ah^{-\frac12}[\xhi_{n-1}^h](\chi- \xhi^{h}_{n-1}\circ X_{-h}^{\bu_{n-1}^k})}^2\right),\label{M1b}
\end{align}

Then, we define the De Giorgi interpolant as:
\begin{align}
	& \xhi^h_{DG}((n-1)h):=\xhi_{n-1}^h,\\&
	\xhi_{DG}^h(t)\in\arg\min_{\chi\in \mathcal M}\left(E[\chi]+\frac 1{2(t-(n-1)h)}\normh{\Ah^{-\frac12}[\xhi_{n-1}^h](\chi- \xhi^{h}_{n-1}\circ X_{-T_h(t)}^{\bu_{n-1}^k})}^2\right),\nonumber\\&
	\text{ for }t\in ((n-1)h,nh].\label{M2b}
\end{align}
Again, for any fixed $t$, the problem admits at least one strongly measurable solution. Also, \eqref{derivative} is still valid up to naturally substituting $\xhi_0^h$ with $\xhi_{m}^h$, with $m=\floor{\frac th}$, for almost any $0\leq t< nh$. We then define, for any $1\leq m\leq n$, 
\begin{align}
&\label{uh1}u_h^m:=-\Ah^{-1}[\xhi_{m-1}^h]\frac1{h}\left(\xhi_m^{h}-\xhi_{m-1}^h\circ X_{-h}^{\bu_{m-1}^k}\right),\\&
w_h(t):=-\Ah^{-1}[\xhi_{m-1}^h]\frac{1}{\Th{t}}\left(\chi_{DG}^{h}(t)-\xhi_{m-1}^h\circ X_{-T_h(t)}^{\bu_{m-1}^k}\right),\quad t\in((m-1)h,mh).\label{wh}
\end{align}
We also set the piecewise constant interpolants as 
\begin{align*}
	\xhi^h(t):=\xhi_{m-1}^h,\quad  u_h(t):=u_h^{m-1}\quad t\in[(m-1)h,mh),\quad 1\leq m\leq n,
\end{align*}
and the affine interpolant

\begin{align*}
\widehat\xhi^h(t):=\frac{t-(m-1)h}{h}\xhi_m^h+\frac{mh-t}{h}\chi^h_{m-1},\quad t\in[(m-1)h,mh], \quad 1\leq m\leq n.
\end{align*}
We also define the piecewise constant interpolant of $\bu_n^k$ as
$$
\widehat{\bu}_h^k(t):=\bu_h^k((m-1)h) =\bu_k^{m-1},\quad t\in[(m-1)h,mh), \quad 1\leq m\leq n.
$$
Finally, the approximated material derivatives are written as
\begin{align}
	\partial^{\bullet,h}_t{\xhi}^h(t):=\frac {\xhi_{m}^{h}-\xhi_{m-1}^h\circ X_{-h}^{\widehat \bu_h^k(t)}}{h},\quad t\in[(m-1)h,mh), \quad 1\leq m\leq n.\label{material}
\end{align}
Then, upon noticing that inequality \eqref{derbc} holds, \textit{mutatis mutandis}, on any interval of the form $[(m-1)h,mh]$, $1\leq m\leq n$, 
we can extend the velocity $\bv_h$ so that it solves (in weak formulation), on $[0,nh]$, 
\begin{align}
     \partial_t \bv_h+\bP_\sigma(\bv_h\cdot\nabla)\bv_h-\bP_\sigma\Div(\nu(\xhi^h)D\bv_h)+\frac1k\bA^2\bv_h=-\bP_\sigma(\nabla w_h\xhi^h), \quad t\in [0,nh],\label{dtv}
 \end{align}
 with $\bv_h(0)=\bv_0$. Note that this extension is possible as the solution to the problem above exists unique. This gives (cf. estimate \eqref{timder} below for the time derivative estimate) 
 \begin{align}
\bv_h\in H^1(0,nh;\dot\bH^2_\sigma(\om)')\cap L^2(0,nh;\dot\bH_\sigma^2(\om))\cap L^\infty(0,nh;\dot{\bL}^2_\sigma(\om)),
    \label{vregularityglobal}
\end{align}
and $\bv_h$ satisfies the global energy inequality 
\begin{align}
 	&\onehalf\int_\om \normmm{\bv_h(t)}^2\dx+\frac1k\int_s^t \normh{\bA\bv_h}^2\d \tau+\int_s^t\normh{\sqrt{\nu(\xhi^h)}D \bv_h}^2\d \tau\nonumber\\&
 	\leq 
 	\onehalf\int_\om \normmm{\bv_h(s)}^2\dx-\int_s^t\int_\om \nabla w_h(\tau)\xhi^h\cdot \bv_h(\tau) dx\d \tau,\label{enerb}
 \end{align}
 for almost any $0\leq s\leq t<nh$, with $s=0$ included. 
\subsection{Global energy estimates}
We now repeat iteratively the procedure of the previous section for any $n\in\mathbb N$. As observed, on any interval of the form $((n-1)h,nh)$, $n\in\N$, we have that \eqref{derivative}  holds, namely 
\begin{align}
	&	 E[\xhi_{n}^{h}]+\frac1{2h}\normh{\Ah^\onehalf[\xhi_{n-1}^h](\chi_n^h- \xhi^{h}_{n-1}\circ X_{-h}^{\bu_{n-1}^k})}^2\nonumber\\&\label{energy1a}+\int_{(n-1)h}^{nh}\frac{1}{2{T_h(\tau)}^2}\normh{\Ah^{-\frac12}[\xhi_{n-1}^h](\chi_{DG}^{h}(\tau)-\xhi^{h}_{n-1}\circ X_{-T_h(\tau)}^{\bu_{n-1}^k})}^2\d \tau\\& \nonumber+\int_{(n-1)h}^{nh}\frac{1}{T_h(\tau)}\left( \xhi_{n-1}^{h} \bu_{n-1}^k,\nabla \left(\left(\Ah^{-1}[\xhi_{n-1}^h]\left(\chi_{DG}^h({\tau})-\xhi_{n-1}^h\circ X_{-T_h(\tau)}^{ \bu_{n-1}^k}\right)\right)\circ X_{T_h(\tau)}^{\bu^k_{n-1}}\right)\right)_\om\d \tau\\&\nonumber\leq E[\xhi_{n-1}^h].
\end{align}
We then define the quantity
\begin{align*}
	g_h(t):=t-\floor{\frac{t}h}h,
\end{align*}
so that, from \eqref{energy1a}, valid for any $n\in\mathbb N$, we infer by telescoping that, for any $n,m\in \mathbb N$, $m<n$, 
\begin{align}
	\nonumber	&E[\xhi_{n}^{h}]+\frac1{2}\int_{mh}^{nh}\normh{	\Ah^{-\onehalf}[\xhi^h(\tau)]\partial^{\bullet,h}_t{\xhi}^h(\tau)}^2\d\tau\nonumber\\&+\frac12\int_{mh}^{nh}\frac{1}{{g_h(\tau)}^2}\normh{\Ah^{-\frac12}[\xhi^h(\tau)](\chi_{DG}^{h}(\tau)-\xhi^h(\tau)\circ X^{\widehat\bu_h^k(\tau)}_{-g_h(\tau)})}^2\d \tau\nonumber
    \nonumber\\&+\int_{mh}^{nh}\frac{1}{g_h(\tau)}\left( \xhi^{h} (\tau)\widehat\bu_{h}^k(\tau),\nabla \left(\left(\Ah^{-1}[\xhi_{n-1}^h]\left(\chi_{DG}^h(\tau)-\xhi^h(\tau)\circ X_{-g_h(\tau)}^{ \widehat\bu_{h}^k(\tau)}\right)\right)\circ X_{-g_h(\tau)}^{ \widehat\bu_{h}^k(\tau)}\right)\right)_\om\d \tau\nonumber
    \\& \leq E[\xhi_{m}^h].
	\label{energy1ab1}
\end{align}
Recalling the definitions of the quantities $w_h$ and $u_h$, the same inequality can be rewritten as 
\begin{align}
	\nonumber	&E[\xhi_{n}^{h}]+\frac1{2}\int_{mh}^{nh}\normh{	\Ah^{\frac12}[\xhi^h(\tau)]u_h(\tau)}^2\d\tau\nonumber\nonumber+\frac12\int_{mh}^{nh}\normh{\Ah^{\frac12}[\xhi^h(\tau)]w_h(\tau)}^2\d \tau\nonumber\\&\nonumber
    -\int_{mh}^{nh}\frac{1}{g_h(\tau)}\left( \xhi^{h} (\tau)\widehat\bu_{h}^k(\tau),\nabla \left(w_h(\tau)\circ X_{-g_h(\tau)}^{ \widehat\bu_{h}^k(\tau)}\right)\right)_\om\d \tau
    \\& \leq E[\xhi_{m}^h].
	\label{energy1ab}
\end{align}
 Let us fix $T^*>0$, and $0<s<\kappa<\tau<T<T^*$. We can take $h<\min\{T-\tau,\kappa-s\}$, so that we can choose $n_0:=\floor{\frac{T}h}$ and $m_0:=\ceil{\frac sh}$, giving
$$
0<s\leq m_0h<\kappa<\tau<n_0h\leq T,\quad T\in[n_0h,(n_0+1)h),\quad s\in((m_0-1)h,m_0h].
$$
Arguing then as for inequality \cite[(5.56)]{AbelsP}, we infer 
\begin{align}
	\nonumber	&E[\xhi_{DG}^{h}(T)]+\frac1{2}\int_{\kappa}^{\tau}\normh{\Ah^{\onehalf}[\xhi^h(t)]	u_h(t)}^2\dt\nonumber\quad +\onehalf\int_{\kappa}^{\tau}\normh{\Ah^{\onehalf}[\xhi^h(t)]w_h(t)}^2\d t\\& \leq E[\xhi_{DG}^{h}(s)]+\frac1{2T_h(s)}{\normh{\Ah^{-\onehalf}[\xhi^h(s)](\xhi_{DG}^{h}(s)-\xhi^h(s)\circ X_{-T_h(s)}^{\widehat{\bu}_h^k(s)})}^2}\nonumber\\&\nonumber\quad +\frac{1}{2T_h(T)}\normh{\Ah^{-\onehalf}[\xhi^h(T)](\xhi^h(T)-\xhi^h(T)\circ  X_{-T_h(T)}^{\widehat{\bu}_h^k(T)}) }^2\nonumber\\&\quad +\int_{s} ^{\floor{\frac Th}h}\frac{1}{g_h(t)}\left( \xhi^h{(t)} \widehat{\bu}_h^k(t),\nabla \left(w_h(t)\circ X_{g_h(t)}^{\widehat{\bu}_h^k(t)}\right)\right)_\om\d t,
	\label{eA}
\end{align}
where we note that, thanks to \eqref{H1c} and \eqref{basic}, recalling that $\normmm{T_h(t)}\leq h$, it holds, for any $t\in(0,\Ts)$,
\begin{align}
    \frac{1}{2T_h(t)}\normh{\Ah^{-\onehalf}[\xhi^h(t)](\xhi^h(t)-\xhi^h(t)\circ  X_{-T_h(t)}^{\widehat{\bu}_h^k(t)}) }^2\leq CT_h(t)\norm{\widehat\bu_h^k}_{C(\om)}^2\leq C(\Ts)h^\frac12,\label{ineqgh}
\end{align}
for any $\Ts>0$ (since when $t\to (nh)^+$ for some $n\in\N$, $T_h(t)\to0$, we tacitly consider the above quantity to be zero at $t=nh$). 
We now introduce the varifolds associated to the varifold lift of the De Giorgi interpolant $\xhi^{h}_{DG}(t)$ in ${\mathbb T^d}$ as 
\begin{align*}
	\widetilde\mu_t^{h}:=\normmm{\nabla\xhi^{h}_{DG}(t)}\otimes \left(\delta_{\frac{\nabla \xhi^{h}_{DG}(t)}{\normmm{\nabla\xhi^{h}_{DG}(t)}}(x)}\right)_{x\in {\mathbb T^d}}\in \mathcal M({{\mathbb T^d}}\times \mathbb S^{d-1}),
\end{align*}
 so that for any $t\in(0,T^*)$ it holds $E[\xhi_{DG}^h(t)]=\normmm{\mu_t^h}_{\mathbb S^{d-1}}({\mathbb T^d})$. In the same way we consider the varifold lift of $\xhi^{h}(t)$ in ${\mathbb T^d}$ as 
\begin{align*}
	\mu_t^h:=\normmm{\nabla\xhi^{h}(t)}\otimes \left(\delta_{\frac{\nabla \xhi^{h}(t)}{\normmm{\nabla\xhi^{h}(t)}}(x)}\right)_{x\in {\mathbb T^d}}\in \mathcal M({{\mathbb T^d}}\times \mathbb S^{d-1}),
\end{align*}

so that for any $t\in(0,T^*)$ it holds $E[\xhi^h(t)]=\normmm{\mu_t^h}_{\mathbb S^{d-1}}({\mathbb T^d})$.

\subsection{Uniform estimates in $h$}
Arguing as in \cite[Sections 5.5.1-5.5.2]{AbelsP}, we can deduce that 
\begin{align}
\int_0^t\norm{\widehat{\bu}_h^k(s)}^2\d s\leq C(\Ts)h^{\frac12}+2C(\Ts)h^{\frac14}\int_0^t\norm{\bv_h(s)}\d s+\int_0^t\norm{\bv_h(s)}^2\d s,
	\label{fund1}
\end{align}
for any $t\in(0,\Ts)$ and any $\Ts>0$, as well as
\begin{align}
&\normmm{\int_{0}^{T}\left( \xhi^h{(t)} \widehat{\bu}_h^k(t),\nabla \left(w_h(t)\circ X_{g_h(t)}^{\widehat{\bu}_h^k(t)}\right)\right)_\om\d t}\nonumber\\&
	\leq 	\frac14\int_0^T\norm{\nabla w_h}^2\dt+C(\Ts)\int_0^T\norm{\bv_h(t)}^2\dt+C(\Ts),\label{integrability}
\end{align}
 so that we can conclude from \eqref{energy1ab}, also recalling \eqref{ineqgh}, that 
\begin{align}
	\nonumber	&E[\xhi_{DG}^{h}(T)]+\frac1{2}\int_{0}^{T}\normh{\Ah^{\onehalf}[\xhi^h(t)]	u_h(t)}^2\dt\nonumber+\frac14\int_{0}^{T}\normh{\Ah^{\onehalf}[\xhi^h(t)]w_h(t)}^2\d t\\& \leq E[\xhi_0]+C(\Ts)\int_0^T\norm{\bv_h(t)}^2\dt+C(\Ts),
	\label{eA1}
\end{align}
for any $0<T<\Ts$ and any $\Ts>0$. Proceeding in the estimates, by means of the Cauchy-Schwarz, Young and Korn inequalities (recall that $\int_\om \bv_h\dx=0$), and since $\nu\geq \nu_*>0$, we obtain from \eqref{enerb} that  
\begin{align}
 	&\onehalf\int_\om \normmm{\bv_h(T)}^2\dx+\frac1k\int_0^T \normh{\bA\bv_h}^2\d \tau+\int_0^T\normh{\sqrt{\nu(\xhi^h)}D \bv_h}^2\d \tau\nonumber\\&
 	\leq 
 	\onehalf\int_\om \normmm{\bv_0}^2\dx+\frac14\int_0^T\normh{\sqrt{\nu(\xhi^h)}D \bv_h}^2\d \tau+C\int_0^T\normh{\nabla w_h(\tau)}^2\d\tau,\label{enerbA1}
 \end{align}
for any $T>0$. Therefore, since $\normh{\nabla f}^2\leq \frac1{\alpha^2}\normh{\Ah^{\frac12}[\xhi]f}^2$ for any $\xhi\in \mathcal M$ (cf. \eqref{identif}), we infer
\begin{align}
 	&\onehalf\int_\om \normmm{\bv_h(T)}^2\dx+\frac1k\int_0^T \normh{\bA\bv_h}^2\d \tau+\int_0^T\normh{\sqrt{\nu(\xhi^h)}D \bv_h}^2\d \tau\nonumber\\&
 	\leq 
 	\onehalf\int_\om \normmm{\bv_0}^2\dx+\frac{C_0}{\alpha^2}\int_0^T\normh{\Ah^{\onehalf}[\xhi^h(\tau)] w_h(\tau)}^2\d\tau,\label{enerbA}
 \end{align}
for some $C_0>0$, so that we can sum this inequality, multiplied by $\omega=\frac{\alpha^2}{8C_0}$, with \eqref{eA1}, obtaining
\begin{align}
	\nonumber	&\frac\omega2\int_\om \normmm{\bv_h(T)}^2\dx+E[\xhi_{DG}^{h}(T)]\\&+\frac1{2}\int_{0}^{T}\normh{\Ah^{\onehalf}[\xhi^h(t)]	u_h(t)}^2\dt\nonumber +\frac18\int_{0}^{T}\normh{\Ah^{\onehalf}[\xhi^h(t)]w_h(t)}^2\d t
    \nonumber\\&\nonumber+\frac\omega k\int_0^T \normh{\bA\bv_h}^2\d \tau+\omega\int_0^T\normh{\sqrt{\nu(\xhi^h)}D \bv_h}^2\d \tau\\& \leq \frac\omega2\int_\om \normmm{\bv_0}^2\dx+E[\xhi_0]+C(\Ts)\int_0^T\norm{\bv_h(t)}^2\dt+C(\Ts),
	\label{eA0}
\end{align}
 We thus obtain, by the integral Gronwall's lemma, the following uniform estimates:
 
 \begin{align}
 	\esssup_{t\in(0,\Ts)}\norm{\bv(t)}_{\dot\bL_\sigma^2(\om)}+\esssup_{t\in (0,\Ts)}E[\xhi_{DG}^h(t)]\leq C(\Ts),
 	\label{dissiprel}
 \end{align}
 for any $\Ts>0$.
Also, we have, by Korn's inequality,
 \begin{align}
 	&	\nonumber\norm{\Ah^\onehalf[\xhi^h]u _h}_{L^2(0,\Ts;L^2({\mathbb T^d}))}+\norm{\Ah^\onehalf[\xhi^h]w_h}_{L^2(0,\Ts;L^2({\mathbb T^d}))}\\&+\norm{\bv_h}_{L^2(0,\Ts;\dot\bH^1_\sigma(\om))}+\frac1{\sqrt k}\norm{\bA\bv_h}_{L^2(0,\Ts;\bL^2(\om))}\leq C(\Ts),\label{regvbisA}
 \end{align}
 for any $\Ts>0$. This means by     \eqref{identif} that we have
 \begin{align}
 \alpha\norm{u_h}_{L^2(0,\Ts;H^1(\om))}+\alpha\norm{w_h}_{L^2(0,\Ts;H^1(\om))}\leq C(\Ts),
     \label{regvbis}
 \end{align}
 as well as
\begin{align}
    \int_0^\Ts\int_{\mathbb T^d} u_h(t)^2\normmm{\rho_h\ast{\nabla \xhi^h(t)}}\dx\dt+\int_0^\Ts\int_{\mathbb T^d} w_h(t)^2\normmm{\rho_h\ast{\nabla \xhi^h(t)}}\dx\dt\leq C(\Ts),\label{hutchinson}
\end{align}
for any $\Ts>0$. Concerning time derivatives, it holds
\begin{align*}
   \partial_{t}^{\bullet,h}\xhi^h(t)= \partial_t\widehat \xhi^h(t)+\frac{\xhi^h(t)-\xhi^h(t)\circ X_{-h}^{\widehat \bu^k_h(t)}}{h},
\end{align*}
so that, recalling \eqref{uh1} and the definition \eqref{material} of material derivative, we can resort to \eqref{bounddiff}, to infer that 
\begin{align*}
   \normh{\Ah^{-\onehalf}[\xhi^h]\partial_t\widehat \xhi^h}&\leq \normh{\Ah^{-\onehalf}[\xhi^h]\partial_{t}^{\bullet,h}\xhi^h}+\normh{\Ah^{-\onehalf}[\xhi^h]\frac{\xhi^h(t)-\xhi^h(t)\circ X_{-h}^{\widehat \bu^k_h(t)}}{h}}\\&
   \leq \normh{\Ah^{\onehalf}[\xhi^h]u_h}+\max\{1,\alpha^{-1}\}\norm{\widehat\bu_h^k(t)}_{\Ls},
\end{align*}
which gives, recalling \eqref{dissiprel} and \eqref{regvbisA},
\begin{align}
  \norm{\Ah^{-\onehalf}[\xhi^h]\partial_t\widehat \xhi^h}_{L^2(0,\Ts;L^2(\om))}\leq C(\Ts),
    \label{timeder}
\end{align}
for any $\Ts>0$. Applying then Lemma \ref{Lemmaequiv1}, we finally deduce the control
\begin{align}
   & \norm{\pt\widehat \xhi^h}_{L^2(0,\Ts;(C(\om)\cap H^1(\om))')}\\&\leq \sqrt3\sup_{t\in(0,\Ts)}\sqrt{\max\{\alpha^2,\normmm{\nabla\xhi^h(t)}(\om)\}}\norm{\Ah^{-\onehalf}[\xhi^h]\partial_t\widehat \xhi^h}_{L^2(0,\Ts;L^2(\om))}\leq C(\Ts),\label{unifa}
\end{align}
 thanks to the uniform estimate \eqref{dissiprel} (recall that by construction $\xhi^h(t)=\xhi_{DG}^h(\floor{\tfrac th}h)$ for any $t\geq0$). Similarly, we have from \eqref{bounddiff} and \eqref{wh} that
 \begin{align*}
     &\normh{\Ah^{-1}[\xhi^h]\frac{1}{\Th{t}}\left(\chi_{DG}^{h}(t)-\xhi^h(t)\right)}\\&\leq \normh{\Ah^{-1}[\xhi^h]\frac{1}{\Th{t}}\left(\chi_{DG}^{h}(t)-\xhi^h(t)\circ X_{-T_h(t)}^{\bu_{m-1}^k}\right)}+\normh{\Ah^{-\onehalf}[\xhi^h]\frac{\xhi^h(t)-\xhi^h(t)\circ X_{-T_h(t)}^{\widehat \bu^k_h(t)}}{T_h(t)}}\\&
   \leq \normh{\Ah^{\onehalf}[\xhi^h]u_h}+\max\{1,\alpha^{-1}\}\norm{\widehat\bu_h^k(t)}_{\Ls},
 \end{align*}
giving, again using Lemma \ref{Lemmaequiv1} and estimate \eqref{dissiprel},
\begin{align}
&\norm{\Ah^{-1}[\xhi^h]\frac{1}{\Th{t}}\left(\chi_{DG}^{h}(t)-\xhi^h(t)\right)}_{L^2(0,\Ts;(C(\om)\cap H^1(\om))')}\nonumber\\& \leq C(\Ts) \norm{\Ah^{-1}[\xhi^h]\frac{1}{\Th{t}}\left(\chi_{DG}^{h}(t)-\xhi^h(t)\right)}_{L^2(0,\Ts;L^2(\om))}\leq C(\Ts).
    \label{timeder1}
\end{align}
About the time derivative of $\bv_h$, by comparison we easily obtain, by \eqref{dtv} and recalling the bounds \eqref{dissiprel}-\eqref{regvbis},  
\begin{align*}
    \normmm{\langle\pt\bv_h,\bzeta\rangle}\leq C(\norm{\bv_h}_{\bL^2(\om)}\norm{\nabla  \bv_h}_{\bL^2(\om)}\norm{\bzeta}_{\bL^\infty(\om)}+\frac1 k\norm{\mathbf A\bv_h}_{\bL^2(\om)}\norm{\bzeta}_{\bH^2(\om)}+\norm{\nabla w_h}_{\bL^2(\om)}\norm{\bzeta}_{\bL^2(\om)}),
\end{align*}
for any $\bzeta\in \dot\bH^2_\sigma(\om)$, entailing
\begin{align}
\nonumber\norm{\pt\bv_h}_{L^2(0,\Ts;\dot\bH_\sigma^2(\om)')}&\leq C(\Ts)(\norm{\bv_h}_{L^2(0,\Ts;\bH^1(\om))}+\frac1k\norm{\bA\bv_h}_{L^2(0,\Ts;\bL^2(\om))}+\norm{\nabla w_h}_{L^2(0,\Ts;\bL^2(\om))})\\&\leq C(\Ts),
    \label{timder}
\end{align}
for any $\Ts>0$, where we used the embedding $\bH^2(\om)\hookrightarrow \bL^\infty(\om)$. Note also that $u_h$ satisfies, for any $\zeta\in C_c^\infty([0,\infty)\times \om)$,
\begin{align}
&\nonumber\int_0^{\infty}\int_{\mathbb T^d}\partial_t^{\bullet,h}\xhi^h(t)\zeta(t)\dx\dt\label{timederiv}\\&\nonumber=
-\int_0^\infty\int_{\mathbb T^d} u_h(t)\zeta(t)\normmm{\rho_h\ast{\nabla \xhi^h(t)}}\dx\dt-\alpha^2\int_0^\infty\int_{\mathbb T^d} u_h(t)\zeta(t)\dx\dt\\&\quad  -\alpha^2\int_0^{\infty}\int_{\mathbb T^d} \nabla u_h(t)\cdot \nabla\zeta(t)\dx\dt.
\end{align}

Additionally, for all $t\in(0,T^*)$, by the minimizing property \eqref{M2} of $\xhi^{h}_{DG}(t)$, Allard’s first
variation formula \cite{Allard}, and \cite[Lemma 10]{SH}, we immediately infer that $\xhi_{DG}^h(t)$ satisfies the approximate Gibbs-Thomson relation:
\begin{align}
	\int_{{\mathbb T^d}\times \mathbb S^{d-1}}(Id-s\otimes s):\nabla B(x)\, \d \widetilde\mu_t^h(x,s)=\int_{\mathbb T^d} \xhi_{DG}^h(t)\text{div}(w_h(\cdot,t)B)\,\d x\label{allard}
\end{align}
for all $t\in(0,T^*)$ and all $B\in C^1(\om;\R^d)$. Analogously, the same relation is satisfied by $\xhi^h(t)$:
\begin{align}
	\int_{{\mathbb T^d}\times \mathbb S^{d-1}}(Id-s\otimes s):\nabla B(x)\,\d \mu_t^h(x,s)=\int_{\mathbb T^d} \xhi^h(t)\text{div}(u_h(\cdot,t)B)\,\d x\label{allard2}
\end{align}
for all $t\in(0,\infty)$ and all $B\in C^1(\om;\R^d)$. 

We now introduce the function $\tF_h:(0,\infty)\to \R$:
\begin{align}
\nonumber	\tF_h(t):=&E[\xhi^h_{DG}(t)]+\frac{1}{2T_h(t)}\normh{\Ah^{-\onehalf}[\xhi^h(t)](\xhi^h_{DG}(t)-\xhi^h(t)\circ X_{-T_h(t)}^{\widehat{\bu}_h^k(t)})}^2\\&- \int_{0}^{t}\left( \xhi^h{(\tau)} \widehat{\bu}_h^k(\tau),\nabla \left(w_h(\tau)\circ X_{g_h(\tau)}^{\widehat{\bu}_h^k(\tau)}\right)\right)_\om\d \tau,\label{F}
\end{align}
where we recall that the second summand is to be assumed to be zero when $t=nh$, $n\in\N$ (it indeed converges to zero as $t\to nh$, cf. \eqref{ineqgh}).
Then, by the same argument as to obtain \cite[(5.90)]{AbelsP} and \cite[(5.94)]{AbelsP}, which relies on \eqref{derivative}, we can infer that
\begin{align}
   \nonumber& \tF_h \text{ is monotone nonincreasing on $t\geq0$},\\&
   \sup_{t\geq0}\normmm{\tF_h(t)}\leq C(\Ts),\quad \forall \Ts>0.\label{Fhb}
\end{align}
We additionally define $\tF_h:(0,\infty)\to \R$: 
$$
F_h(t):=\tF_h(\floor{\tfrac th}h),\quad \forall t\geq0,
$$
which is of course monotone nonincreasing and uniformly bounded as well.
As a fundamental consequence of these results, we see that 
\begin{align}
    \tF_h(t)\leq F_h(t)\leq \tF_h(s),\quad \forall\ 0\leq t,\ s<t-h,\label{chain}
\end{align}
since $s<\floor{\tfrac th}h$ for any $s<t-h$, where we extended $\tF_h(t)=E[\xhi_0]$ for $t\leq 0$.
\subsection{Limit as $h\to0$}
\subsubsection{Compactness as $h\to0$}\label{compact}
From \eqref{dissiprel}-\eqref{regvbis} we deduce that there exist $ \tw\in L^2_{loc}([0,\infty);H^1({\mathbb T^d}))$, $u\in L^2_{loc}([0,\infty);H^1({\mathbb T^d}))$, and $\bv\in L^\infty_{loc}([0,\infty);\Ls)\cap L^2_{loc}([0,\infty);\dot\bH^2_\sigma(\om))\cap H^1_{loc}([0,\infty);\bH_\sigma^2(\om)')$ (thus globally defined on $(0,\infty)$), such that, also applying Aubin-Lions theorem, up to subsequences, as $h\to0$, 
\begin{align}
\label{whconv}	w_h\rightharpoonup  \tw &\quad\text{weakly in }L^2(0,\Ts;\Hz),\\
	u_h\rightharpoonup u &\quad\text{weakly in }L^2(0,\Ts;\Hz),\label{uhA}\\
    \bv_h\rightharpoonup \bv&\quad\text{weakly* in }L^\infty(0,\Ts;\Ls),\\
    \bv_h\rightharpoonup \bv&\quad\text{weakly in }L^2(0,\Ts;\dot\bH^1_\sigma(\om)),\label{weakconv1}\\
     \pt\bv_h\rightharpoonup \bv&\quad\text{weakly in }L^2(0,\Ts;\dot\bH^2_\sigma(\om)'),\\
     \bv_h\to \bv&\quad\text{strongly in }L^2(0,\Ts;\bL^2_\sigma(\om)),\label{str1}
\end{align}
for any $\Ts>0$.
 
 Proceeding in the argument, since $F_h$ is monotone nonincreasing, and uniformly bounded by \eqref{Fhb}, by Helly's selection theorem there exists a nonincreasing function $\widetilde F$ such that, up to subsequences,
 	\begin{align}
 		\tF_h(t)\to \widetilde F(t)\quad \text{as }h\to 0,\quad \forall t\geq 0.\label{Helly}
 	\end{align}
Similarly, concerning $\tF_h$, which is nonincreasing and uniformly bounded as well, Helly's selection theorem gives the existence of a nonincreasing function $F$ such that, up to subsequences,
 	\begin{align}
 		F_h(t)\to F(t)\quad \text{as }h\to 0,\quad \forall t\geq 0.\label{Helly2}
 	\end{align} 
Therefore, passing to the limit in \eqref{chain}, we infer that, for any $t\geq0$,
\begin{align*}
    \tF(t)\leq F(t)\leq \tF(s),\quad \forall 0\leq s<t,
\end{align*}
and thus, since $\tF$ is monotone nonincreasing, with a countable set of discontinuities, letting $s\to t^-$ at the continuity points of $\tF$ allows to conclude that
\begin{align}
F(t)=\tF(t),\quad \text{ for almost any }t\geq 0.
    \label{identification}
\end{align}
We will later on see how to pass from this limit to a limit on the interface energy $E[\cdot]$.
\subsubsection{Identification of the limit characteristic function}
Arguing now as in \cite[Section 3.5.1]{PL}, exploiting the uniform bound \eqref{unifa}, we infer that there exists $\delta_1\in(0,\Ts)$, such that, for any $\delta\in(0,\delta_1)$, 
\begin{align}
	\int_0^{\Ts-\delta}\Ld{\wchi^h(t+\delta)-\wchi^h(t)}^2\dt \leq  C_A(\Ts)\delta,\quad \forall h\in(0,1),
	\label{control1b}
\end{align}
where $C_A>0$ does not depend on $h$, as well as 
\begin{align}
	\int_0^{\Ts}\Ld{\chi^h(t)-{\wchi}^h(t)}^2\dt\leq C_B(\Ts)h,
	\label{controla}
\end{align}
for some $C_B(\Ts)$ independent of $h$. Also, there exists $h_1>0$ such that for any $h\in(0,h_1)$ we have, up to a subsequence,
\begin{align}
	&\sup_{h\in(0,h_1)}\int_0^{\Ts-\delta}\Ld{\chi^{h}(t+\delta)-\chi^{h}(t)}^2\dt\to 0,\quad\text{as }\delta\to0.\label{ttt}
\end{align}
Finally, concerning $\chi_{DG}^h(t)$, observe that it holds, recalling \eqref{timeder1} and $T_h(t)\leq h$,
\begin{align}
\nonumber	&\int_0^\Ts\Ld{\xhi_{DG}^h(t)-\xhi^h(t)}^2\dt \\&\leq \nonumber	\int_0^\Ts\Th{t}^2\Ld{\frac{1}{T_h(t)}\left(\xhi_{DG}^h(t)-\xhi^h(t)\right)}^2\dt\nonumber\\&
	\leq C(\Ts)h^2,\label{CC}
\end{align}
 where $C(\Ts)$ does not depend on $h$. Using now the estimates above, namely \eqref{ttt} and \eqref{CC}, we can easily conclude that there exists $h_2\in(0,h_1)$ such that, up to a subsequence, 
\begin{align}
	\sup_{h\in(0,h_2)}\int_0^{\Ts-\delta}\Ld{\chi^{h}_{DG}({t+\delta})-\chi^{h}_{DG}({t})}^2\dt\to 0\quad\text{ as } \delta\to0.\label{ttt2}
\end{align}
Morever, \eqref{dissiprel} gives
\begin{align*}
	\wchi^h(\cdot),\xhi^h(\cdot),\xhi_{DG}^h(\cdot)	\in L^\infty_{w*}(0,\Ts;BV(\om;\{0,1\})),
\end{align*}
with norms uniformly bounded in $h$. As a consequence, by the embedding $$BV(\om;\{0,1\})\hookrightarrow\hookrightarrow L^{p}({\mathbb T^d})\hookrightarrow \dual,$$ where $p\in(\tfrac{2d}{d+2},\tfrac{d}{d-1})$ ($d=2,3$), recalling \eqref{control1b}, \eqref{ttt}, and \eqref{ttt2}, we can apply the Aubin-Lions-Simon theorem, together with a diagonal argument, to infer the existence of $\xhi\in L^2_{loc}([0,\infty);L^1({\mathbb T^d}))$, such that, up to a subsequence,
\begin{align*}
	\wchi^h(\cdot),\xhi^h(\cdot),\xhi_{DG}^h(\cdot)\to \xhi(\cdot) \text{ strongly in }L^2(0,\Ts;L^1({\mathbb T^d})),
\end{align*}
for any $\Ts>0$. Note that we identified the three limits by means of \eqref{controla} and \eqref{CC}. Of course, this also entails, up to another subsequence, 
\begin{align}
	&\nonumber\wchi^h(\cdot),\xhi^h(\cdot),\xhi_{DG}^h(\cdot)\to \xhi(\cdot) \text{ for almost any   }(t,x)\in (0,\Ts)\times{\mathbb T^d},\\&
\wchi^h(\cdot),\xhi^h(\cdot),\xhi_{DG}^h(\cdot)\to \xhi(\cdot)\text{ strongly in }L^p((0,\Ts)\times\om),\quad \forall p\in[1,\infty),
	\label{finalconv}
\end{align}
and thus $\xhi\in \{0,1\}$ almost everywhere in $(0,\infty)\times\om$. Since $\xhi^h$ is uniformly bounded in the space $L^\infty_{w*}(0,\infty;BV(\om;\{0,1\}))$, this implies that $\xhi^h\to \xhi$ weakly* in  $L^\infty_{w*}(0,\infty;BV(\om;\{0,1\}))$, and then $\xhi\in L^\infty_{w*}(0,\infty;BV(\om;\{0,1\}))$.

Finally, thanks to \eqref{unifa}, we deduce $\xhi\in H^1(0,\infty;H^2(\om)')$ and, using also \eqref{finalconv},
\begin{align}
\pt \wchi^h\rightharpoonup \pt \xhi\quad \text{ weakly in }L^2(0,\Ts;H^2(\om)'),
\label{dtchi1}
\end{align}
for any $\Ts>0$. 
Adapting \cite[Lemma C.1]{AbelsP}, substituting $H^1({\mathbb T^d})'$ with $H^2({\mathbb T^d})'$, we then obtain the regularity 
\begin{align}
	\xhi\in C([0,\Ts];L^p(\om))\quad \forall \Ts>0,\quad 	\forall p\in [1,\infty).
	\label{regchi1a}
\end{align}
Now, following word by word the proof of \cite[(5.123)]{AbelsP}, we can obtain also here that, up to subsequences, 
\begin{align}
\widehat\bu_h^k\to \bv\quad\text{strongly in }L^2(0,\Ts;\bL^2(\om))\quad \forall \Ts>0\quad\text{ as }h\to0.
    \label{ukh}
\end{align}
As a consequence, we see that, recalling \eqref{basic}, \eqref{dissiprel}, and \eqref{ukh}, we can apply Lemma \ref{base2}, with $\bv_n=\widehat \bu_h^k,\ w_n=w_h,\ \xhi_n=\xhi^h$ to infer that 
\begin{align}
\frac{\chi^h-\xhi^h\circ X_{-h}^{\widehat\bu^k_h}}h\to \bv\cdot\nabla \xhi\quad\text{ strongly in }L^2(0,\Ts;H^1(\om)')
    \label{matder}
\end{align}
for any $\Ts>0$, as well as 
\begin{align*}
    \nabla \left(w_h(\cdot)\circ X_{g_h(\cdot)}^{\widehat{\bu}_h^k(\cdot)}\right)\rightharpoonup \nabla w\quad\text{weakly in }L^2(0,\Ts;\bL^2(\om)).
\end{align*}
Therefore, together with \eqref{finalconv} and \eqref{ukh} we can finally infer that 
\begin{align}
	&\int_0^\Ts\left( \xhi^h{(\tau)} \widehat{\bu}_h^k(\tau),\nabla \left(w_h(\tau)\circ X_{g_h(\tau)}^{\widehat{\bu}_h^k(\tau)}\right)\right)_\om\d \tau\to \int_{0}^\Ts \left( \xhi{(\tau)} \bv(\tau),\nabla \tw(\tau)\right)_\om\d \tau\label{reusable1}
\end{align}
as $h\to 0$, for any $\Ts>0$.
Finally, from \eqref{dtchi1} and \eqref{matder} we easily infer 
\begin{align}
\partial^{\bullet,h}_t\xhi_h=\partial_t\widehat\xhi^h+\frac {\xhi^{h}-\xhi^h\circ X_{-h}^{ \widehat{\bu}_h^k}}{h}\rightharpoonup \partial_t\xhi+\bv\cdot \nabla \xhi,\ \text{weakly in }L^2(0,\Ts; H^2(\om)').
	\label{pchia}
\end{align}
\subsubsection{Limit in the fluid equations as $h\to0$}
We can now pass to the limit in the fluid equations \eqref{dtv}. First, by \eqref{regvbisA} we see that, since $k\to\infty$ as $h\to0$,
\begin{align*}
    \frac1k\norm{\bA\bv_k}_{L^2(0,\Ts;\bL^2(\om))}\leq \frac 1{\sqrt k} C(\Ts)\to 0\qquad\text{ as }h\to0.
\end{align*}
Moreover, recalling \eqref{weakconv1} and \eqref{finalconv}, we infer
\begin{align*}
    \int_0^\Ts\int_\om\nu(\xhi^h)\nabla\bv_h:\nabla\zeta\dx\dt\to \int_0^\Ts\int_\om\nu(\xhi)\nabla\bv:\nabla\zeta\dx\dt\qquad \text{ as }h\to0
\end{align*}
for any $\zeta\in L^2(0,\Ts;\dot\bH_\sigma^1(\om))$. Finally, using \eqref{whconv}, and \eqref{finalconv} we infer
\begin{align*}
    \int_0^\Ts\int_\om \zeta\cdot \nabla w_h\xhi^h\dx\dt\to  \int_0^\Ts\int_\om \zeta\cdot \nabla w\xhi\dx\dt\qquad\text{ as }h\to 0 
\end{align*}
for any $\zeta\in L^2(0,\Ts;\dot\bH_\sigma^1(\om))$.
As a consequence, using convergences \eqref{whconv}-\eqref{str1} together with the convergences above, we can pass to the limit as $h\to 0$ in the corresponding weak formulation of \eqref{dtv}, obtaining
\begin{align}
&\int_0^\Ts\int_\om \langle \pt\bv,\zeta\rangle_{\dot\bH^2_\sigma(\om)',\dot\bH^2_\sigma(\om)}\dx\dt-\int_0^\Ts\int_\om \bv\otimes\bv:\nabla \zeta\dx\dt+\int_0^\Ts\int_\om\nu(\xhi)\nabla\bv:\nabla\zeta\dx\dt\nonumber\\&= -\int_0^\Ts\int_\om \zeta\cdot \nabla w\xhi\dx\dt.
    \label{weakafterlim}
\end{align}
for any $\zeta\in L^2(0,\Ts;\dot\bH_\sigma^2(\om))$. Of course, here we could improve the regularity of $\pt\bv$, but since we do not expect higher regularity in the next passage to the limit as $\alpha\to0$, we do not proceed further in this direction.

\subsubsection{Limit varifolds: Disintegration of the measure} 
For the sake of brevity, we will now define
 \begin{align*}
     \wtilde D_h(t):=\int_0^{t}\left( \xhi^h{(\tau)} \widehat{\bu}_h^k(\tau),\nabla \left(w_h(\tau)\circ X_{g_h(\tau)}^{\widehat{\bu}_h^k(\tau)}\right)\right)_\om\d \tau \to D(t):=\int_0^t\left( \xhi{(\tau)} \bv(\tau),\nabla \tw(\tau)\right)_\om\d \tau
 \end{align*}
as $h\to0$  for any $t\geq 0$ by \eqref{reusable1}, and $D_h(t):=\wtilde D_h(\floor{\tfrac th}h)\to D(t)$ as $h\to0$ as well.
Thanks to the uniform bound of the energy \eqref{dissiprel}, there exists a measure ${\mu}:{\mathcal B((0,\infty)	\times{\mathbb T^d}\times \mathbb S^{d-1})}\to \mathbb R^+$, such that, up to subsequences, 
\begin{align}
	\mathcal L^1\llcorner{(0,\Ts)}\otimes (\mu_h^t)_{t\in(0,\Ts)}\rightharpoonup {\mu},\quad \text{ weakly* in } \mathcal M((0,\Ts)\times{\mathbb T^d}\times \mathbb S^{d-1}),\label{convmeas}
\end{align}
 for any $\Ts>0$, where $\mathcal B((0,\infty)	\times{\mathbb T^d}\times \mathbb S^{d-1})$ is the Borel $\sigma$-algebra on $(0,\infty)	\times{\mathbb T^d}\times \mathbb S^{d-1}$. This measure $\mu$ can be sliced in time thanks to the convergences \eqref{Helly2} and \eqref{reusable1}, following, e.g., \cite{AbelsP}. 
In particular, there
exists an $\mathcal L^1\llcorner(0,\infty)$-measurable function  $\beta$ on $(0, \infty)$ and a weakly-* $\beta$-measurable family of Radon probability measures $(\tau_t)_t$ on ${\mathbb T^d}\times\mathbb S^{d-1}$, so that
\begin{align}
	{\mu}=\mathcal L^1\llcorner(0,\infty)\otimes (\mu_t)_{t\in(0,\infty)},
	\label{representation}
\end{align}
with $\mu_t={\beta}(t)\tau_t$.
In particular (cf. eq. \cite[(5.132)]{AbelsP}), for any $\zeta\in C(\om\times\mathbb S^{d-1})$, $0\leq\zeta\leq1$, and any nonnegative $\eta\in C_c((0,\infty))$,
\begin{align}	
	&\nonumber\int_{(0,\infty)\times \om\times \mathbb S^{d-1}}\zeta(x,s)\eta(t)\d \mu_t\dt\\&
	\leq \int_{\text{supp}(\eta)} \eta(t) F(t)\dt+\int_{\text{supp}(\eta)} \eta(t)D(t)\dt,\label{bound1}
\end{align}
for almost any $t\geq0$, which gives, after a localization in time and choosing $\zeta\equiv 1$ (which is valid since $\om\times\sf$ is compact), recalling \eqref{identification},
\begin{align}
  \normmm{\mu_t}_{\sf}(\om)-D(t)&\leq  \wtilde F(t)
 \label{A1}
\end{align}
for a.a. $t\geq0$.
By the very same arguments, exploiting this time \eqref{Helly}, we infer that 
there exists $\wtilde\mu:{\mathcal B((0,\infty)	\times{\mathbb T^d}\times \mathbb S^{d-1})}\to \mathbb R^+$, such that, up to subsequences, 
\begin{align}
	\mathcal L^1\llcorner{(0,\Ts)}\otimes (\wtilde\mu_h^t)_{t\in(0,\Ts)}\rightharpoonup \wtilde\mu,\quad \text{ weakly* in } \mathcal M((0,\Ts)\times{\mathbb T^d}\times \mathbb S^{d-1}),\label{convmeas1}
\end{align}
 for any $\Ts>0$, where 
\begin{align}
	\wtilde\mu=\mathcal L^1\llcorner(0,\infty)\otimes (\wtilde\mu_t)_{t\in(0,\infty)},
	\label{representation1}
\end{align}
and $\wtilde\mu_t$ is weakly-* $\mathcal L^1\llcorner(0,\infty)$ measurable.
 
 Now, considering the set $[0,\Ts]\times {\mathbb T^d}\times \mathbb S^{d-1}$ which is compact in $[0,\infty)\times{\mathbb T^d}\times \mathbb S^{d-1}$, we can apply \cite[Proposition 4.26]{Maggi} to deduce from \eqref{convmeas} that 
\begin{align}
   \limsup_{h\to\infty} \int_0^\Ts E[\xhi^h(t)]\dt\leq \int_0^\Ts \normmm{\mu_t}_{\mathbb S^{d-1}}({\mathbb T^d})\dt
   \label{limsup}
\end{align}
for any $\Ts>0$. Also, by Lebesgue's Dominated Convergence theorem, using \eqref{Helly2}, \eqref{identification}, \eqref{reusable1}, and \eqref{limsup}, gives
\begin{align*}
    &\int_0^\Ts \widetilde F(t)\dt=\lim_{h\to\infty} \int_0^\Ts (E[\xhi^h(t)]-D_h(t))\dt\\&\leq \limsup_{h\to\infty} \int_0^\Ts E[\xhi^h(t)]\dt-\int_0^\Ts D(t)\dt\\&\leq \int_0^\Ts \normmm{\mu_t}_{\mathbb S^{d-1}}({\mathbb T^d})\dt-\int_0^\Ts D(t)\dt
\end{align*}
for any $\Ts>0$, which together with \eqref{A1} allows to infer
\begin{align*}
    \int_0^\Ts \widetilde F(t)\dt= \int_0^\Ts (\normmm{\mu_t}_{\mathbb S^{d-1}}({\mathbb T^d})-D(t))\dt,
\end{align*}
entailing, as $\Ts$ is arbitrary, 
\begin{align}
    \widetilde F(t)=\normmm{\mu_t}_{\mathbb S^{d-1}}({\mathbb T^d})-D(t)\label{identif1}
\end{align}
for almost any $t>0$.

Then, by \eqref{convmeas1}, we have,  since $\mathbb S^{d-1}$ is compact,  
\begin{align}
	\mathcal L^1\llcorner{(0,\Ts)}\otimes (\normmm{\mu_h^t}_{\mathbb S^{d-1}})_{t\in(0,\Ts)}\rightharpoonup \mathcal L^1\llcorner{(0,\Ts)}\otimes (\normmm{\mu_t}_{\mathbb S^{d-1}})_{t\in(0,\Ts)}\quad \text{ weakly* in } \mathcal M((0,\Ts)\times{\mathbb T^d}),\label{convmeas2}
\end{align}
entailing, arguing as in \cite{PL}, that
\begin{align}
	\mathcal L^1\llcorner{(0,\Ts)}\otimes (\rho_h\ast\vert{\nabla\xhi^h(t)}\vert)_{t\in(0,\Ts)}\rightharpoonup \mathcal L^1\llcorner{(0,\Ts)}\otimes (\normmm{\mu_t}_{\mathbb S^{d-1}})_{t\in(0,\Ts)}\quad \text{ weakly* in } \mathcal M((0,\Ts)\times{\mathbb T^d}),\label{convmeas3}
\end{align}
where we wrote $$\rho_h\ast \normmm{\nabla \xhi^h(t)}(A)=\int_A \rho_h\ast \normmm{\nabla \xhi^h(t)}(x)\dx$$ for any Borel measurable set $A\subset \R^d$.
\subsubsection{Gibbs-Thomson law as $h\to0$ and generalized mean curvature.}
\label{gibbs}
For any $\Ts>0$, we consider a vector field $ B:(0,\Ts)\times\om\to \R^d$ continuous and compactly supported in $ (0,\Ts)\times\om$, such that, for any $t\in(0,\Ts)$, $B(\cdot,t)\in C^1({\mathbb T^d},\mathbb R^d)$. As such, by \eqref{convmeas} and \eqref{representation} it holds 
\begin{align}
	\int_0^\Ts \delta\widetilde\mu_t^h(B)\dt\to \int_0^\Ts \delta\widetilde{\mu}_t(B)\dt. \label{v1}
\end{align}
Multiplying \eqref{allard} by a smooth and
compactly supported test function on $(0, \Ts)$, integrating in time, and letting $h\to0$, using \eqref{whconv} and \eqref{finalconv}, and then, after a time localization, we deduce that, for almost any  $t>0$, we have
\begin{align}
	&	\delta\wtilde\mu_t(B)=
	\int_{\mathbb T^d} \chi(\cdot,t)\text{div}(\tw(\cdot,t)B)\,\d x,\label{variation2}
\end{align} 
for any $B\in C^1({\mathbb T^d};\R^d)$. In the same way, the uniform control \eqref{regvbis} on $u_h$ and \eqref{allard2} give 
\begin{align}
	&	\delta\mu_t(B)=
	\int_{\mathbb T^d} \chi(\cdot,t)\text{div}(u(\cdot,t)B)\,\d x,\label{variation2b}
\end{align} 
for any $B\in C^1({\mathbb T^d};\R^d)$ and almost any $t>0$. 
Following then \cite[Lemmas 4.1-4.2]{Roger}, thanks to \eqref{finalconv} we can prove that, for any $B\in L^2(0,\Ts; C^1({\mathbb T^d};\R^d))$,
\begin{align}
	&	\delta\wtilde\mu_t(B)=\delta \mu_{t}(B),\label{variation4}
\end{align} 
for any $B\in C^1({\mathbb T^d};\R^d)$ and for almost any $t>0$. This is possible after introducing the generalized mean curvature in the sense of \cite{Roger}, which is crucially intrinsic to the surface $\text{supp}\normmm{\nabla\xhi(\cdot,t)}$ for almost every $t>0$.

\subsubsection{Limits of dissipative terms in the energy as $h\to0$: mean curvature and $w_h$}

We first consider the term which will be associated to the generalized mean curvature. Namely we write, by Young's inequality, for any $0\leq s<T'$,
\begin{align}
    \nonumber&\frac12\int_s^{T'}\int_{\mathbb T^d} w_h(t)^2\normmm{\rho_h\ast{\nabla \xhi^h(t)}}\dx\dt\\&\geq \int_s^{T'}\int_{\mathbb T^d} w_h(t)\zeta\normmm{\rho_h\ast{\nabla \xhi^h(t)}}\dx\dt-\frac 12\int_s^{T'}\int_{\mathbb T^d} \zeta^2\normmm{\rho_h\ast{\nabla \xhi^h(t)}}\dx\dt,\label{inequality}
\end{align}
for $\zeta(t):=\xi(t)\cdot\frac{\rho_h\ast\nabla\xhi^h(t)}{\normmm{\rho_h\ast \nabla\xhi^h(t)}}$, with $\xi\in C^\infty_c((0,\Ts)\times{\mathbb T^d};\R^d)$ and $\xi(t)\in C^1(\om;\R^d)$ for any $t\in(0,\Ts)$, \textcolor{black}{where we tacitly set $\frac{\rho_h\ast\nabla\xhi^h(t)}{\normmm{\rho_h\ast \nabla\xhi^h(t)}}=0$ when $\normmm{\rho_h\ast \nabla\xhi^h(t)}=0$}. Arguing exactly as to obtain \cite[(3.71)]{PL}, exploiting \eqref{convmeas3} and the fundamental identification \eqref{variation4}, we get, 
for any $0\leq s<T'<\infty$,
\begin{align}
&\nonumber\liminf_{h\to0}\frac12\int_s^{T'}\int_{\mathbb T^d} w_h(t)^2\normmm{\rho_h\ast{\nabla \xhi^h(t)}}\dx\dt\\&\geq \int_s^{T'} \delta\mu_t(\xi(t))\dt-\frac 12 \int_s^{T'} \norm{\xi(t)}^2_{\mathcal V_t}\dt,\quad\forall{ \xi\in C^1_c((s,T');C^1(\om;\R^d))}.\label{curvat5}
\end{align}
Taking then the supremum over $\xi\in C^1_c((s,T');C^1(\om;\R^d))$ finally gives the desired 
\begin{align}
\liminf_{h\to0}\frac12\int_s^{T'}\int_{\mathbb T^d} w_h(t)^2\normmm{\rho_h\ast{\nabla \xhi^h(t)}}\dx\dt\geq \onehalf\int_s^{T'} \normmm{\partial E(\mu_t)}_{{\mathcal V_t}}^2\dt,\label{curvat6}\quad \forall 0\leq s<T',
\end{align}
recalling the metric slope definition \eqref{slope}.

The other dissipative terms in the energy inequality \eqref{eA}, still related to $w_h$, can be easily treated by lower semicontinuity, namely, recalling, \eqref{whconv}, 
\begin{align}
 & \nonumber \alpha\norm{ w}_{L^2(s,T';L^2({\mathbb T^d}))}^2+  \alpha\norm{\nabla  w}_{L^2(s,T';L^2({\mathbb T^d}))}^2 \\&\leq \liminf_{h\to0}\left( \alpha\norm{w_h}_{L^2(s,T';L^2({\mathbb T^d}))}^2+  \alpha\norm{\nabla w_h}_{L^2(s,T';L^2({\mathbb T^d}))}^2\right)\label{uwineq}
\end{align}
for any $0<s<T'$.

\subsubsection{Limits of dissipative terms in the energy as $h\to0$: interface velocity vs $u_h$}
We now study the interface velocity term appearing in the energy, which is the one associated to $u_h$. We introduce the Hilbert space 
$$W:=L^2(0,\infty;H^1({\mathbb T^d}))\cap L^2((0,\infty);L^2({\mathbb T^d},\d\normmm{\mu_{(\cdot)}}_{\mathbb S^{d-1}})),$$
inducing the norm
\begin{align*}
\norm{v}_{W}^2:={\alpha\norm{v}_{L^2(0,\infty;L^2({\mathbb T^d}))}^2+\alpha\norm{\nabla v}_{L^2(0,\infty;L^2({\mathbb T^d}))}^2+\norm{v}_{L^2(0,\infty;L^2({\mathbb T^d},\d\normmm{\mu_{(\cdot)}}_{\mathbb S^{d-1}}))}^2},
\end{align*}
and inner product
\begin{align*}
    (v,w)_W=\int_0^\infty\int_{\mathbb T^d} vw\d\normmm{\mu_t}_{\mathbb S^{d-1}}\dt+\alpha\int_0^\infty\int_{\mathbb T^d} vw\dx\dt+\alpha\int_0^\infty\int_{\mathbb T^d} \nabla v\cdot\nabla w\dx\dt,\quad \forall v,w\in W.
\end{align*}
Given $\zeta\in C^\infty_c((0,\infty)\times\om)$, we define the weak material time derivative as 
$$
\langle\pt^\bullet \xhi,\zeta\rangle:=-\int_0^\infty\int_\om \xhi\pt\zeta\dx\dt-\int_0^\infty\int_\om \xhi\Div(\bv\zeta)\dx\dt,
$$
noticing that, by \eqref{pchia},
\begin{align*}
    \int_0^\infty\int_\om(\pt^{\bullet,h} \xhi,\zeta)\dt\to \langle\pt^\bullet \xhi,\zeta\rangle,\quad\text{ as }h\to0.
\end{align*}

 As a consequence, by means \eqref{uhA}, for any $\zeta\in C_c^\infty({\mathbb T^d}\times(0,\Ts))$ we get, by Cauchy-Schwarz inequality,
\begin{align}
&\nonumber\normmm{\langle\partial_t^\bullet{\xhi}(t),\zeta(t)\rangle}\\&\nonumber= \lim_{h\to 0}\normmm{\int_0^{\Ts}\int_{\mathbb T^d}\pt^{\bullet,h}{\xhi_h}(t)\zeta(t)\dx\dt}\\&\nonumber
\leq \liminf_{h\to0}\normmm{\int_0^\Ts \int_{\mathbb T^d} {u_h(t)\zeta(t)}\normmm{\rho_h\ast \nabla\chi^h(t)}\dx\dt}\\&
\quad +\alpha^2\normmm{\int_0^\Ts\int_\om u(t)\zeta(t)\dx\dt} +\alpha^2\normmm{\int_0^{\Ts}\int_{\mathbb T^d} \nabla u(t)\cdot \nabla\zeta(t)\dx\dt}.\label{dtc}
\end{align}
Recalling \eqref{hutchinson} and \eqref{convmeas3}, by Cauchy-Schwarz inequality we then get
\begin{align*}
   & \liminf_{h\to0}\normmm{\int_0^\Ts \int_{\mathbb T^d} {u_h(t)\zeta(t)}\normmm{\rho_h\ast \nabla\chi^h(t)}\dx\dt}\\&\leq \liminf_{h\to0}\left(\left(\int_0^\Ts \int_{\mathbb T^d} \normmm{u_h(t)}^2\normmm{\rho_h\ast \nabla\chi^h(t)}\dx\dt\right)^\frac12\left(\int_0^\Ts \int_{\mathbb T^d} \normmm{\zeta(t)}^2\normmm{\rho_h\ast \nabla\chi^h(t)}\dx\dt\right)^\frac12\right)\\&
    \leq C(\Ts)\liminf_{h\to0}\left(\int_0^\Ts \int_{\mathbb T^d} \normmm{\zeta(t)}^2\rho_h\ast \normmm{\nabla\chi^h(t)}\dx\dt\right)^\frac12\\&  \leq C(\Ts)\left(\int_0^\Ts \int_{\mathbb T^d} \normmm{\zeta(t)}^2\d\normmm{\mu_t}_{\mathbb S^{d-1}}\dt\right)^\frac12.
\end{align*}
Therefore, coming back to \eqref{dtc}, after using Cauchy-Schwarz inequality also for the remaining terms in the right-hand side, we deduce
\begin{align*}
&\normmm{\langle \partial_t^\bullet \xhi,\zeta\rangle}\\&
\leq C(\Ts)\left(\int_0^\Ts \int_{\mathbb T^d} \normmm{\zeta(t)}^2\d\normmm{\mu_t}_{\mathbb S^{d-1}}\dt\right)^\frac12\\&
\quad +\alpha^\frac12\left(\alpha^2{\int_0^\Ts\int_\om \normmm{u(t)}^2\dx\dt}\right)^\frac12\left(\alpha{\int_0^\Ts\int_\om \normmm{\zeta(t)}^2\dx\dt}\right)^\frac12\\&\quad +\alpha^\frac12\left(\alpha^2{\int_0^{\Ts}\int_{\mathbb T^d} \normmm{\nabla u(t)}^2\dx\dt}\right)^\frac12\left(\alpha{\int_0^{\Ts}\int_{\mathbb T^d} \normmm{\nabla \zeta(t)}^2\dx\dt}\right)^\frac12\\&
\leq (1+2\alpha^\frac12)C(\Ts) \norm{\zeta}_W,
\end{align*}
where we used the uniform control \eqref{regvbis}.
Extending by density the material derivative $\partial_t^\bullet\chi\in W'$, by the arbitrariness of $\Ts>0$ and $\zeta$, there exists $\widehat V\in W$ such that  
\begin{align}
    \langle \partial_t^\bullet \xhi, \zeta\rangle=(\widehat V,\zeta)_W=\int_0^\infty\int_{\mathbb T^d}  \widehat V\zeta\d\normmm{\mu_t}\dt+\alpha\int_0^\infty\int_{\mathbb T^d} \widehat V\zeta\dx\dt+\alpha\int_0^\infty\int_{\mathbb T^d}\nabla \widehat V\cdot\nabla \zeta\dx\dt\label{base11}
\end{align}
for any $\zeta\in W$. We now define the space $W_\Ts:=W_{|(0,\Ts)}$. Then, again from \eqref{dtc}, by Cauchy-Schwarz inequality and the lower semicontinuity of the norms we get, for any $\zeta\in C_c^\infty((0,\Ts)\times\om)$, 
\begin{align*}
    \normmm{(\widehat V,\zeta)_{W_\Ts}}&=\normmm{\langle \pt^\bullet \xhi,\zeta\rangle}\\&\quad +\alpha^\onehalf\left(\alpha^2\norm{u}_{L^2(0,\Ts;L^2({\mathbb T^d}))}^2\right)^\frac12\alpha^\onehalf\norm{\zeta}_{L^2(0,\Ts;L^2(\om))}\\&\quad +\alpha^\onehalf\left(\alpha^2\norm{\nabla u}_{L^2(0,\Ts;L^2({\mathbb T^d}))}^2\right)^\frac12\alpha^\onehalf\norm{\zeta}_{L^2(0,\Ts;H^1(\om))}\\&\leq \liminf_{h\to0}\left(\int_0^\Ts \int_{\mathbb T^d} \normmm{u_h(t)}^2\normmm{\rho_h\ast \nabla\chi^h(t)}\dx\dt\right)^\frac12\norm{\zeta}_{W_\Ts}\\&\quad +\alpha^\frac12\liminf_{h\to0}\left(\alpha^2\norm{u_h}_{L^2(0,\Ts;L^2({\mathbb T^d}))}^2\right)^\frac12\norm{\zeta}_{W_\Ts}\\&\quad +\alpha^\frac12\liminf_{h\to0}\left(\alpha^2\norm{\nabla u_h}_{L^2(0,\Ts;L^2({\mathbb T^d}))}^2\right)^\frac12\norm{\zeta}_{W_\Ts},
\end{align*}
where we used, recalling again \eqref{convmeas3}, the inequality
\begin{align*}
   & \liminf_{h\to0}\normmm{\int_0^\Ts \int_{\mathbb T^d} {u_h(t)\zeta(t)}\normmm{\rho_h\ast \nabla\chi^h(t)}\dx\dt}\\& \leq \liminf_{h\to 0}\left(\int_0^\Ts \int_{\mathbb T^d} \normmm{u_h(t)}^2\normmm{\rho_h\ast \nabla\chi^h(t)}\dx\dt\right)^\frac12\lim_{h\to0}\left(\int_0^\Ts \int_{\mathbb T^d} \normmm{\zeta(t)}^2\rho_h\ast \normmm{\nabla\chi^h(t)}\dx\dt\right)^\frac12\\&  =\liminf_{h\to 0}\left(\int_0^\Ts \int_{\mathbb T^d} \normmm{u_h(t)}^2\normmm{\rho_h\ast \nabla\chi^h(t)}\dx\dt\right)^\frac12\left(\int_0^\Ts \int_{\mathbb T^d} \normmm{\zeta(t)}^2\d\normmm{\mu_t}_{\mathbb S^{d-1}}\dt\right)^\frac12.
\end{align*}
Then, applying Young's inequality and arguing as in \cite[Sec. 3.5.5]{PL}, we find
\begin{align}
   \nonumber &\frac1{1+4\alpha^\onehalf}\norm{\widehat V}_{W_{\Ts}}^2
   \\& \leq \liminf_{h\to 0}\left(\int_0^\Ts\int_{\mathbb T^d} \normmm{u_h(t)}^2\normmm{\rho_h\ast \nabla\chi^h(t)}\dx\dt+\alpha^2\norm{u_h}_{L^2(0,\Ts;L^2({\mathbb T^d}))}^2+\alpha^2\norm{\nabla u_h}_{L^2(0,\Ts;L^2({\mathbb T^d}))}^2\right),\label{uvelox}
\end{align}
The same argument works similarly if we use $(s,T')$, for $0<s<T'$ in place of the interval $(0,\Ts)$. 
\subsubsection{Energy dissipation inequality as $h\to0$: fluid velocity terms}
Concerning the fluid velocity terms appearing in the energy inequality \eqref{enerb}, here we immediately see from \eqref{weakconv1} and \eqref{finalconv}, by lower semicontinuity, that
\begin{align}
   \int_s^{T'} \int_\om \nu(\xhi)\normmm{\nabla \bv}^2\dx\dt \leq \liminf_{h\to0}\int_s^{T'} \int_\om \nu(\xhi^h)\normmm{\nabla \bv_h}^2\dx\dt \label{lscA}
\end{align}
as $h\to0$, for any $0\leq s<T'<\infty$.
Moreover, strong convergence \eqref{str1} guarantees that
\begin{align}
\onehalf\int_\om \normmm{\bv_h(t)}^2\dx\to \onehalf\int_\om \normmm{\bv(t)}^2\dx\qquad \text{as }h\to0
    \label{aarìt}
\end{align}
for almost any $t\geq0$.
Finally, using \eqref{whconv}, \eqref{str1}, and \eqref{finalconv} we obtain
\begin{align}
    \int_s^{T'}\int_\om \chi^h\bv_h\cdot\nabla w_h \dx\dt\to   \int_s^{T'}\int_\om \chi\bv\cdot\nabla w\dx\dt\label{lscB},
\end{align}
as $h\to0$, for any $0\leq s<T'<\infty$. Therefore, recalling that $\tfrac1k\norm{\bA\bv_h}_{\bL^2(\om)}^2\geq0$, we can pass to the limit as $h\to0$ in \eqref{enerb}, obtaining 
\begin{align}
 	&\onehalf\int_\om \normmm{\bv(T')}^2\dx+\int_s^{T'}\normh{\sqrt{\nu(\xhi)}D \bv}^2\d t\nonumber\\&
 	\leq 
 	\onehalf\int_\om \normmm{\bv(s)}^2\dx-\int_s^{T'}\int_\om \chi\nabla w\cdot \bv dx\d t,\label{enerbfinal}
 \end{align}
for almost any $0<s<T'$, with $s=0$ included.
\subsubsection{Energy dissipation inequality as $h\to0$: advective Mean Curvature Flow} 
By the results of the previous sections, we can now pass to the limit in \eqref{eA}, recalling \eqref{ineqgh}, \eqref{reusable1} (valid also on any interval $(s,T')$), \eqref{convmeas1}, \eqref{curvat6}, \eqref{uwineq}, and \eqref{uvelox}, together with the fundamental identification \eqref{identif1},   first taking $h\to0^+$, then $\tau\to T$ and $\kappa\to s$,  to infer  
\begin{align}
	\nonumber	&E[{\mu_{T'}}]-\int_0^{T'}\int_\om \chi\nabla w\cdot \bv dx\d t+\onehalf\int_s^{T'} \normmm{\partial E(\mu_t)}_{{\mathcal V_t}}^2\dt\\&+\frac1{2(1+4\alpha^\frac12)}\int_s^{T'}\int_{\mathbb T^d} \widehat V(t)^2\d  \normmm{\mu_t}_{\mathbb S^{d-1}}\dt\nonumber+\frac{\alpha}{2(1+4\alpha^\frac12)}\int_{s}^{T'}\left(\normh{\widehat V(t)	}^2+\normh{\nabla \widehat V(t)}^2\right)\dt\\&\nonumber+\frac{\alpha^2}{2}\int_{s}^{T'}(\normh{\tw(t)}^2+\normh{\nabla\tw(t)}^2)\d t\\& \leq E[{\mu_s}]-\int_0^{s}\int_\om \chi\nabla w\cdot \bv dx\d t
	\label{eb1}
\end{align}
for almost any $0<s<T'$. To be precise, observing the terms appearing in the right-hand side of \eqref{eA}, here we have also used the fact that, since $0\leq t-\floor{\tfrac th}h\leq h$, by \eqref{reusable1} we have
\begin{align*}
   & \left\vert\int_{0} ^{{T'}}\frac{1}{g_h(t)}\left( \xhi^h{(t)} \widehat{\bu}_h^k(t),\nabla \left(w_h(t)\circ X_{g_h(t)}^{\widehat{\bu}_h^k(t)}\right)\right)_\om\d t\right.\\&\left.-\int_{0} ^{\floor{\frac {T'}h}h}\frac{1}{g_h(t)}\left( \xhi^h{(t)} \widehat{\bu}_h^k(t),\nabla \left(w_h(t)\circ X_{g_h(t)}^{\widehat{\bu}_h^k(t)}\right)\right)_\om\d t\right\vert\to 0
\end{align*}
as $h\to0$. The validity of \eqref{eb1} for $s=0$ can be inferred analogously, where we recall that we have set $E[\mu_t]:=E[\xhi_0]$ when $t=0$.

Since in the final step as $\alpha\to0$ we need a unique energy inequality, due to the fact that the term $\int_0^{T'}\int_\om \chi\nabla w\cdot \bv dx\d t$ may not be finite anymore as $\alpha\to0$, we now sum up \eqref{enerbfinal} and \eqref{eb1}, thus simplifying the desired term and obtaining
\begin{align}
	\nonumber	&\onehalf\int_\om\normmm{\bv(T')}^2\dx+E[{\mu_{T'}}]+\int_s^{T'}\normh{\sqrt{\nu(\xhi)}D \bv}^2\d t+\onehalf\int_s^{T'} \normmm{\partial E(\mu_t)}_{{\mathcal V_t}}^2\dt\\&+\frac1{2(1+4\alpha^\frac12)}\int_s^{T'}\int_{\mathbb T^d} \widehat V(t)^2\d  \normmm{\mu_t}_{\mathbb S^{d-1}}\dt\nonumber+\frac{\alpha}{2(1+4\alpha^\frac12)}\int_{s}^{T'}\left(\normh{\widehat V(t)	}^2+\normh{\nabla \widehat V(t)}^2\right)\dt\\&\nonumber+\frac{\alpha^2}{2}\int_{s}^{T'}(\normh{\tw(t)}^2+\normh{\nabla\tw(t)}^2)\d t\\& \leq \onehalf\int_\om\normmm{\bv(s)}^2\dx+E[{\mu_s}],
	\label{finaleenergy}
\end{align}
for almost any $0<s<T'$, with $s=0$ included.
\subsubsection{Admissibility of the couple $(\xhi,\mu)$.}
\label{admissibilityA}
The couple $(\xhi,\mu)$ is admissible in the sense of Definition \ref{admissibility}. In particular, concerning Property (1), equation \eqref{PE} can be shown by means of the corresponding definition of $\mu_t^h$, the Gau{\ss} theorem for BV functions and passing to the limit as $h\to 0$, see \cite{SH,AbelsP}, after localization in time. 
 Finally, about Property (5), i.e., the measurability in time of the total mass measure, observe that from \eqref{identif1} it holds 
 \begin{align*}
     E[\mu_t]=\widetilde F(t)-D(t),
 \end{align*}
 where $\tF$ is measurable since it is a monotone function, whereas $D$ is even absolutely continuous on $[0,\Ts]$ for any $\Ts>0$, as $D(t)$ is the integral on $(0,t)$ of the $L^1_{loc}([0,\infty))$ function $t\mapsto (\xhi(t)\bv(t),\nabla w(t))_\om$.  As a consequence, the couple $(\xhi,\mu)$ is admissible in the sense of Definition \ref{admissibility}.  Clearly, when $\alpha>0$ many more properties can be shown (cf. \cite{AbelsP}), but since we will take the limit $\alpha\to 0$ we do not need to verify further properties.

\subsection{Final limit as $\alpha\to0$}
We need now to pass to the limit as $\alpha\to0$. In particular, considering a sequence of solutions $\{(\xhi_\alpha,\mu_\alpha,\bv_\alpha)\}_{\alpha}$ indexed by $\alpha>0$, from the energy inequality\eqref{finaleenergy}, also using Korn's inequality, we infer
\begin{align}
&\sup_{t\geq 0}\normmm{\nabla \xhi_\alpha(t)}(\om)\leq \sup_{t\geq0}\normmm{\mu_{\alpha}^t}_{\sf}(\om)\leq (E[\xhi_0]+\frac12\norm{\bv_0}_{\bL^2(\om)}^2)\label{energia1},\\&
\sup_{t\geq0}\norm{\bv_\alpha}_{\dot\bL_\sigma^2(\om)}+\norm{\bv_\alpha}_{L^2(0,\infty;\dot\bH_\sigma^1(\om))}\leq C(E[\xhi_0],\norm{\bv_0}),
\label{velox}\\& \norm{\widehat V_\alpha}_{L^2(0,\infty;L^2(\om;\d\normmm{\mu_{\alpha}^{(\cdot)}}))}\leq C(E[\xhi_0],\norm{\bv_0}), \label{energiagiusta}\\&
    \sqrt \alpha \norm{\widehat V_\alpha}_{L^2(0,\infty;H^1(\om))}\leq C(E[\xhi_0],\norm{\bv_0}),\label{Vhat}
\end{align}
 for any $\alpha>0$, where the constants $C>0$ are independent of $\alpha$. Also note that that the energy $t\mapsto E[\mu_{\alpha,t}]+\frac12\norm{\bv_\alpha(t)}^2_{\bL^2(\om)}$ is monotone nonincreasing. Therefore, by Helly's selection theorem there exists a nonincreasing function $\mathcal Z$ such that, up to subsequences,
 	\begin{align}
E[\mu_{\alpha,t}]+\frac12\norm{\bv_\alpha(t)}^2_{\bL^2(\om)}\to \cZ(t)\quad \text{as }h\to 0,\quad \forall t\geq 0.\label{Hellybis}
 	\end{align}
    Then, \eqref{velox} entails that there exists $\bv\in L^\infty(0,\infty;\dot\bL_\sigma^2(\om))\cap L^2(0,\infty;\Hs)$ such that
    \begin{align}
        \label{v1a}
       & \bv_\alpha\rightharpoonup \bv \quad\text{ weakly* in }L^\infty(0,\infty;\dot\bL_\sigma^2(\om)),\\&
        \bv_\alpha\rightharpoonup \bv \quad\text{ weakly in }L^2(0,\infty;\Hs)
        \label{v2a}
    \end{align}
    as $\alpha\to0$.
Additionally, recalling that $\langle \pt \xhi_\alpha, \zeta\rangle=-\int_0^\infty\int_\om \xhi_\alpha\pt\zeta\dx\dt$ for any $\zeta\in C^\infty_c((0,\infty);H^2(\om))$, we infer from \eqref{base11}, exploiting \eqref{energia1}-\eqref{energiagiusta}, that
 \begin{align*}
     &\normmm{\langle \pt \xhi_\alpha, \zeta\rangle}\\&\leq \norm{\bv_\alpha}_{L^2(0,\infty;\dot\bL_\sigma^2(\om))}\norm{\nabla\zeta}_{L^2(0,\infty;H^1(\om))}+\norm{\widehat V_\alpha}_{L^2(0,\infty;L^2(\om;\d\normmm{\mu_{\alpha}^{(\cdot)}}))}\sup_{t\geq0}\sqrt{\normmm{\mu_{\alpha}^t}_{\sf}(\om)}\norm{\zeta}_{L^2(0,\infty;C(\om))}\\&\quad +\alpha\norm{\widehat V_\alpha}_{L^2(0,\infty;H^1(\om))}\norm{\zeta}_{L^2(0,\infty;H^1(\om))}\\&
     \leq C(E[\xhi_0],\norm{\bv_0}_{\bL^2(\om)})\norm{\xi}_{L^2(0,\infty;H^2(\om)},
 \end{align*}
 where we used the embedding $H^2(\om)\hookrightarrow C(\om)$, together with Poincaré's inequality to obtain $\norm{\bv_\alpha}_{L^2(0,\infty;\dot\bL_\sigma^2(\om))}\leq C\norm{\bv_\alpha}_{L^2(0,\infty;\dot\bH^1_\sigma(\om))}$. Therefore, we infer that 
 \begin{align}
 \norm{\pt\chi_\alpha}_{L^2(0,\infty;H^2(\om)')}\leq C,
     \label{dtxhi}
 \end{align}
 for some $C>0$ uniform in $\alpha$. As a consequence, recalling also \eqref{energia1}, we can apply Aubin-Lions theorem to obtain, up to subsequences,
 \begin{align}
    \xhi_\alpha\to \xhi\text{ strongly in }L^p((0,\Ts)\times\om;\{0,1\})\quad\forall \Ts>0 \label{xhialpha}
\end{align}
as $\alpha\to0$, for some $\xhi\in L^\infty(0,\infty;\{0,1\})$. Then, by weak compactness, from \eqref{dtxhi} we deduce
 \begin{align}
     \label{weakcomp}\pt\xhi_\alpha\rightharpoonup\pt\xhi\quad\text{weakly in }L^2(0,\infty;H^2(\om)').
 \end{align}
Moreover, arguing as in \eqref{regchi1a}, we get 
$$
\xhi\in C([0,\Ts];L^p(\om))\quad\forall p\in[1,\infty),\quad \forall \Ts>0,
$$
as well as 
$$
\xhi\in L^\infty_{w*}(0,\infty;BV(\om;\{0,1\})).
$$
On the other hand, about the fluid equation \eqref{weakafterlim}, we first notice that, by \eqref{variation2} and \eqref{variation4}, we have
\begin{align}
    \int_\om \xhi_\alpha \nabla w_\alpha\cdot\bzeta\dx=\delta\mu_{t,\alpha}(\bzeta),\quad \forall \bzeta\in C^1(\om;\R^d)\quad\text{ with }\Div\bzeta=0.\label{firstvariation}
\end{align}
Therefore, concerning the time derivative of the fluid velocity, we obtain from \eqref{weakafterlim} that, since $\dot\bH^3_\sigma(\om)\hookrightarrow C^1(\om;\R^d)$, 
\begin{align*}
    &\normmm{\int_0^\Ts\langle \pt\bv_\alpha,\bzeta\rangle_{\dot\bH_\sigma^3(\om)',\dot\bH^3_\sigma(\om)}}\\&\leq \norm{\bv_\alpha}_{L^2(0,\Ts;\dot\bL_\sigma^2(\om))}\norm{\bv_\alpha}_{L^2(0,\infty;\dot\bH_\sigma^1(\om))}\norm{\bzeta}_{L^2(0,\Ts;\bL^\infty(\om))}+C\norm{\bv_\alpha}_{L^2(0,\infty;\Hs)}\\&\quad +\norm{\nabla \bzeta}_{L^2(0,\Ts;\mathbf C(\om))}\sup_{t\geq0}\sqrt{\normmm{\mu_{\alpha}^t}_{\sf}(\om)}\sqrt\Ts,
\end{align*}
for any $\bzeta\in L^2(0,\Ts;\dot\bH^3_\sigma(\om))$ and any $\Ts>0$, entailing by \eqref{energia1}-\eqref{velox} that
\begin{align*}
\norm{\pt\bv_\alpha}_{L^2(0,\Ts;\dot\bH^3_\sigma(\om)')}\leq C(\Ts)\quad\forall \Ts>0.
\end{align*}
This, together with \eqref{velox}, observing that $\Hs\hookrightarrow\hookrightarrow \dot\bL^2_\sigma(\om)\hookrightarrow \bH^3_\sigma(\om)'$, allows to apply Aubin-Lions theorem, obtaining
\begin{align}
    \label{strongbella}\bv_\alpha\to\bv\quad&\text{ strongly in }L^2(0,\Ts;\dot\bL^2_\sigma(\om)),\\
    \pt\bv_\alpha\rightharpoonup \pt\bv\quad&\text{weakly in }L^2(0,\Ts;\bH^3_\sigma(\om)')\label{dtweak}
\end{align}
as $\alpha\to0$ for any $\Ts>0$. One of the fundamental consequences of \eqref{strongbella} is that, up to a further subsequence, it holds
\begin{align}
\frac12\norm{\bv_\alpha(t)}_{\bL^2(\om)}\to \frac12\norm{\bv(t)}_{\bL^2(\om)}^2\quad\text{ for a.a. }t\geq0\label{aeconv}
\end{align}
as $\alpha\to0$, which gives, recalling \eqref{Hellybis},
\begin{align}
E[\mu_{\alpha}^t]\to E(t):=\cZ(t)-\onehalf\norm{\bv(t)}_{\bL^2(\om)}^2\quad \text{as }h\to 0,\quad \text{ for a.a. } t\geq 0.\label{Hellytris}
 	\end{align}
Now, thanks to the uniform bound \eqref{energia1}, there exists ${\wtilde\mu}:{\mathcal B((0,\infty)	\times{\mathbb T^d}\times \mathbb S^{d-1})}\to \mathbb R^+$, such that, up to subsequences, 
\begin{align}
	\mathcal L^1\llcorner{(0,\Ts)}\otimes (\mu_\alpha^t)_{t\in(0,\Ts)}\rightharpoonup {\widehat\mu}_{\llcorner (0,\Ts)\times \om\times \sf},\quad \text{ weakly* in } \mathcal M((0,\Ts)\times{\mathbb T^d}\times \mathbb S^{d-1}),\label{convmeas12}
\end{align}
 for any $\Ts>0$. Then, recalling \eqref{Hellytris}, we can perform a classical slicing argument, letting $\eta \in C_c((0,\infty))$ with $\eta\geq 0$ and $\zeta \in C_c(\om\times \mathbb S^{d-1})$ with $\zeta \in [0, 1]$. Then we have
\begin{align}	
&\nonumber\int_{(0,\infty)\times \om\times \mathbb S^{d-1}}\zeta(x,s)\eta(t)\d 	\left(\mathcal L^1\llcorner{(0,\infty)}\otimes (\mu_{\alpha}^t)_{t\in(0,\infty)}\right)\nonumber\leq \int_0^\infty \eta(t)E[\mu_{\alpha}^t]\d t,
\end{align}
and thus, passing to the limit using, for the right-hand side, Lebesgue's Dominated Convergence theorem, we infer
\begin{align}	
&\nonumber\int_{(0,\infty)\times \om\times \mathbb S^{d-1}}\zeta(x,s)\eta(t)\d 	\widehat \mu\nonumber\leq \int_0^\infty \eta(t)E(t)\d t.
\end{align}
 As the right-hand side is finite for any nonnegative $\eta \in C_c((0, \infty))$, there
exists a Radon measure $\gamma$ on $(0, \infty)$ and a weakly-* $\gamma$-measurable family of Radon probability measures $(\tau_t)_t$ on $\om\times\mathbb S^{d-1}$, so that $\widehat\mu = \gamma\otimes(\tau_t)_t$ and $\gamma\leq E(t)\dt$. This means that there exists an $\mathcal L^1$-measurable function $\widetilde\gamma$ on $(0, \infty)$ with $\widetilde\gamma(t) \leq E(t)$ for almost any $t \in (0, \infty)$ such that $\gamma = \widetilde{\gamma}\mathcal L^
1\llcorner(0, \infty)$, giving
\begin{align}
	\widehat{\mu}=\mathcal L^1\llcorner(0,\infty)\otimes (\mu_t)_{t\in(0,\infty)},
	\label{representationA}
\end{align}
with $\mu_t=\widetilde{\gamma}(t)\tau_t$.
Additionally, following the argument to obtain \eqref{identif1}, we can deduce also here that $E(t)=E[\mu_t]$ for almost any $t>0$, and thus 
\begin{align}
    \lim_{\alpha\to 0} E[\mu_{\alpha}^t]=E[\mu_t],\quad \text{ for a.a. }t>0.\label{Emu}
\end{align}
Proceeding in the convergence argument, by \eqref{Vhat} we have
\begin{align}
    \alpha \widehat V_\alpha\to 0,\text{ strongly in }L^2(0,\infty;H^1({\mathbb T^d})).\label{strongto0}
\end{align}
Additionally, recalling \eqref{energiagiusta} and \eqref{convmeas12}
we can apply a \textcolor{black}{result of Hutchinson} \cite[Theorem 4.4.2]{Hutchinson} to infer that the measure-function pairs $(\mathcal L^1{\llcorner}(0,\infty)\otimes\normmm{\mu^t_\alpha}_{\mathbb S^{d-1}})_{t\in(0,\Ts)},\widehat V_\alpha)$ converge in weak-* sense, up to a further subsequence, to a pair $(\mathcal L^1\llcorner(0,\infty)\otimes (\normmm{\mu_t}_{\mathbb S^{d-1}})_{t\in(0,\Ts)},V(t))$, i.e.,
\begin{align}
  &  \int_0^\infty\int_{\mathbb T^d} (\varphi \widehat V_\alpha)\d\normmm{\mu^t_\alpha}_{\mathbb S^{d-1}} \dt\to \int_0^\infty\int_{\mathbb T^d} \varphi V \d  \normmm{\mu_t}_{\mathbb S^{d-1}}\dt,\label{uhlim}
\end{align}
for any $\varphi\in C_c((0,\infty)\times\om)$, as well as it holds 
\begin{align}
    &\int_s^{T'}\int_{\mathbb T^d} \normmm{V(t)}^2\d  \normmm{\mu_t}_{\mathbb S^{d-1}}\dt\leq \liminf_{\alpha\to0}\int_{s}^{T'}\int_{\mathbb T^d} \normmm{\widehat V_\alpha}^2\d\normmm{\mu^t_\alpha}_{\mathbb S^{d-1}}\dt
    \label{lsc0}
\end{align}
for any $0\leq s<T'$. Therefore, since $\frac{1}{2(1+4\alpha^\frac12)}\to \frac12$, we clearly get from \eqref{lsc0} that
\begin{align}
    &\frac12\int_s^{T'}\int_{\mathbb T^d} \normmm{V(t)}^2\d  \normmm{\mu_t}_{\mathbb S^{d-1}}\dt\leq \liminf_{\alpha\to0}\frac1{2(1+4\alpha^\frac12)}\int_s^{T'}\int_{\mathbb T^d} \normmm{\widehat V_\alpha}^2\d\normmm{\mu^t_\alpha}_{\mathbb S^{d-1}}\dt
    \label{lsc}
\end{align}
for any $0\leq s<T'$. Now the quantity $$V\in L^2((0,\infty)\times {\mathbb T^d};\mathcal L^1{\llcorner}(0,\infty)\otimes (\tmmu)_{t\in(0,\infty)})=L^2(0,\infty;L^2({\mathbb T^d};\normmm{\mu_{(\cdot)}}_{\sf}))$$ can be interpreted as the total material interface (normal) velocity, in the final limit as $\alpha \to 0$. Indeed, also recalling the strong convergences \eqref{xhialpha}, \eqref{weakcomp},  \eqref{strongbella}, \eqref{strongto0}, and \eqref{uhlim}, we can pass to the limit in \eqref{base11} and deduce 
\begin{align*}
    &\langle \pt^\bullet\xhi,\zeta\rangle=\int_0^\infty\langle \pt\xhi,\zeta\rangle_{H^2(\om)',H^2(\om)}\dt-\int_0^\infty \int_\om\xhi\bv\cdot\nabla \zeta\dx\dt\\&
    =\int_0^\infty \int_\om V\zeta\d\normmm{\mu_t}_{\sf}\dt\quad \forall \zeta\in C_c^\infty((0,\infty)\times\om),
\end{align*}
confirming the interpretation of $V$ as the total material interface velocity (cf. \eqref{materialvelox}). Therefore, \eqref{kinetic} holds, up to choosing $-V$ as material normal velocity.
\subsubsection{Limit fluid equations}
Thanks to \eqref{v1a}-\eqref{v2a}, and \eqref{dtweak}, together with \eqref{xhialpha}, and using \eqref{firstvariation} and the convergence \eqref{convmeas12}, we can pass to the limit as $\alpha\to 0$ in \eqref{weakafterlim}, to obtain
\begin{align}
&\int_0^\Ts\int_\om \langle \pt\bv,\bzeta\rangle_{\dot\bH^3_\sigma(\om)',\dot\bH^3_\sigma(\om)}\dx\dt-\int_0^\Ts\int_\om \bv\otimes\bv:\nabla \bzeta\dx\dt+\int_0^\Ts\int_\om\nu(\xhi)\nabla\bv:\nabla\bzeta\dx\dt\nonumber\\&= -\int_0^\Ts\delta\mu_t(\bzeta)\dt,
    \label{weakafterlim2}
\end{align}
for any $\bzeta\in L^2(0,\Ts;\dot\bH_\sigma^3(\om))$ and any $\Ts>0$, which gives \eqref{NavS}.
\subsubsection{Energy dissipation inequality as $\alpha\to0$} 
Notice now that by construction we have 
\begin{align}\label{slope1}
    \onehalf\int_s^{T'} \normmm{\partial E(\mu^t_\alpha)}_{{\mathcal V_t}}^2\dt\geq \int_s^{T'}\delta \mu_\alpha^t(\xi(t))\dt-\frac 12 \int_s^{T'} \norm{\xi(t)}^2_{{\mathcal V_t}}\dt,\quad\forall{ \xi\in C^1_c((s,T');C^1(\om;\R^d))},
\end{align}
for any $0\leq s<T'$. Thanks to \eqref{convmeas12} we can then let $\alpha \to 0$ and infer 
\begin{align}\label{slope1a}
    &\liminf_{\alpha\to0}\onehalf\int_s^{T'} \normmm{\partial E(\mu^t_\alpha)}_{{\mathcal V_t}}^2\dt\nonumber\\&\geq \int_s^{T'}\delta \mu_t(\xi(t))\dt-\frac 12 \int_s^{T'} \norm{\xi(t)}^2_{{\mathcal V_t}}\dt,\quad\forall{ \xi\in C^1_c((s,T') \times \om)},
\end{align}
so that, taking the supremum over $\xi\in C^1_c((s,T') \times \om)$, entails 
\begin{align}\label{slope1bb}
    \liminf_{\alpha\to0}\onehalf\int_s^{T'} \normmm{\partial E(\mu^t_\alpha)}_{{\mathcal V_t}}^2\dt\geq \onehalf\int_s^{T'} \normmm{\partial E(\mu_t)}_{{\mathcal V_t}}^2\dt
\end{align}
for any $0\leq s<T'$.

Therefore, we can finally pass to the limit in the energy inequality \eqref{finaleenergy}.
Indeed, thanks to \eqref{v2a}, \eqref{strongbella}, \eqref{Emu}, \eqref{lsc}, and \eqref{slope1bb},  recalling that 
\begin{align*}
  &0\leq \frac{\alpha}{2(1+4\alpha^\frac12)}\int_{s}^{T'}\left(\normh{\widehat V(t)	}^2+\normh{\nabla \widehat V(t)}^2\right)\dt\nonumber\\&\quad \quad +\frac{\alpha^2}{2}\int_{s}^{T'}(\normh{\tw(t)}^2+\normh{\nabla\tw(t)}^2)\d t,
\end{align*}
we can take the liminf as $\alpha\to0$ in the energy inequality  \eqref{finaleenergy} and finally deduce the validity of \eqref{advectiveMullins}. This concludes the existence proof of Theorem \ref{thm1}, as the admissibility of the couple $(\xhi,\mu)$ can be verified as in Section \ref{admissibilityA}.
\section{Proof of Theorem \ref{theo:weakStrongUniqueness}
}\label{proofws}
First, we assume from now on that the generic constant $C>0$ may vary from line to line.
To perform the proof, we need to introduce a suitable gradient-flow calibration for the strong solution. In particular, we have the following lemma, which can be proven similarly as in \cite[Lemma 2.2]{AbelsFischerMoser}. Indeed, also in this case we assume to include in the definition of the extended material interface velocity $B$ the entire contribution of the fluid velocity $\bu$, and not only the extension of its normal component along the interface as done, for instance, in \cite[Section 6]{HenselLiu}. As it will also be remarked during the proof, this is crucial to allow the difference $B-\bu$ to be normal to the fluid interface, which is false in the calibration proposed in \cite{HenselLiu}. The lemma thus guarantees the existence of smooth vector fields $(\xi,B)$ with the following interpretation: the vector field $\xi$ is a suitable extension of the unit normal vector field $\bn_\mathcal I$ for the smoothly evolving
sharp interface $\cI$. On the other hand, $B$ denotes a suitable extension of the associated velocity vector $V_\II$. Of course, this choice of $B$ does not allow to obtain that $B$ is compactly supported, but this is not needed in our computations, as we have that $B-\bu$ is compactly supported, which is enough for our purposes. Then $\normmm{\vartheta}$ is the standard smooth version of $(x,t)\mapsto \min\{1,{\frac1c}\ \dist(x,\cI(t))\}$, for some $c>0$ suitably chosen.

\begin{lemma}(Existence of a calibration for Navier-Stokes-Mean Curvature Flow)\label{lemma:calibration}\ \ \ \ \  \
Let $(\xhi_\mA,\bu)$ be a strong solution on $[0,\Tstrong]$ according to Definition \ref{def:strongSolSharpInterfaceLimit}. There exist $\xi,B:\om\times[0,\Tstrong]$, $\vartheta:\om\times[0,\Tstrong]\to\R$, so that $(\xi,B,\vartheta)$ is called a calibration for the Navier-Stokes-Mean Curvature Flow system, satisfying the following properties.
	\begin{itemize}
		\item [(i)]The vector fields $\xi,B$ and the scalar function $\vartheta$ are such that 
		\begin{align}
	&	\label{reg1}\xi\in C^{0,1}(\om\times[0,\Tstrong];\R^d)\cap L^\infty(0,\Ts;C^{1,1}(\om)),\\&
        B\in L^\infty(0,\Ts;C^{0,1}(\om;\R^d)\cap C^{1}(\om;\R^d)),
        \\& \vartheta\in C^{0,1}(\om\times[0,\Tstrong])\cap L^\infty(\om\times[0,\Tstrong]),\label{controltheta}
		\end{align}
		with a bound on the highest order derivatives of the form
	\begin{align}
\label{eq:conditionsXiB3}
\|(\partial_t\xi,\nabla^2\xi,\nabla B)\|_{L^\infty(\om{\times}(0,T))} &\leq C,
\end{align}
		as well as, for any $t\in[0,\Tstrong]$,
		\begin{align}
			&\supp \ \xi(\cdot,t)={\{x\in \om:\ \dist(x,\cI(t))\leq c\}},\label{supports}\\&
            \supp (B(\cdot,t)-\bu(\cdot,t))\subset {\{x\in \om:\ \dist(x,\cI(t))\leq 2c\}},\label{supports2}
		\end{align}
        for some $c>0$, where $\bu$ is the fluid velocity. 

		\item[(ii)] (Normal direction of $B-\bu$ on the interface).  The vector field $B$ satisfies, {for any $(x,t)\in\om\times [0,\Tstrong]$ such that $\dist(x,\cI(t))\leq 2c$},
		\begin{align}
			\label{normal}(Id-\nabla s_\cI\otimes \nabla s_\cI) (B(\cdot,t)-\bu(\cdot,t))=O(\dist(\cdot,\cI(t))),
		\end{align}
		where $s_\cI(\cdot,t)$ is the signed distance function with respect to $\cI(t)$, which is well defined over $\supp\ B(\cdot,t)$, and where the constants are independent of $t\in[0,\Tstrong]$.
		\item[(iii)] (Normal extension and shortness). The vector field $\xi$ extends the (interior) unit normal vector field of $\cI$, i.e., 
		\begin{align}
		\xi(\cdot,t)=\bn_\cI(\cdot,t),\qquad \nabla\cdot \xi = -H_{\II},\quad\text{ on }\cI(t),
		\label{identityn}
		\end{align}
        \begin{align}
        \xi(x,t)=-\normmm{\xi(x,t)}\nabla s_\cI(x,t), \quad \forall (x,t)\in\om\times[0,\Tstrong],
            \label{xiprop}
        \end{align}
		and there exists a constant $c>0$ such that
		\begin{align}
			\label{short}
			\min\{1,\frac1{c^2}\ \dist(x,\cI(t))^2\}\leq 1-\normmm{\xi(x,t)},\quad \forall (x,t)\in\om\times[0,\Tstrong].
		\end{align}
		\item[(iv)] (Approximate transport equations). The weight $\vartheta$ is approximately transported by $B$:
		\begin{align}
			\label{thetabasic}
			(\pt\vartheta+(B\cdot\nabla)\vartheta)(\cdot,t)=O\left(\min\{1,\frac1 c\dist(\cdot,\cI(t))\}\right),\quad \forall t\in[0,\Tstrong],
		\end{align}
		and also the length of $\xi$ is approximately transported by the flow:
		\begin{align}
			\label{transport}
			(\pt\normmm{\xi}^2+(B\cdot\nabla)\normmm{\xi}^2)(\cdot,t)=2\xi(\cdot,t)\cdot(\pt{\xi}+(B\cdot\nabla){\xi})(\cdot,t)=O(\dist(\cdot,\cI(t))^2),
		\end{align}
		{for any $(x,t)\in\om\times [0,\Tstrong]$ such that $\dist(x,\cI(t))\leq c$}, where the constants are independent of $t\in[0,\Tstrong]$.
		Additionally, there exists a constant $C>0$ such that the vector field $\xi$ is almost transported by the flow $B$, i.e.,
		\begin{align}
			\label{almosttranported}
			(\pt\xi+(B\cdot \nabla)\xi+(I-\xi\otimes \xi)(\nabla B)^T\xi)(\cdot,t)=O(\dist(\cdot,\cI(t))),
		\end{align} 
			{for any $(x,t)\in\om\times [0,\Tstrong]$ such that $\dist(x,\cI(t))\leq c$}, where the constants are independent of $t\in[0,\Tstrong]$.
		\item[(v)] (Geometric evolution equation). It holds the approximated geometric evolution
		\begin{align}
			\label{geometric}
			(B(\cdot,t)-\bu(\cdot,t))\cdot\xi(\cdot,t)+\Div\xi(\cdot,t)=O(\dist(\cdot,\cI(t))),
		\end{align}
       with $\bu$ the fluid velocity,
		 {for any $(x,t)\in\om\times [0,\Tstrong]$ such that $\dist(x,\cI(t))\leq c$}, and the constants are independent of $t\in[0,\Tstrong]$.
		\item[(vi)] (Sign and coercivity of the transported weight). It holds
		\begin{align*}
	     &\vartheta(\cdot,t)<0,\quad\text{ in }\mA(t),\\&
	     \vartheta(\cdot,t)>0,\quad\text{ in }\om\setminus\mA(t).
		\end{align*}
		Moreover, there exists a constant $C_\vartheta>0$, such that
		\begin{align}
			\min\{1,\frac1c\ \dist(x,\cI(t))\}\leq \normmm{\vartheta(x,t)}\leq C_\vartheta \min\{1,\frac1c\ \dist(x,\cI(t))\}.
			\label{varthetacontrol}
		\end{align} 
	\end{itemize}
	\end{lemma}
    \begin{proof}
        Most of the proof can be found in the proof of \cite[Lemma 2.2]{AbelsFischerMoser} (cf. also \cite{HenselLiu}). We thus only highlight the main differences. Let us fix $T \in [0,\Tstrong]$. Thanks to the regularity of the strong solutions~$(\chi_\mA,\bu)$,
there exists $\delta=\delta(\chi_\mA,T) \in (0,\smash{\frac{1}{2}}]$
such that the signed distance~$s_{\II}(\cdot,t)$ to~$\II(t)$ as well as the
nearest point projection~$P_{\II}(\cdot,t)$ onto~$\II(t)$, $t \in [0,T]$,
are well defined and smooth. In particular (cf. \cite{HenselLiu})
\begin{align}
\label{eq:proofMainResultAux1}
s_{\II} &\in (C^1_tC^{1}_x \cap C_tC^{3}_x)\big(\overline{B_{2\delta}(\II)}\big),
\\
\label{eq:proofMainResultAux2}
\big((x,t) \mapsto P_{\II}(x,t) := (x {-} (s_{\II} \nabla s_{\II})(x,t),t)\big) 
&\in (C^1_tC_x {\cap} C_tC^{2}_x)\big(\overline{B_{2\delta}(\II)}\big).
\end{align}
Note that the signed distance is such that $\nabla s_{\II} = \bn_{\II}$ on~$\II$, and
that for any $r \in (0,1]$ the corresponding space-time tubular neighborhood~$B_r(\II)$
of~$\II$ is defined as
\begin{align}
\label{eq:proofMainResultAux3}
B_r(\II) := \bigcup_{t \in [0,T]} B_r(\II(t)) \times \{t\}. 
\end{align}
Now, following \cite{HenselLiu, AbelsFischerMoser}, let $\bar\eta\colon\mathbb{R}\to[0,1]$ be a smooth
and even function with $\supp\ \bar\eta \subset [-1,1]$ such that
\begin{align}
\label{eq:proofMainResultAux5}
c_{\bar\eta} r^2 \leq 1 - \bar\eta(r) \leq C_{\bar\eta} r^2
\quad\forall r \in [-1,1],
\end{align}
for some $0<c_{\bar\eta}<C_{\bar\eta}<\infty$. 
We also choose a smooth and even function~$\widetilde\eta\colon\mathbb{R}\to[0,1]$
such that $\supp\ \widetilde\eta\subset [-2,2]$ and
\begin{align}
\label{eq:proofMainResultAux13}
\widetilde\eta \equiv 1 \quad\text{on } [-1,1].
\end{align}
We then define the cutoff functions
\begin{align}
\label{eq:proofMainResultAux6}
\eta_{\II}(x,t) &:= \bar\eta\big(s_{\II}(x,t)/\delta\big),
&&(x,t) \in \Rd{\times}[0,T],
\\
\label{eq:proofMainResultAux14}
\widetilde\eta_{\II}(x,t) &:= \widetilde\eta\big(s_{\II}(x,t)/\delta\big),
&&(x,t) \in \Rd{\times}[0,T],
\end{align}
and clearly it holds
\begin{align}
\label{eq:proofMainResultAux12}
\supp\eta_{\II}(\cdot,t) &\subset \overline{B_{\delta}(\II(t))}
&&\forall t \in [0,T].
\\
\label{eq:proofMainResultAux15}
\supp\widetilde\eta_{\II}(\cdot,t) &\subset \overline{B_{2\delta}(\II(t))}
&&\forall t \in [0,T],\\
\label{eq:proofMainResultAux16}
\widetilde\eta_{\II}(\cdot,t) &\equiv 1 
&&\text{on } \overline{B_{\delta}(\II(t))},\ \forall t \in [0,T].
\end{align}
We then define the vector fields~$\xi$ and~$B$ similarly to \cite[Lemma 2.2]{AbelsFischerMoser}. Namely we set
\begin{align}
\label{eq:proofMainResultAux7}
\xi &:= \eta_{\II} \nabla s_{\II}
&&\text{on } \om \times [0,T],
\\
\label{eq:proofMainResultAux8}
B &:= \bu
- ((\Delta s_{\II}) \circ P_{\II}\big) \,\widetilde\eta_{\II}\nabla s_{\II}
&&\text{on } \om \times [0,T],
\end{align}
where~$\bu$ is the velocity field of the strong solution. Note that this differs from the calibration adopted in \cite{HenselLiu}, which only accounts for $\big((\bu \circ P_{\II}) \cdot \nabla s_{\II}\big)\widetilde\eta_{\II}\nabla s_{\II}$ in place of the full fluid velocity $\bu$. Finally, the construction of a suitable weight~$\vartheta$ is standard. We 
first fix a smooth and odd map $\bar\vartheta\colon\mathbb{R}\to [-1,1]$ such that
\begin{align}
\label{eq:proofMainResultAux9}
&\bar\vartheta \equiv 1 \text{ on } (-\infty,-1],\quad \bar\vartheta \equiv -1 \text{ on } [1,\infty),
\\
\label{eq:proofMainResultAux9b}
&\bar\vartheta > 0 \text{ in } (-1,0),\quad \bar\vartheta < 0 \text{ in } (0,1),
\\
\label{eq:proofMainResultAux10}
&c_{\bar\vartheta}r \leq |\bar\vartheta(r)| \leq C_{\bar\vartheta} r,
\quad\forall r \in [-1,1],
\end{align}
for some $0<c_{\bar\vartheta}<C_{\bar\vartheta}<\infty$, and we set
\begin{align}
\label{eq:proofMainResultAux11}
\vartheta &:= \bar\vartheta\big(s_{\II}/\delta\big)
&&\text{on } \om \times [0,T].
\end{align}
Now the rhe required regularity~\eqref{reg1}-\eqref{eq:conditionsXiB3}
is a trivial consequence of  definitions~\eqref{eq:proofMainResultAux7}
and~\eqref{eq:proofMainResultAux8} together with the
regularity~\eqref{eq:conditionsStrongSolSharpInterfaceLimit4}, 
\eqref{eq:proofMainResultAux1} and~\eqref{eq:proofMainResultAux2}. Furthermore, because the presence of  cutoff functions $\ov\eta,\widetilde\eta$, the 
vector field~$\xi$ and~$B-\bu$ are indeed compactly supported within~$\om$, so that \eqref{supports}-\eqref{supports2} hold with $c=\delta$. Note that this is not true if we consider the entire vector field $B$. Then, \eqref{normal} comes from 
$$
B-\bu=- ((\Delta s_{\II}) \circ P_{\II}\big) \,\widetilde\eta_{\II}\nabla s_{\II}
$$
for any $x\in \overline{B_{2\delta}(\II(t))}$ and any $t\in[0,T]$,
entailing 
$$
(I-\nabla s_\II\otimes \nabla s_\II)(B-\bu)=0
$$
for any $x\in \overline{B_{2\delta}(\II(t))}$ and any $t\in[0,T]$, i.e.,  \eqref{normal} with even the right-hand side equal zero. Since for our purposes \eqref{normal} is enough (and this is what we expect to hold in other cases, cf. Volume Preserving Mean Curvature Flow \cite{Poiatti}), we keep the condition allowing the distance from the interface in the right-hand side. All the other properties \eqref{identityn}-\eqref{varthetacontrol} can now be shown following the arguments of the proof of \cite[Lemma 2.2, Lemma 2.4]{AbelsFischerMoser}. In particular, \eqref{thetabasic} is just a consequence of identity \cite[(2.24)]{AbelsFischerMoser} and the definition of $\vartheta$.
    \end{proof}
We can now rigorously construct the relative entropy functional and the bullk error functional measuring the distance between a (Giorgi-type) varifold solution $(\xhi_{A},\mu,\bv)$ according to Definition \ref{weaksol} (where we set $A(t):=\{x\in \om:\ \xhi(t)=1\}$ and $\bn:=\frac{\nabla \xhi_{A(t)}}{\normmm{\nabla \chi_{A(t)}}}$), and a strong solution $(\xhi_\mA,\bu)$ to the Navier-Stokes/Mean Curvature Flow system according to Definition \ref{def:strongSolSharpInterfaceLimit}, which admits a calibration by means of Lemma \ref{lemma:calibration}. As introduced in \eqref{eq:relativeEntropyIntro} and \eqref{eq:bulkErrorIntro}, we set the relative entropy functional 
	\begin{align}
	\label{eq:relativeEntropy1}
	\Er := \int_{\om {\times} \S} 
	1 - p \cdot \xi(x,t) \d\mu_t(x,p) \geq 0, 
	\quad t \in [0,\Tstrong],
\end{align}
as well as the bulk error
\begin{align}
	\label{eq:bulkError1}
	\Eb := \int_{\om} 
	\normmm{\chi_{A(t)}(x) {-} \chi_{\mA(t)}(x)} \normmm{\vartheta(x,t)}\dx , 
	\quad t \in [0,\Tstrong].
\end{align}	
We recall now some coercivity properties of the relative entropy functional (cf., e.g., \cite{FischerHensel}). First, by the sign properties of $\vartheta$ we have
\begin{align}
	\label{eq:bulkError2}
	\Eb = \int_{\om} 
	({\chi_{A(t)}(x) {-} \chi_{\mA(t)}(x)}) {\vartheta(x,t)}\dx \geq 0, 
	\quad t \in [0,\Tstrong].
\end{align}	

Moreover, it holds (see, e.g., \cite[Section 4]{LH}, \cite{Poiatti, FischerHensel})
\begin{align}
	\label{coearciv1}
	\int_{\om {\times} \S} \normmm{p {-} \xi(\cdot,t)}^2 \d\mu_t 
	&\leq 2\Er,
	 \\ \label{coerciv2}
	 \int_{\om} \min\{1,\frac1{c^2}\ \dist(\cdot,\mathcal{I}(t))^2\} \d\normmm{\mu_t}_{\sf}
	 &\leq 2\Er,
	\\ \label{coerciv4}
	\int_{\om} \normmm{\chi_{A(t)} {-} \chi_{\mA(t)}}
	 \min\{1,\frac1c\ \dist(\cdot,\mathcal{I}(t))\} \dx
	 &\leq \Eb.
\end{align}
The relative entropy functional also admits equivalent expressions. For instance, by means of the compatibility condition~\eqref{PE} for the varifold solution, together with an integration by parts, one immediately obtains
\begin{align}
	\Er
	\label{eq:relEntropy2}
	= \int_{\om} 1 \d\normmm{\mu_t}_{\sf} + \int_{\om} \chi_A(\cdot,t)
	(\Div \xi)(\cdot,t) \dx.
\end{align}
We now introduce the Radon-Nikod\'ym derivative of $\normmm{\nabla\chi_{A(t)}}$ with respect to $\normmm{\mu_t}_{\sf}$, namely we write
\begin{align}
\rho:=\frac{\d\normmm{\nabla \xhi_{A(t)}}}{\d\normmm{\mu_t}_{\sf}},
    \label{RNder}
\end{align}
and, adding and subtracting $\rho$ to the definition of relative entropy, and exploiting again the compatibility \eqref{PE}, we get
\begin{align}
	\label{relEntropy3}
	\Er
	&= \int_{\om} 1 - \rho(\cdot,t) \d\normmm{\mu_t}_{\sf} 
	+ \int_{\p^*A(t)} 1 - \bn(\cdot,t)\cdot\xi(\cdot,t) \d\H^{d-1}.
\end{align}
Therefore, we also obtain further coercivity properties:
\begin{align}
	\label{coerciv5}
	\int_{\om} 1 - \rho(\cdot,t) \d\normmm{\mu_t}_{\sf}
	&\leq \Er,\\
		\label{coerciv5b}
	\int_{\p^*A(t)} 1 - \bn(\cdot,t)\cdot \xi(\cdot,t)  \d\H^{d-1} 
	&\leq \Er,
	\\ \label{coerciv6}
	\int_{\p^*A(t)} \normmm{\bn(\cdot,t){-}\xi(\cdot,t)}^2 \d\H^{d-1}
	&\leq 2\Er,
	\\ \label{coerciv7}
	\int_{\p^*A(t)} \min\{1,\frac1{c^2}\ \dist(\cdot,\mathcal{I}(t))^2\}\d\H^{d-1}\leq \int_{{\p^*A(t)}} (1-\normmm{\xi})\d\H^{d-1} 
	&\leq 2\Er.
\end{align}
{We also observe that, by the definition of $\rho\in[0,1]$ and the fact that $\xi\cdot\bn\leq 1$ over $\partial^*A(t)$, we have
\begin{align}
\nonumber&\int_{\om\setminus\supp\xi(t)} 1\d\normmm{\mu_t}_{\sf}\\&\leq \int_{\om\setminus\supp\xi(t)} 1\d\normmm{\mu_t}_{\sf}+\int_{\supp\xi(t)} (1-\rho(\cdot,t))\d\normmm{\mu_t}_{\sf}+\int_{\p^*A(t)\cap \supp\xi(t)}(1-\bn(\cdot,t)\cdot\xi(\cdot,t))\d\cH^{d-1}\nonumber\\&
= \Er.
\label{necessaria}
\end{align}
}
Finally, by following the same argument as to prove \cite[Lemma 20]{FischerHensel}, we deduce the following fundamental inequality, for any $\zeta\in \bL^\infty(\om\times(0,T))$, any $T\in(0,\Tstrong]$ and any $\delta\in(0,1]$,
\begin{align}
   &\nonumber \normmm{\int_0^T\int_\om ({\xhi_{A(t)}-\xhi_{\mA(t)}})({\bv(t)-\bu(t)})\cdot\zeta(t)\dx\dt}\\&\nonumber
   \leq  \delta\int_0^T\int_\om\normmm{\nabla(\bu(t)-\bv(t))}^2\dx\dt+\frac{C}{\delta}(1+\norm{\zeta}_{\bL^\infty(\om\times(0,T))})\int_0^T\int_\om \normmm{\bu(t)-\bv(t)}^2\dx\dt\\&
   \quad +\frac{C\norm{\zeta}_{\bL^\infty(\om\times(0,T))}}\delta\int_0^T\Eb\dt\label{inequcmp}.
\end{align}

Now, starting from the fluid equation for the strong solution $\bu$, we can choose as as test function in \eqref{eq:conditionsStrongSolSharpInterfaceLimit7}, after suitable mollification arguments, the function $\eta=\bv-\bu$, obtaining, after an integration by parts in time, with $T\in(0,\Tstrong]$,
\begin{align}
   &\nonumber \int_0^T\int_\om \pt\bu\cdot(\bv-\bu)\dx\dt+\int_0^T\int_\om (\bu\cdot\nabla)\bu\cdot (\bv-\bu)\dx\dt+\int_0^T\int_\om D\bu:D(\bv-\bu)\dx\dt\\&=\int_0^T\int_\om \nabla(\Div \xi)\cdot (\bv-\bu)\xhi_{\mA}\dx\dt, \label{u-v}
\end{align}
where in the right-hand side of the identity above we used 
\begin{align*}
  &\int_0^T\int_{\om} \frac{\nabla\xhi_\mA}{|\nabla\xhi_\mA|} 
					 \otimes \frac{\nabla\xhi_\mA}{|\nabla\xhi_\mA|} : \nabla(\bv-\bu) \,d|\nabla\xhi_\mA| dt\\&=-\int_0^T\int_{\om} \left(I-\frac{\nabla\xhi_\mA}{|\nabla\xhi_\mA|} 
					 \otimes \frac{\nabla\xhi_\mA}{|\nabla\xhi_\mA|}\right) : \nabla (\bv-\bu) \,d|\nabla\xhi_\mA|\\& 
                     = \int_0^T\int_\om H_\II\bn_\II\cdot(\bv-\bu)\d\normmm{\nabla \xhi_\mA}\dt=-\int_0^T\int_\om\Div \xi\bn_\II\cdot(\bv-\bu)\d\normmm{\nabla \xhi_\mA}\dt\\& =\int_0^T\int_\om \Div((\bv-\bu) \Div \xi)\xhi_\mA\dx\dt=\int_0^T\int_\om(\bv-\bu)\cdot\nabla(\Div\xi)\xhi_\mA\dx\dt.
\end{align*}
Note that in the computation above we made use of \eqref{identityn} and of the fact that $\Div(\bv(\cdot,t)-\bu(\cdot,t))=0$ for almost any $t\in[0,\Tstrong]$. On the other hand, we can also choose $\bzeta=\bu$ in \eqref{NavS}, which is the fluid equation satisfied by the weak solution $\bv$, to obtain, after an integration by parts in time
\begin{align}
   & \int_\om \bu(T)\cdot\bv(T)\dx-\int_0^T\int_\om \pt\bu\cdot\bv\dx\dt+\int_0^T\int_\om (\bv\cdot\nabla)\bv\cdot \bu\dx\dt+\int_0^T\int_\om {D\bv}:D\bu\dx\dt\\&=\int_\om \bu_0\cdot\bv_0\dx-\int_0^T\int_\om (I-p\otimes p):\nabla\bu 
   \d{\mu_t}\dt.\label{seconda}
\end{align}
Now, using the alternative definition of the relative entropy \eqref{relEntropy3}, we can use the sharp energy inequality \eqref{advectiveMullins} (evaluated at $s=0$) and \eqref{kinetic}, with $\zeta=\Div\xi$, to infer after an integration by parts, using \eqref{u-v} and \eqref{seconda}, and recalling \eqref{initi}, that
\begin{align}
&\nonumber\Erel{T}\\&\nonumber\leq E[\xhi_0,\bv_0|\xhi_{\mA(0)},\bu_0]- \int_0^T\int_\om \normmm{D(\bv-\bu)}^2\dx\dt
\nonumber\\&\nonumber- \frac{1}{2} \int_{0}^{T} \int_{\om} |V|^2 \d\normmm{\mu_t}_{\sf} \dt
- \frac{1}{2} \int_{0}^{T} \int_{\om} \abs{\bH}^2 \d\normmm{\mu_t}_{\sf} \dt
\\&
- \int_{0}^{T} \int_{\om} V\Div\xi \d\normmm{\mu_t}_{\sf} \dt
- \int_{0}^{T} \int_{\pA(t)} \bn\cdot\partial_t\xi \d\H^{d-1} \dt+\int_0^T\int_\om \xhi_A\bv\cdot\nabla (\Div\xi)\dx\dt\nonumber\\&
+\int_0^T\int_\om ((\bu-\bv)\cdot\nabla )\bu\cdot(\bv-\bu)\dx\dt\nonumber\\&
-\int_0^T\int_\om \nabla(\Div \xi)\cdot (\bv-\bu)\xhi_{\mA}\dx\dt+\int_0^T\int_\om (I-p\otimes p):\nabla\bu 
   \d{\mu_t}\dt.	
	\label{estA}
\end{align}
Note that here we used the notion of generalized mean curvature vector, here denoted by $\bH$, for the weak solution, which is possible recalling Remark \ref{meancurvature}. Also, we used the identity
\begin{align*}
    &\int_0^T\int_\om (\bu\cdot\nabla )\bu\cdot(\bv-\bu)\dx\dt+\int_0^T\int_\om (\bv\cdot\nabla )\bv\cdot\bu\dx\dt\\&=\int_0^T\int_\om ((\bu-\bv)\cdot\nabla )\bu\cdot(\bv-\bu)\dx\dt+\int_0^T\int_\om (\bv\cdot\nabla )(\bu-\bv)\cdot(\bv-\bu)\dx\dt\\&
    =\int_0^T\int_\om ((\bu-\bv)\cdot\nabla )\bu\cdot(\bv-\bu)\dx\dt,
\end{align*}
which holds recalling that $\Div \bu=0$, $\Div \bv=0$, and thus 
\begin{align*}
   & \int_0^T\int_\om (\bu\cdot\nabla )\bu\cdot\bu\dx\dt=0,\\&
 \int_0^T\int_\om (\bv\cdot\nabla )(\bu-\bv)\cdot(\bv-\bu)\dx\dt=-\frac12\int_0^T\int_\om \bv\cdot\nabla \normmm{\bv-\bu}^2\dx\dt=0. 
\end{align*}
We now estimate all these terms separately. First, we let the approximate transport equations \eqref{transport}-\eqref{almosttranported} clearly enter into the proof. In particular, we write
\begin{align}
&\nonumber	- \int_{0}^{T} \int_{\pA(t)} \bn\cdot\partial_t\xi \d\H^{d-1} \dt\\&\nonumber=- \int_{0}^{T} \int_{\pA(t)} \bn\cdot(\partial_t\xi+(B\cdot\nabla )\xi+(I-\xi\otimes \xi)(\nabla B)^T\xi) \d\H^{d-1} \dt\\&\nonumber\quad+\int_{0}^{T} \int_{\pA(t)} \bn\cdot((B\cdot\nabla )\xi+(I-\xi\otimes \xi)(\nabla B)^T\xi) \d\H^{d-1} \dt\\&\nonumber
=-\int_{0}^{T} \int_{\pA(t)} (\bn-\xi)\cdot(\partial_t\xi+(B\cdot\nabla )\xi+(I-\xi\otimes \xi)(\nabla B)^T\xi) \d\H^{d-1} \dt\\\nonumber&\quad-\int_{0}^{T} \int_{\pA(t)} \xi\cdot(\partial_t\xi+(B\cdot\nabla )\xi) \d\H^{d-1} \dt\\&\quad\nonumber+\int_{0}^{T} \int_{\pA(t)} \xi\otimes(\bn-\xi):\nabla B \d\H^{d-1} \dt-\int_0^T\int_{\pA}(\xi\cdot(\nabla B)^T\xi)\xi\cdot(\bn-\xi)\d\mathcal H^{d-1}\dt\\&\quad\nonumber+\int_{0}^{T} \int_{\pA(t)} \bn\cdot(B\cdot\nabla )\xi \d\H^{d-1} \dt\\&\nonumber
\leq \int_{0}^{T} \int_{\pA(t){\cap \supp\xi(t)}} \normmm{\bn-\xi}O(\dist(\cdot,\mathcal I(t))) \d\H^{d-1} \dt\\&\quad\nonumber+
 \int_{0}^{T} \int_{\pA(t){\cap \supp\xi(t)}} O(\dist(\cdot,\mathcal I(t))^2)\d\H^{d-1} \dt\\&\quad +\int_0^T\int_{\pA(t)\cap \supp\xi(t)} \norm{\xi\cdot(\nabla B)^T\xi}_{L^\infty(\om\times(0,\Tstrong
	))}\normmm{\xi\cdot(\bn-\xi)}\d\Hd\dt\nonumber\\&\quad\nonumber+\int_{0}^{T} \int_{\pA(t)} \xi\otimes(\bn-\xi):\nabla B \d\H^{d-1} \dt+\int_{0}^{T} \int_{\pA(t)} \bn\cdot(B\cdot\nabla )\xi \d\H^{d-1} \dt,\\&
	\leq \int_{0}^{T} \int_{\pA(t)} \xi\otimes(\bn-\xi):\nabla B \d\H^{d-1} \dt\nonumber\\&\quad +\int_{0}^{T} \int_{\pA(t)} \bn\cdot(B\cdot\nabla )\xi \d\H^{d-1} \dt+C\int_0^T\Erel{t}\dt,\label{last}
\end{align}
where we used the regularity \eqref{eq:conditionsXiB3}, estimates \eqref{transport}-\eqref{almosttranported}, together with the coercivity estimates  \eqref{coerciv5b},\eqref{coerciv7} exploiting the trivial inequalities $\normmm{(\bn-\xi)\cdot \xi}\leq (1-\normmm{\xi}^2)+(1-\bn\cdot \xi)$ and
\begin{align}
1-\normmm{\xi}^2\leq 2(1-\normmm{\xi})\leq 2(1-\bn\cdot\xi)\label{ctrlxi},
\end{align}
 which comes from $\normmm{\xi},\normmm{\bn}\leq 1$.
Then, about the first term in the right-hand side of \eqref{last}, the compatibility condition \eqref{PE} with the definition of $\rho$ in \eqref{RNder} give 
\begin{align}
	&\nonumber \int_{0}^{T} \int_{\pA(t)} \xi\otimes(\bn-\xi):\nabla B \d\H^{d-1} \dt\\&
	=\int_{0}^{T} \int_{\om\times \mathbb S^{d-1}} \xi\otimes(p-\xi):\nabla B \d\mu_t\dt +\int_{0}^{T} \int_{\om} (1-\rho)\xi\otimes\xi:\nabla B \d\normmm{\mu_t}_{\sf} \dt\nonumber\\&
	\stackrel{\eqref{coerciv5}} {\leq}\int_{0}^{T} \int_{\pA(t)} \xi\otimes(p-\xi):\nabla B \d\mu_t\dt +C\int_0^T\Erel{t}\dt.
	\label{ctrl}
\end{align}
Proceeeding with the second term in the right-hand side of \eqref{last}, we can write, using \eqref{PE}, 
\begin{align}
&\nonumber\int_{0}^{T} \int_{\pA(t)} \bn\cdot(B\cdot\nabla )\xi \d\H^{d-1} \dt\\&\nonumber=\int_{0}^{T} \int_{\pA(t)} \bn\cdot\Div(\xi\otimes B ) \d\H^{d-1} \dt-\int_{0}^{T} \int_{\pA(t)} \bn\cdot\xi\Div B \d\H^{d-1} \dt\\&\nonumber
=\int_{0}^{T} \int_{\pA(t)} \bn\cdot\Div(\xi\otimes B ) \d\H^{d-1} \dt-\int_{0}^{T} \int_{\pA(t)} (\bn\cdot\xi-1)\Div B \d\H^{d-1} \dt\\&
\quad\nonumber-\int_{0}^{T} \int_{\om} (\rho-1)\Div B \d\mmu \dt
-\nonumber \int_{0}^{T} \int_{\om\times \mathbb S^{d-1}}(Id-p\otimes p):\nabla B \d\mu_t \dt\\&
\quad  - \int_{0}^{T} \int_{\om\times \mathbb S^{d-1}}(p\otimes p):\nabla B \d\mu_t \dt\nonumber
\\&\stackrel{\text{\eqref{coerciv5}-\eqref{coerciv5b}}}{\leq} \int_{0}^{T} \int_{\pA(t)} \bn\cdot\Div(\xi\otimes B ) \d\H^{d-1} \dt-\nonumber \int_{0}^{T} \int_{\om\times \mathbb S^{d-1}}(Id-p\otimes p):\nabla B \d\mu_t \dt\\&
\quad  - \int_{0}^{T} \int_{\om\times \mathbb S^{d-1}}(p\otimes p):\nabla B \d\mu_t \dt+C\int_0^T\Erel{t}\dt.\label{ptensorp}
\end{align}
Then the symmetric relation $\Div(\Div(\xi\otimes B))=\Div(\Div (B\otimes \xi))$ allows, after the Gau\ss\ theorem, to infer $\bn\cdot  \Div (B\otimes \xi)=\bn\cdot  \Div (\xi\otimes B)=(\bn \cdot B)\Div \xi+\bn\otimes \xi:\nabla B$, so that, resorting again to \eqref{PE}, we infer 
\begin{align*}
	&\int_{0}^{T} \int_{\pA(t)} \bn\cdot\Div(\xi\otimes B ) \d\H^{d-1} \dt\\&
=	\int_{0}^{T} \int_{\pA(t)} (\bn \cdot B)\Div \xi \d\H^{d-1} \dt+\int_{0}^{T} \int_{\om\times \sf} p\otimes \xi:\nabla B\d\mu_t\dt
\end{align*}
As a consequence, making use of \eqref{last}, \eqref{ctrl}, and \eqref{ptensorp}, we get
\begin{align}
		&- \int_{0}^{T} \int_{\pA(t)} \bn\cdot\partial_t\xi \d\H^{d-1} \dt\nonumber\\&
		\leq\nonumber \int_{0}^{T} \int_{\pA(t)} (\bn \cdot B)\Div \xi \d\H^{d-1} \dt-\int_{0}^{T} \int_{\om\times \mathbb S^{d-1}}(Id-p\otimes p):\nabla B \d\mu_t \dt\\&\quad -\int_{0}^{T} \int_{\om\times \sf} (p-\xi)\otimes (p-\xi):\nabla B\d\mu_t\dt+C\int_0^T\Erel{t}\dt\nonumber\\&
       \leq  \int_{0}^{T} \int_{\pA(t)} (\bn \cdot B)\Div \xi \d\H^{d-1} \dt-\int_{0}^{T} \int_{\om\times \mathbb S^{d-1}}(Id-p\otimes p):\nabla B \d\mu_t \dt\nonumber\\&\quad +C\int_0^T\Erel{t}\dt,
	\label{final1}
\end{align}
where in the last inequality we made use of \eqref{eq:conditionsXiB3} and the coercivity condition \eqref{coearciv1}.
As a consequence, we can upgrade \eqref{estA} to 
\begin{align}
&\nonumber\Erel{T}\\&\nonumber\leq E[\xhi_0,\bv_0|\xhi_{\mA(0)},\bu_0]- \int_0^T\int_\om \normmm{D(\bv-\bu)}^2\dx\dt\\&\nonumber- \frac{1}{2} \int_{0}^{T} \int_{\om} |V|^2 \d\normmm{\mu_t}_{\sf} \dt
- \frac{1}{2} \int_{0}^{T} \int_{\om} \abs{\bH}^2 \d\normmm{\mu_t}_{\sf} \dt
\\&
- \int_{0}^{T} \int_{\om} V\Div\xi \d\normmm{\mu_t}_{\sf} \dt+\int_{0}^{T} \int_{\pA(t)} (\bn \cdot B)\Div \xi \d\H^{d-1} \dt+\int_0^T\int_\om \xhi_A\bv\cdot\nabla (\Div\xi)\dx\dt\nonumber\\&
+\int_0^T\int_\om ((\bu-\bv)\cdot\nabla )\bu\cdot(\bv-\bu)\dx\dt\nonumber\\&
-\int_0^T\int_\om \nabla(\Div \xi)\cdot (\bv-\bu)\xhi_{\mA}\dx\dt-\int_0^T\int_\om (I-p\otimes p):\nabla(B-\bu) 
   \d{\mu_t}\dt\nonumber\\&+C\int_0^T\Erel{t}\dt
   .	
	\label{estAB}
\end{align}
 We now start from the dissipative term in the energy inequality related to the (weak) total material interface velocity $V$. Here we need to make careful attention due to the presence of the advective velocity $\bu$. Namely,  recalling the compatibility condition \eqref{PE}, together with the fundamental estimate \eqref{geometric}, we have

\begin{align}
&\nonumber- \frac{1}{2} \int_{0}^{T} \int_{\om} |V|^2 \d\mmu \dt
	- \int_{0}^{T} \int_{\om} V\Div\xi\d\mmu \dt+\int_{0}^{T} \int_{\pA(t)} (\bn \cdot B)\Div \xi \d\H^{d-1} \dt\nonumber\\&\nonumber
	=- \frac{1}{2} \int_{0}^{T} \int_{\om} |V+\Div \xi|^2 \d\mmu \dt+\frac12 \int_{0}^{T} \int_{{\supp\xi(t)}} (\Div\xi)^2 \d\mmu \dt\\&\quad +\int_{0}^{T} \int_{\pA(t)} (\bn \cdot B)\Div \xi \d\H^{d-1} \dt\\&\nonumber=
	- \frac{1}{2} \int_{0}^{T} \int_{\om} |V+\Div \xi|^2 \d\mmu \dt+\frac12 \int_{0}^{T} \int_{{\supp \xi(t)}} (\Div\xi+(B-\bu)\cdot\xi)^2 \d\mmu \dt\\&\nonumber\quad- \int_{0}^{T} \int_{\om} \Div\xi(B-\bu)\cdot \xi \d\mmu \dt\\&\nonumber\quad +\int_{0}^{T} \int_{\pA(t)} (\bn \cdot B)\Div \xi \d\H^{d-1} \dt-\frac12 \int_{0}^{T} \int_{\om} ((B-\bu)\cdot \xi)^2 \d\mmu \dt\\&\nonumber
	=	- \frac{1}{2} \int_{0}^{T} \int_{\om} |V+\Div \xi|^2 \d\mmu \dt+\frac12 \int_{0}^{T} \int_{{\supp \xi(t)}} (\Div\xi+(B-\bu)\cdot\xi)^2 \d\mmu \dt\\&\nonumber\quad +\int_{0}^{T} \int_{\om} (\rho-1)(\Div\xi) (B-\bu)\cdot \xi \d\mmu \dt+\int_0^T\int_{\pA}(\Div \xi)\bu\cdot\bn\d\H^{d-1}\dt\\&\quad +\int_0^T\int_{\pA(t)}(\Div\xi) (B-\bu)\cdot (\bn-\xi) \d\H^{d-1} \nonumber\dt\nonumber-\frac12 \int_{0}^{T} \int_{\om} ((B-\bu)\cdot \xi)^2 \d\mmu \dt\\&
		\stackrel{\text{\eqref{geometric}, \eqref{coerciv2}-\eqref{coerciv5}}}{\leq} 	- \frac{1}{2} \int_{0}^{T} \int_{\om} |V+\Div \xi|^2 \d\mmu \dt +\int_0^T\int_{\pA(t)}(\Div\xi) (B-\bu)\cdot (\bn-\xi) \d\H^{d-1} \nonumber\dt\\&\quad+\int_0^T\int_{\pA}(\Div \xi)\bu\cdot\bn\d\H^{d-1}\dt-\frac12 \int_{0}^{T} \int_{\om} ((B-\bu)\cdot \xi)^2 \d\mmu \dt\nonumber\\&\quad+C\int_0^T\Erel{t}\dt.
			\label{fundamental1}
\end{align}
 We now pass to work on the mean curvature term in \eqref{estAB}. Note that here we need to exploit the fundamental control \eqref{normal}, corresponding to the fact that $B-\bu$ points in the normal direction on $\cI(t)$, justifying the choice of $B$ different from the one in \cite{HenselLiu}. We then have, by \eqref{supports} and \eqref{xiprop},
 \begin{align}
   & \nonumber - \frac{1}{2} \int_{0}^{T} \int_{\om} \abs{\bH}^2 \d\mmu \dt-\frac12 \int_{0}^{T} \int_{\om} ((B-\bu)\cdot \xi)^2 \d\mmu \dt\\
   &\nonumber\leq - \frac{1}{2} \int_{0}^{T} \int_{\supp (B(t)-\bu(t))} \abs{\bH}^2 \d\mmu \dt-\frac12 \int_{0}^{T} \int_{\om} ((B-\bu)\cdot \xi)^2 \d\mmu \dt\\&
   =- \frac{1}{2} \int_{0}^{T} \nonumber\int_{\supp (B(t)-\bu(t))} \abs{(Id-\nabla s_\cI\otimes \nabla s_\cI)\bH}^2 \d\mmu \dt\\&\nonumber\quad-\frac12\int_0^T\int_{\supp (B(t)-\bu(t))}\normmm{\nabla s_\cI\cdot \bH}^2(1-\normmm{\xi}^2)\d\mmu\dt\\&\quad-\frac12\int_0^T\int_{\supp (B(t)-\bu(t))} \normmm{(\bH\cdot\xi-(B-\bu)\cdot \xi)}^2\d\mmu\dt\nonumber\\&\quad -\int_0^T\int_\om((B-\bu)\cdot \xi)\bH\cdot\xi\d\mmu\dt.\label{ctrl1}
  \end{align}
Moreover, using \eqref{PH}, which is possible since $\bu$ is sufficiently regular by \eqref{eq:conditionsStrongSolSharpInterfaceLimit4}, we obtain
\begin{align}
	&\nonumber -\int_{0}^{T} \int_{\om\times \mathbb S^{d-1}}(Id-p\otimes p):\nabla (B-\bu) \d\mu_t \dt\stackrel{\eqref{PH}}{=}\int_0^T\int_\om \bH\cdot (B-\bu)\d\mmu\dt \\&\nonumber
    \stackrel{\eqref{xiprop}}{=}\int_0^T\int_{\supp (B(t)-\bu(t))}(Id-\nabla s_\cI\otimes \nabla s_\cI)\bH:(Id-\nabla s_\cI\otimes \nabla s_\cI)(B-\bu)\d\mmu\dt\\&\nonumber\quad+\int_0^T\int_{\supp (B(t)-\bu(t))}((B-\bu)\cdot \xi)\bH\cdot\xi\d\mmu\dt \\&\quad 
    +\int_0^T\int_{\supp (B(t)-\bu(t))}(1-\normmm{\xi}^2)((B-\bu)\cdot \nabla s_\cI)\bH\cdot\nabla s_\cI\d\mmu\dt.\label{ctrl2}
    \end{align}
Summing up \eqref{ctrl1} together with \eqref{ctrl2}, by means of Cauchy-Schwarz and Young's inequalities, we get
\begin{align}
    \nonumber &- \frac{1}{2} \int_{0}^{T} \int_{\om} \abs{\bH}^2 \d\mmu \dt-\frac12 \int_{0}^{T} \int_{\om} ((B-\bu)\cdot \xi)^2 \d\mmu \dt\\&\nonumber-\int_{0}^{T} \int_{\om\times \mathbb S^{d-1}}(Id-p\otimes p):\nabla (B-\bu) \d\mu_t \dt\\&
    \leq\nonumber - \frac{1}{2} \int_{0}^{T} \nonumber\int_{\supp (B(t)-\bu(t))} \abs{(Id-\nabla s_\cI\otimes \nabla s_\cI)\bH}^2 \d\mmu \dt \\&\quad\nonumber-\frac12\int_0^T\int_{\supp (B(t)-\bu(t))}\normmm{\nabla s_\cI\cdot \bH}^2(1-\normmm{\xi}^2)\d\mmu\dt\\&\nonumber \quad-\frac12\int_0^T\int_{\supp (B(t)-\bu(t))} \normmm{\bH\cdot\xi-(B-\bu)\cdot \xi)}^2\d\mmu\dt\\&\nonumber
    \quad+\int_0^T\int_{\supp (B(t)-\bu(t))}(Id-\nabla s_\cI\otimes \nabla s_\cI)\bH:(Id-\nabla s_\cI\otimes \nabla s_\cI)(B-\bu)\d\mmu\dt \nonumber\\&\nonumber\quad 
    +\int_0^T\int_{\supp (B(t)-\bu(t))}(1-\normmm{\xi}^2)((B-\bu)\cdot \nabla s_\cI)\bH\cdot\nabla s_\cI\d\mmu\dt\nonumber
    \\&
    \leq - \frac{1}{4} \int_{0}^{T} \nonumber\int_{\supp (B(t)-\bu(t))} \abs{(Id-\nabla s_\cI\otimes \nabla s_\cI)\bH}^2 \d\mmu \dt\\&\quad \nonumber-\frac14\int_0^T\int_{\supp (B(t)-\bu(t))}\normmm{\nabla s_\cI\cdot \bH}^2(1-\normmm{\xi}^2)\d\mmu\dt\\&\nonumber \quad-\frac12\int_0^T\int_{\supp (B(t)-\bu(t))} \normmm{\bH\cdot\xi-(B-\bu)\cdot \xi)}^2\d\mmu\dt \\&\nonumber\quad +C\int_0^T\int_{\om}\normmm{(Id-\nabla s_\cI\otimes \nabla s_\cI)(B-\bu)}^2\d\mmu\dt \\&\quad 
    +C\int_0^T\int_{\om}(1-\normmm{\xi}^2)((B-\bu)\cdot \nabla s_\cI))^2\d\mmu\dt\nonumber\\&
   {\leq} C\int_0^T\Erel{t}\dt.\label{f1}
\end{align}
Observe that, in the last step we first used the compact support of $B-\bu$ (cf. property \eqref{supports2}), together with the regularity of $B$ and the coercivity properties \eqref{coerciv5}, \eqref{coerciv7}, to also infer 
\begin{align}
    &\nonumber C\int_0^T\int_{\om}(1-\normmm{\xi}^2)((B-\bu)\cdot \nabla s_\cI)^2\d\mmu\dt\\&\nonumber=C\int_0^T\int_{\om}(1-\normmm{\xi}^2)((B-\bu)\cdot \nabla s_\cI)^2\rho\d\mmu\dt\\&\nonumber\quad +C\int_0^T\int_{\om}(1-\rho)(1-\normmm{\xi}^2)((B-\bu)\cdot \nabla s_\cI)^2\d\mmu\dt\\&\nonumber
    \leq C\int_0^T\int_{\pA(t)}(1-\normmm{\xi})\d\H^{d-1}\dt+C\int_0^T\int_{\om}(1-\rho)\d\mmu\dt\\&
    \leq C\int_0^T\Erel{t}\dt.\label{firste}
\end{align}
Secondly, for the other term we crucially use property \eqref{normal}, together with \eqref{supports2} and \eqref{short}, to infer
\begin{align*}
   &\int_0^T\int_{\om}\normmm{(Id-\nabla s_\cI\otimes \nabla s_\cI)(B-\bu)}^2\d\mmu\dt\\&\leq C\int_0^T\int_{\supp (B(t)-\bu(t))} \dist(\cdot,\II(t)))^2\d\mmu\dt \\&
   \leq C\int_{\supp (B(t)-\bu(t))} \min\{4,\frac1{c^2}\dist(\cdot,\II(t)))^2\}\d\mmu\dt\\&
   \leq C\int_0^T\int_\om 4(1-\normmm{\xi}^2)\d\mmu\dt\\&
   \leq 
   C\int_0^T\Erel{t}\dt,
\end{align*}
where for the last inequality we can argue as in \eqref{firste}.
Then, in order to complete the estimates of the right-hand side of \eqref{fundamental1}, we recall again that, by \eqref{ctrlxi}, it holds
\begin{align}
\label{fondamentalexi}
\normmm{\xi\cdot(\bn-\xi)}\leq (1-\normmm{\xi}^2)+(1-\xi\cdot \bn)\leq 3(1-\bn\cdot \xi),
\end{align}
so that we can write
\begin{align}
&\nonumber\int_0^T\int_{\pA(t)}(\Div\xi) (B-\bu)\cdot (\bn-\xi) \d\H^{d-1} \nonumber\dt\\&\nonumber=\int_0^T\int_{\pA(t)}(\Div\xi) (\bn-\xi)(Id-\xi\otimes \xi)(B-\bu) \d\H^{d-1} \nonumber\dt\\&\quad +\int_0^T\int_{\pA(t)}(\Div\xi)((B-\bu)\cdot \xi) (\bn-\xi)\cdot \xi \d\H^{d-1} \dt\nonumber\\&\nonumber
\stackrel{\eqref{coerciv5b},\eqref{fondamentalexi}}{\leq} \int_0^T\int_{\pA(t)}(\Div\xi) (\bn-\xi)(Id-\xi\otimes \xi)(B-\bu) \d\H^{d-1} \dt\\&\quad + C\int_0^T\Erel{t}\dt.
    \label{diff1}
\end{align}
Additionally, notice that, by \eqref{xiprop} and \eqref{ctrlxi}, we have 
\begin{align}
\normmm{\nabla s_\cI\otimes \nabla s_\cI-\xi\otimes \xi}\leq \sqrt{ 1+\normmm{\xi}^4-2\normmm{\xi}^2}=1-\normmm{\xi}^2\leq 2(1-\bn\cdot \xi),
    \label{ctrlxi1}
\end{align}
on $ \supp \xi(t)$, entailing, by Cauchy-Schwarz and Young's inequalities, and again the crucial property \eqref{normal},
\begin{align}
    &\nonumber\int_0^T\int_{\pA(t)}(\Div\xi) (\bn-\xi)(Id-\xi\otimes \xi)(B-\bu) \d\H^{d-1} \dt\\&
    \leq \int_0^T\int_{\pA(t)}(\Div\xi) (\bn-\xi)(Id-\nabla s_\cI\otimes \nabla s_\cI)(B-\bu) \d\H^{d-1} \dt\nonumber\\&\nonumber
    \quad +\int_0^T\int_{\pA(t)}(\Div\xi) (\bn-\xi)(\nabla s_\cI\otimes \nabla s_\cI-\xi\otimes \xi)(B-\bu) \d\H^{d-1} \dt\\&\nonumber
   {\leq} \int_0^T\int_{\pA(t)}\normmm{\Div\xi} \normmm{\bn-\xi}^2 \d\H^{d-1} \dt\nonumber\\&\nonumber\quad +C\int_0^T\int_{\pA(t)}\normmm{\Div\xi} \normmm{(Id-\nabla s_\cI\otimes \nabla s_\cI)(B-\bu) }^2\d\H^{d-1} \dt\\&
    \quad +C\int_0^T\int_{\pA(t)}\normmm{\Div\xi} (1-\xi\cdot\bn) \d\H^{d-1} \dt\nonumber\\&
     \stackrel{\text{\eqref{normal},\eqref{coerciv5b}-\eqref{coerciv7}}}{\leq}C\int_0^T\Erel{t}\dt.\label{copia}
\end{align}

We can now plug estimates \eqref{fundamental1}, \eqref{f1}, \eqref{diff1} and \eqref{copia} into \eqref{estAB}, to obtain  
\begin{align}
&\nonumber\Erel{T}\\&\nonumber\leq E[\xhi_0,\bv_0|\xhi_{\mA(0)},\bu_0]- \int_0^T\int_\om \normmm{D(\bv-\bu)}^2\dx\dt- \frac{1}{2} \int_{0}^{T} \int_{\om} |V|^2 \d\normmm{\mu_t}_{\sf} \dt
\\&
\quad - \frac{1}{2} \int_{0}^{T} \int_{\om} |V+\Div \xi|^2 \d\mmu \dt+\int_0^T\int_\om \xhi_A\bv\cdot\nabla (\Div\xi)\dx\dt+\int_0^T\int_{\pA}(\Div \xi)\bu\cdot\bn\d\H^{d-1}\dt\nonumber\\&
\quad +\int_0^T\int_\om ((\bu-\bv)\cdot\nabla )\bu\cdot(\bv-\bu)\dx\dt\nonumber\\&
\quad -\int_0^T\int_\om \nabla(\Div \xi)\cdot (\bv-\bu)\xhi_{\mA}\dx\dt+C\int_0^T\Erel{t}\dt
   .	
	\label{estABC}
\end{align}
We are left with the control of the terms related to the fluid velocities. We first notice that, using the Gau\ss\ theorem, we have 
\begin{align*}
    \int_0^T\int_{\pA}(\Div \xi)\bu\cdot\bn\d\H^{d-1}\dt=-\int_0^T\int_\om \Div(\bu\Div\xi)\xhi_A\dx\dt=-\int_0^T\int_\om \bu\cdot\nabla (\Div\xi)\xhi_A\dx\dt,
\end{align*}
where we used $\Div \bu=0$. As a consequence,we can group four terms appearing in \eqref{estABC} above, namely, using the crucial \eqref{inequcmp} with $\zeta=\nabla(\Div\xi)$ and $\delta=\frac{1}{2C_{K}}$, with $C_K=2$ as the (squared) constant appearing in the Korn's inequality \eqref{korn},
\begin{align}
    \nonumber&\normmm{\int_0^T\int_\om \xhi_A\bv\cdot\nabla (\Div\xi)\dx\dt+\int_0^T\int_{\pA}(\Div \xi)\bu\cdot\bn\d\H^{d-1}\dt-\int_0^T\int_\om \nabla(\Div \xi)\cdot (\bv-\bu)\xhi_{\mA}\dx\dt}\\&\nonumber
    =\normmm{\int_0^T\int_\om (\xhi_A-\xhi_{\mA})\nabla(\Div \xi)\cdot (\bv-\bu)\dx\dt}\\&\nonumber
    \leq \frac1{2C_K}\int_0^T\int_\om \normmm{\nabla (\bv-\bu)}^2\dx\dt+C(1+\norm{\nabla(\Div\xi)}_{L^\infty(\om\times[0,T])})\int_0^T\int_\om \normmm{\bv-\bu}^2\dx\dt\\&\quad \nonumber+C\norm{\nabla(\Div\xi)}_{L^\infty(\om\times[0,T])}\int_0^T \Eb\dt\\&
    \leq \onehalf\int_0^T\int_\om \normmm{D (\bv-\bu)}^2\dx\dt+C\int_0^T \Er\dt+C\int_0^T\Eb\dt,\label{ff1}
\end{align}
where in the last step we used the definition of $\Er$, together with Korn's inequality \eqref{korn} and the regularity \eqref{eq:conditionsXiB3} of $\xi$. Observe that this estimate requires also a control for the bulk error term, which we will develop below. 

The last terms to be estimated in \eqref{estABC} are the advective terms for the fluid velocity. Namely, we have, by Gagliardo-Nirenberg's ineqality (here we show the case $d=3$, as the case $d=2$ is analogous), Young's and Korn's inequalities,
\begin{align*}
   &\normmm{\int_\om ((\bu-\bv)\cdot\nabla )\bu\cdot(\bv-\bu)\dx}\\&\leq \norm{\nabla\bu}_{\bL^2(\om)}\norm{\bv-\bu}_{\bL^3(\om)}\norm{\bv-\bu}_{\bL^6(\om)}\\&\leq C\norm{D(\bv-\bu)}_{\bL^2(\om)}^\frac32\norm{\bv-\bu}_{\bL^2(\om)}^\onehalf\leq \frac 14\int_\om \normmm{D(\bv-\bu)}^2\dx+C\int_\om \normmm{\bv-\bu}^2\dx\,
\end{align*}
of course entailing 
\begin{align}
   \nonumber &\normmm{\int_0^T\int_\om ((\bu-\bv)\cdot\nabla )\bu\cdot(\bv-\bu)\dx\dt}\\&\leq  \frac 14\int_0^T\int_\om \normmm{D(\bv-\bu)}^2\dx\dt+C\int_0^T \Er\dt.\label{ff2}
\end{align}
Now plugging \eqref{ff1}-\eqref{ff2} in \eqref{estABC} finally gives the desired relative entropy inequality
\begin{align}
&\nonumber\Erel{T}\\&\nonumber\leq E[\xhi_0,\bv_0|\xhi_{\mA(0)},\bu_0]\\&\nonumber\quad - \frac14\int_0^T\int_\om \normmm{D(\bv-\bu)}^2\dx\dt- \frac{1}{2} \int_{0}^{T} \int_{\om} |V|^2 \d\normmm{\mu_t}_{\sf} \dt - \frac{1}{2} \int_{0}^{T} \int_{\om} |V+\Div \xi|^2 \d\mmu \dt\\&\quad +C\int_0^T\Erel{t}\dt
 +C\int_0^T\Eb\dt  .	
	\label{estABCD}
\end{align}
As anticipated, to complete the proof and use Gronwall's inequality we are left with a control on the bulk error $\Eb$. We use the approximate transport equation \eqref{thetabasic} for the weight $\vartheta$. {First, using the evolution equation \eqref{kinetic} with $\zeta=\vartheta$, we obtain 
\begin{align}
 \int_{\om} 
	({\chi_{A(T)}\vartheta(T) {-} \chi_{A(0)}} {\vartheta(0)})\dx=\int_0^T\int_\om \xhi_{A}(\pt\vartheta+\bv\cdot\nabla\vartheta)\dx\dt-\int_0^T\int_\om V\vartheta \d\mmu\dt.\label{I1}  
\end{align}
Concerning the strong solution, since $\partial_t\xhi_{\mA(t)}$ is supported on $\cI(t)$ and  $\vartheta(\cdot,t)=0$ on $\cI(t)$, integrating by parts in time we have 
\begin{align}
   \nonumber\int_0^T\int_\om \xhi_{\mA(t)}\pt\vartheta\dx\dt&\nonumber=-\int_0^T\int_\om \partial_t\xhi_{\mA(t)}\vartheta(t)\dx\dt+\int_\om \xhi_{\mA(T)}\vartheta(T)\dx-\int_\om \xhi_{\mA(0)}\vartheta(0)\dx\\&= \int_\om \xhi_{\mA(T)}\vartheta(T)\dx-\int_\om \xhi_{\mA(0)}\vartheta(0)\dx.\label{I2}
\end{align}
Also, again since $\vartheta(\cdot,t)=0$ on $\cI(t)$ we deduce, after using the Gau\ss\ theorem,
\begin{align}
    &\int_0^T\int_\om\xhi_\mA\bu\cdot\nabla\vartheta\dx\dt=-\int_0^T\int_{\II(t)}\vartheta\bu\cdot\bn_\II\d\H^{d-1}\dt=0,\label{zero1}\\&
    \int_0^T\int_\om\xhi_\mA\bv\cdot\nabla\vartheta\dx\dt=-\int_0^T\int_{\II(t)}\vartheta\bv\cdot\bn_\II\d\H^{d-1}\dt=0,\label{zero2}
\end{align}
where we also used $\Div\bu=0$ and $\Div \bv=0$. Therefore, from \eqref{I1}-\eqref{zero1} we get
\begin{align*}
 &\int_{\om} 
	({\chi_{A(T)}-\xhi_{\mA(T)})\vartheta(T) }\dx\\&=\int_\om(\xhi_{A(0)}-\xhi_{\mA(0)})\vartheta(0)\dx +\int_0^T\int_\om (\xhi_{A(t)}-\xhi_{\mA(t)})\pt\vartheta\dx\dt\\&
    +\int_0^T\int_\om\xhi_\mA\bu\cdot\nabla\vartheta\dx\dt+\int_0^T\int_\om(\xhi_A-\xhi_\mA)\bv\cdot\nabla\vartheta\dx\dt-\int_0^T\int_\om V\vartheta \d\mmu\dt.      
\end{align*} 
Considering now the alternative formulation of the bulk error \eqref{eq:bulkError2} evaluated at times $T$ and $0$, we obtain 
}
\begin{align}
	\Ebulk{T}=& E_{\mathrm{bulk}}[\xhi_0|\mA(0)]-	\int_0^T\int_\om V\vartheta \d\mmu\dt +\int_{0}^{T}\int_{\om}({\xhi_{A(t)}-\xhi_{\mA(t)}})\pt\vartheta \dx \dt\nonumber\\&+\int_0^T\int_\om\xhi_\mA\bu\cdot\nabla\vartheta\dx\dt+\int_0^T\int_\om(\xhi_A-\xhi_\mA)\bv\cdot\nabla\vartheta\dx\dt.
	\label{estE}
\end{align}
{Then, again using that $\vartheta(\cdot,t)=0$ on $\cI(t)$, together with the alternative definition  \eqref{eq:bulkError2} and by means of the Gau\ss\ theorem, we infer} 
\begin{align}
&\int_{0}^{T}\int_{\om}({\xhi_{A(t)}-\xhi_{\mA(t)}})\pt\vartheta \dx \dt\nonumber\\&\stackrel{\eqref{thetabasic}}{=}\int_{0}^{T}\int_{\om}(\xhi_{A(t)}-\xhi_{\mA(t)}){O\left(\min\{1,\frac1{c}\dist(\cdot,I(t))\}\right) }\dx \dt-\int_{0}^{T}\int_{\om}(\xhi_{A(t)}-\xhi_{\mA(t)})B\cdot\nabla\vartheta \dx \dt\nonumber\\&
\stackrel{\eqref{coerciv4}}{\leq} C\int_{0}^{T} \Ebulk{t} \dt-\int_{0}^{T}\int_{\om}(\xhi_{A(t)}-\xhi_{\mA(t)})\Div(\vartheta B) \dx \dt\nonumber\\&\quad+\int_{0}^{T}\int_{{\supp B(t)}}(\xhi_{A(t)}-\xhi_{\mA(t)})\vartheta\Div B \dx \dt\nonumber
\\&\leq \int_{0}^{T}\int_{\pA(t)}\vartheta B\cdot\bn \d\H^{d-1} \dt+C\int_{0}^{T} \Ebulk{t} \dt\label{diff}.
\end{align} 
We now need to treat the first term appearing in the right-hand side of \eqref{diff}. We aim to reconstruct the term $V+\Div \xi$, since it appears squared as a dissipative term in \eqref{estABCD}. Also, we make use of the fundamental geometric evolution \eqref{geometric}. We write, using again the compatibility \eqref{PE},
\begin{align}
&\nonumber -\int_0^T \int_\om V\vartheta \d\mmu\dt+\int_{0}^{T}\int_{\pA(t)}\vartheta (B-\bu)\cdot\bn \d\H^{d-1} \dt+\int_{0}^{T}\int_{\pA(t)}\vartheta \bu\cdot\bn \d\H^{d-1} \dt\\&
\nonumber = -\int_0^T \int_\om V\vartheta \d\mmu\dt+\int_{0}^{T}\int_{\pA(t)}\vartheta (B-\bu)\cdot(\bn-\xi) \d\H^{d-1} \dt\\&\nonumber
\quad +\int_0^T\int_{\pA(t){\cap \supp \xi(t)}}\vartheta ((B-\bu)\cdot\xi+\Div\xi)\d\H^{d-1}\dt+\int_0^T\int_{{\supp \xi(t)}}\vartheta (1-\rho) \Div\xi\d\mmu\dt\\&\nonumber\quad -\int_0^T\int_{{\supp \xi(t)}}\vartheta \Div\xi\d\mmu\dt+\int_{0}^{T}\int_{\pA(t)}\vartheta \bu\cdot\bn \d\H^{d-1} \dt\\&
=\underbrace{\int_{0}^{T}\int_{\pA(t)}\vartheta (B-\bu)\cdot(\bn-\xi) \d\H^{d-1} \dt}_{J_1} +\underbrace{\int_0^T\int_{\pA(t){\cap \supp \xi(t)}}\vartheta ((B-\bu)\cdot\xi+\Div\xi)\d\H^{d-1}\dt}_{J_2}\nonumber\\&\quad\underbrace{+\int_0^T\int_{{\supp \xi(t)}}\vartheta (1-\rho) \Div\xi\d\mmu\dt}_{J_3} \underbrace{-\int_0^T\int_{\om}\vartheta (V+\Div\xi)\d\mmu\dt}_{J_4}\nonumber\\&\quad +\int_{0}^{T}\int_{\pA(t)}\vartheta \bu\cdot\bn \d\H^{d-1} \dt.\label{complete}
	\end{align}
 First, by Cauchy-Schwarz and Young's inequalities, using the Lipschitz property \eqref{varthetacontrol} of $\vartheta$ and the coercivity \eqref{coerciv6} and \eqref{coerciv7}, we have
\begin{align*}
	\normmm{J_1}\leq \int_{0}^{T}\int_{\pA(t)}\normmm{B-\bu}^2\vartheta^2 \d\H^{d-1} \dt+C\int_{0}^{T}\int_{\pA(t)}\normmm{\bn-\xi}^2 \d\H^{d-1} \dt \leq C\int_0^T\Erel\dt.
\end{align*}
Then, by \eqref{varthetacontrol}, using the geometric evolution equation \eqref{geometric} and the coercivity \eqref{coerciv7}, we have
\begin{align*}
	\normmm{J_2}\leq \int_{0}^{T}\int_{\pA(t){\cap \supp \xi(t)}}O(\dist(\cI(\cdot,t))^2) \d\H^{d-1} \dt\leq C\int_0^T\Erel{t}\dt.
\end{align*}
{Concerning $J_3$, the coercivity   \eqref{coerciv5}, recalling the regularity of $\xi$, gives
\begin{align*}
	\normmm{J_3}\leq C\int_0^T\int_{\om} (1-\rho)\d\mmu\dt\leq  C\int_0^T\Erel{t}\dt.
\end{align*}}
Finally, Cauchy-Schwarz and Young's inequalities, together with the Lipschitz property \eqref{varthetacontrol} and the coercivity \eqref{coerciv2} allow to obtain
\begin{align*}
	\normmm{J_4}&\leq \frac14 \int_0^T\int_\om \normmm{V+\Div\xi}^2\d\mmu\dt+C\int_0^T\int_{{\supp \xi(t)}} \normmm{\vartheta}^2\d\mmu\dt\\&\quad+C\int_0^T\int_{{\om\setminus\supp \xi(t)}} \normmm{\vartheta}^2\d\mmu\dt\\&\leq \frac14 \int_0^T\int_\om \normmm{V+\Div\xi}^2\d\mmu\dt+C\int_0^T\Erel{t}\dt,
\end{align*}
where in the last inequality we have used also \eqref{necessaria} together with the construction of $\vartheta$ so that $\normmm{\vartheta(t)}\leq C$ on $\om\setminus \supp\xi(t)$ (cf. \eqref{varthetacontrol}).
Coming back to \eqref{estE}, we can now employ \eqref{diff}, \eqref{complete}, and the estimates of $J_i$, $i=1,\ldots,4$ to deduce 
\begin{align}
		\nonumber&\Ebulk{T}\\&\leq E_{\mathrm{bulk}}[\xhi_0|\mA(0)]+\frac14 \int_0^T\int_\om \normmm{V+\Div\xi}^2\d\mmu\dt\nonumber\\&\quad +C\int_0^T\Erel{t}\dt+C\int_0^T\Ebulk{t}\dt\nonumber
\\&\quad+\int_0^T\int_\om\xhi_\mA\bu\cdot\nabla\vartheta\dx\dt+\int_0^T\int_\om(\xhi_A-\xhi_\mA)\bv\cdot\nabla\vartheta\dx\dt\nonumber\\&\quad +\int_{0}^{T}\int_{\pA(t)}\vartheta \bu\cdot\bn \d\H^{d-1} \dt.
	\label{estF}
\end{align}
Now, again an application of the Gau\ss\  theorem leads to 
\begin{align*}
    \int_{0}^{T}\int_{\pA(t)}\vartheta \bu\cdot\bn \d\H^{d-1} \dt=-\int_0^T\int_\om\xhi_A\bu\cdot\nabla \vartheta\dx\dt,
\end{align*}
where we used $\Div \bu=0$. Therefore, we have, using \eqref{inequcmp} with $\zeta=\nabla\vartheta$ and $\delta=\frac 1{8C_K}$, with $C_K=2$ as the Korn's inequality squared constant in \eqref{korn}, 
\begin{align*}
    &\normmm{\int_0^T\int_\om\xhi_\mA\bu\cdot\nabla\vartheta\dx\dt+\int_0^T\int_\om(\xhi_A-\xhi_\mA)\bv\cdot\nabla\vartheta\dx\dt\nonumber +\int_{0}^{T}\int_{\pA(t)}\vartheta \bu\cdot\bn \d\H^{d-1} \dt}\\&
    =\normmm{\int_0^T\int_\om(\xhi_A-\xhi_\mA)(\bv-\bu)\cdot\nabla\vartheta\dx\dt} \\&
    \leq \frac1{8C_K}\int_0^T\int_\om \normmm{\nabla (\bv-\bu)}^2\dx\dt+C(1+\norm{\nabla\vartheta}_{L^\infty(\om\times[0,T])})\int_0^T\int_\om \normmm{\bv-\bu}^2\dx\dt\\&\quad \nonumber+C\norm{\nabla\vartheta}_{L^\infty(\om\times[0,T])}\int_0^T \Eb\dt\\&
    \leq \frac18\int_0^T\int_\om \normmm{D (\bv-\bu)}^2\dx\dt+C\int_0^T \Er\dt+C\int_0^T\Eb\dt.
\end{align*}
Therefore, plugging this estimate in \eqref{estF} and summing up with \eqref{estABCD} finally gives
\begin{align}
&\nonumber\Erel{T}+\Ebulk{T}+ \frac18\int_0^T\int_\om \normmm{D(\bv-\bu)}^2\dx\dt\\&\nonumber+ \frac{1}{2} \int_{0}^{T} \int_{\om} |V|^2 \d\normmm{\mu_t}_{\sf} \dt +\frac{1}{4} \int_{0}^{T} \int_{\om} |V+\Div \xi|^2 \d\mmu \dt\\&\leq E[\xhi_0,\bv_0|\xhi_{\mA(0)},\bu_0]+\Ebulk{0}\nonumber\\&\quad  +C\int_0^T\Erel{t}\dt
 +C\int_0^T\Eb\dt,	
	\label{estABCDE}
\end{align}
so that we can apply integral Gronwall's Lemma and conclude the proof of \eqref{eq:stabilityglobal}. Weak-strong uniqueness is then an immediate consequence of this result. Thus the proof of Theorem \ref{theo:weakStrongUniqueness} is concluded.

\medskip \ \\ \ \\
\textbf{Acknowledgments.} AP warmly thanks Tim Laux for useful comments on a preliminary version of the manuscript. AP also gratefully acknowledges support from the Alexander von Humboldt Foundation and is a member of Gruppo Nazionale per l’Analisi Matematica, la Probabilità e le loro Applicazioni (GNAMPA) of
Istituto Nazionale per l’Alta Matematica (INdAM). 

\medskip

\textbf{Data availability.} Since our paper only consists of mathematical proofs, no additional datasets have been
generated or analyzed. All relevant materials can be obtained from the references cited in our paper.
	
\medskip 

\textbf{Conflict of interest.} There is no conflict of interest.

\bibliographystyle{abbrvnat}
	\bibliography{Bib}
	
\end{document}